\documentclass[12pt]{article}

\input colordvi
\usepackage[active]{srcltx}
\usepackage{enumitem,amsmath,amssymb,amsthm,amscd,subfigure}
\usepackage{amsfonts,graphicx,psfrag}
\usepackage{color}

\DeclareMathOperator{\esssup}{ess\,sup}

\usepackage{amsmath,amssymb,amscd,subfigure}
\usepackage{enumitem,amsfonts,graphicx,psfrag}

\renewcommand{\Re}{\mathbb{R}}
\newcommand{\ds}{\displaystyle}

\newcommand{\half}{\frac{1}{2}}
\newcommand{\weak}{\rightharpoonup}
\newcommand{\embed}{\hookrightarrow}
\newcommand{\cembed}{\hookrightarrow \!\!\!\! \rightarrow}

\newcommand{\vph}{\vphantom{A^{A}_{A}}}
\newcommand{\sst}{\,\mid\,}

\newcommand{\dbyd}[2]{\frac{d #1}{d #2}}
\newcommand{\dbydp}[2]{\frac{\partial #1}{\partial #2}}

\newcommand{\eqnref}[1]{(\ref{eqn#1})}

\newcommand{\phat}{\hat{p}}

\newcommand{\shat}{\hat{s}}

\newcommand{\uhat}{\hat{u}}

\newcommand{\Ahat}{\hat{A}}

\newcommand{\Chat}{\hat{C}}

\newcommand{\Fhat}{\hat{F}}

\newcommand{\What}{\hat{W}}

\newcommand{\bbA}{\mathbb{A}}

\newcommand{\bbE}{\mathbb{E}}

\newcommand{\bbN}{\mathbb{N}}

\newcommand{\bbP}{\mathbb{P}}
\newcommand{\bbQ}{\mathbb{Q}}
\newcommand{\bbR}{\mathbb{R}}
\newcommand{\bbS}{\mathbb{S}}

\newcommand{\bbX}{\mathbb{X}}
\newcommand{\bbY}{\mathbb{Y}}

\newcommand{\bbEt}{\tilde{\mathbb{E}}}

\newcommand{\bbPt}{\tilde{\mathbb{P}}}

\newcommand{\bbXt}{\tilde{\mathbb{X}}}

\newcommand{\gbar}{{\bar{g}}}

\newcommand{\ubar}{{\bar{u}}}

\newcommand{\Fbar}{{\bar{F}}}

\newcommand{\Wbar}{{\bar{W}}}

\newcommand{\ft}{\tilde{f}}
\newcommand{\gt}{\tilde{g}}
\newcommand{\ut}{\tilde{u}}

\newcommand{\Ft}{\tilde{F}}

\newcommand{\Wt}{\tilde{W}}

\newcommand{\omegat}{\tilde{\omega}}

\newcommand{\Omegat}{\tilde{\Omega}}

\newcommand{\bfb}{{\bf b}}

\newcommand{\bfu}{{\bf u}}
\newcommand{\bfv}{{\bf v}}

\newcommand{\calA}{{\cal A}}
\newcommand{\calB}{{\cal B}}

\newcommand{\calF}{{\cal F}}
\newcommand{\calG}{{\cal G}}
\newcommand{\calH}{{\cal H}}

\newcommand{\calL}{{\cal L}}
\newcommand{\calM}{{\cal M}}
\newcommand{\calN}{{\cal N}}

\newcommand{\calP}{{\cal P}}

\newcommand{\calU}{{\cal U}}

\newcommand{\calX}{{\cal X}}

\newcommand{\calFt}{\tilde{\cal F}}

\newcommand{\norm}[1]{\| {#1} \|}
\newcommand{\Hone}{{H^1(D)}}

\newcommand{\Honeo}{{H^1_0(D)}}

\newcommand{\Lone}{{L^1(D)}}
\newcommand{\lone}[1]{\norm{#1}_\Lone}
\newcommand{\Ltwo}{{L^2(D)}}
\newcommand{\ltwo}[1]{\norm{#1}_\Ltwo}
\newcommand{\Ltwoo}{{L^2(D)/\Re}}
\newcommand{\ltwoo}[1]{\norm{#1}_\Ltwoo}

\newcommand{\Woneqo}{{W^{1,q}_0(D)}}

\newcommand{\Lq}{{L^q(D)}}

\newcommand{\Lfour}{{L^4(D)}}
\newcommand{\lfour}[1]{\norm{#1}_\Lfour}
\newcommand{\Lsix}{{L^6(D)}}
\newcommand{\lsix}[1]{\norm{#1}_\Lsix}

\newcommand{\LtwoLtwo}{{L^2[0,T;\Ltwo]}}

\newtheorem{theorem}{Theorem}[section]
\newtheorem{lemma}[theorem]{Lemma}
\newtheorem{corollary}[theorem]{Corollary}
\newtheorem{proposition}[theorem]{Proposition}

\newtheorem{definition}[theorem]{Definition}
\newtheorem{assumption}[theorem]{Assumption}
\newtheorem{remark}[theorem]{Remark}

\theoremstyle{definition}
\newtheorem{example}[theorem]{Example}

\newenvironment{theorem*}{{\bf Theorem}\em}{\rm\mbox{}}
\newenvironment{lemma*}{{\bf Lemma}\em}{\rm\mbox{}}
\newenvironment{corollary*}{{\bf Corollary}\em}{\rm\mbox{}}
\newenvironment{proposition*}{{\bf Proposition}\em}{\rm\mbox{}}
\newenvironment{claim}{{\it Claim:}\em}{\rm\mbox{}}

\def\keywords{\par{\bf Keywords:}\ \small \ignorespaces}

\newcommand{\bbEhat}{\hat{\mathbb{E}}}
\newcommand{\bbPhat}{\hat{\mathbb{P}}}
\newcommand{\bbXhat}{\hat{\mathbb{X}}}

\renewcommand{\Hone}{{H^1(D)}}

\renewcommand{\Honeo}{{H^1_0(D)}}

\renewcommand{\Lone}{{L^1(D)}}
\renewcommand{\lone}[1]{\norm{#1}_\Lone}
\renewcommand{\Ltwo}{{L^2(D)}}
\renewcommand{\ltwo}[1]{\norm{#1}_\Ltwo}
\renewcommand{\Ltwoo}{{L^2(D)/\Re}}
\renewcommand{\ltwoo}[1]{\norm{#1}_\Ltwoo}

\newcommand{\normH}[1]{\norm{#1}_H}
\newcommand{\normU}[1]{\norm{#1}_U}
\newcommand{\normUp}[1]{\norm{#1}_{U'}}

\newcommand{\LtwoH}{{L^2[0,T;H]}}
\newcommand{\LtwoU}{{L^2[0,T;U]}}
\newcommand{\ltwoH}[1]{\norm{#1}_\LtwoH}
\newcommand{\ltwoU}[1]{\norm{#1}_\LtwoU}

\newcommand{\LpH}{{L^p[0,T;H]}}

\newcommand{\LqU}{{L^q[0,T;U]}}
\newcommand{\lqU}[1]{\norm{#1}_\LqU}
\newcommand{\LqpUp}{{L^{q'}[0,T;U']}}
\newcommand{\lqpUp}[1]{\norm{#1}_\LqpUp}

\newcommand{\LrU}{{L^r[0,T;U]}}
\newcommand{\lrU}[1]{\norm{#1}_\LrU}

\newcommand{\LinfH}{{L^\infty[0,T;H]}}

\newcommand{\linfH}[1]{\norm{#1}_\LinfH}

\newcommand{\LtwoUp}{{L^2[0,T;U']}}
\newcommand{\ltwoUp}[1]{\norm{#1}_\LtwoUp}
\newcommand{\LinfUp}{{L^\infty[0,T;U']}}
\newcommand{\linfUp}[1]{\norm{#1}_\LinfUp}

\begin{document}
\title{Numerical Approximation of Nonlinear SPDE's}

\author{Martin Ondrej\'{a}t\thanks{The Czech Academy of Sciences,
    Institute of Information Theory and Automation,
    Pod Vod\'arenskou v\v e\v z\'{\i} 4, 
    182 00 Prague 8,
    Czech Republic, 
    Supported by the Czech Science Foundation grant no. 19-07140S
  }
\and Andreas Prohl\thanks{Mathematisches Institut der Universit\"at T\"ubingen,
  Auf der Morgenstelle 10,
  D-72076 T\"ubingen, Germany. 
}
\and Noel J. Walkington\thanks{Department of Mathematics, Carnegie Mellon
    University, Pittsburgh, PA 15213, USA. Supported in part by National
    Science Foundation Grant DREF-1729478.  This
    work was also supported by the NSF through the Center for
    Nonlinear Analysis. 
  }
}

\date{\today}

\maketitle

\begin{abstract} 
  The numerical analysis of stochastic parabolic partial
  differential equations of the form
  $$
  du + A(u) = f \,dt + g \, dW,
  $$
  is surveyed, where $A$ is a partial operator and $W$ a Brownian
  motion. This manuscript unifies much of the theory developed over
  the last decade into a cohesive framework which integrates
  techniques for the approximation of deterministic partial
  differential equations with methods for the approximation of
  stochastic ordinary differential equations. The manuscript is
  intended to be accessible to audiences versed in either of these
  disciplines, and examples are presented to illustrate the
  applicability of the theory. 
\end{abstract}

\keywords{SPDE's, weak martingale solution, fully discrete scheme,
  numerical analysis}

\section{Introduction}\label{s-intro}
We consider the numerical approximation of solutions
of stochastic partial differential equations (SPDE's) of the form
\begin{equation} \label{eqn:spde}
du + A(u) \, dt = f \, dt + g \, dW,
\qquad u(0) = u^0.
\end{equation}
The solution $u := \{u(t)\, \vert \, t \in [0,T]\}$ is a stochastic
process taking values in a Banach space $U$. The function $A:U
\rightarrow U'$, processes $f$, $g$, and the random variable $u^0$ are
specified, and $W := \{W_t\, \vert \, t \geq 0 \}$ is a Wiener process
on a filtered probability space $(\Omega, \calF, \{ \calF(t)\}_{0 \leq
  t \leq T},\bbP)$.

The existence theory for \eqnref{:spde} was first developed for
linear spatial operators and then extended in various directions. 
The analysis of numerical schemes to approximate solutions of
\eqnref{:spde} has paralleled this development within the last decade.
\begin{enumerate} 
\item[(i)] The stochastic linear heat equation: $A(u) =
   -\Delta u$; \cite{GY1,Y1}.

\item[(ii)] Problems with Lipschitz nonlinearities: $A(u) =
  -\Delta u + F(u)$; \cite{GM1,D1,LPS1,K1}.

\item[(iii)] Semi-linear equations which involve locally
  Lipschitz nonlinearities:
  \begin{enumerate} 
    \item The Allen-Cahn equation: 
      $A(u) = -\Delta u + (\vert u\vert^2 - 1) u$; \cite{KLM1,SS1,MP1},

    \item The nonlinear Schr\"odinger equation: 
      $A(u) = -i (\Delta u + \vert u\vert^2 u)$; \cite{DB1}.

    \item The incompressible Navier-Stokes equation 
      $A(u) = -\Delta u + (u \cdot \nabla)u$; \cite{Pri1,BCP1,CP1,GTW2017}.

    \item The Landau-Lifshitz equation: $A(u) = u
      \times (u \times \Delta u) - u \times \Delta u$; \cite{BBNP1}.
    \end{enumerate}
  
  \item[(iv)] Very few results are available for the numerical
    approximation of stochastic versions of degenerate parabolic equations,
    such as the stochastic porous-medium equation \cite{GG1}.
\end{enumerate}
For the first two cases, semigroup techniques are often used to
construct mild solutions of \eqnref{:spde}; a comprehensive exposition
of this theory may be found in the monograph \cite{DPZ1}.  Variational
approaches were developed in \cite{KR1,PrRo07} to accommodate nonlinear
equations where the concept of a mild solution is not available. The
more general notion of a ``weak martingale solution'' is required
to obtain the existence of solutions for the last two equations in (iii), and
(iv).

The collective effort of this work is a unification of techniques from
stochastic analysis  and numerical analysis of PDE's, resulting in a
general convergence theory for implementable discretizations of a wide class
of nonlinear SPDE's. This theory
provides the technical tools needed to realize the Lax Richtmeyer
meta--theorem:
\begin{quote}
  \em A numerical scheme converges if (and only if) it is stable and
  consistent.
\end{quote}
For this purpose, we distill and adapt ideas from
\cite{BBNP1,BCP1,HS1,HS2,OS1} to develop a general convergence theory
for numerical schemes comprising of the following steps:
\begin{enumerate}
\item {\bf Estimates:} Structural properties of the particular SPDE
  inherited by the discrete schemes are used to bound the numerical
  approximations uniformly with respect to discretization parameters. While
  the specific structure and bounds are problem dependent,
  standard tools from stochastic analysis (independence, filtrations,
  adaptedness) are utilized to accommodate the stochastic term.

\item {\bf Compactness:} Compactness properties of Banach spaces are
  used in an essential fashion when the operator $A$ is nonlinear. For
  deterministic PDE's ($g \equiv 0$ in \eqnref{:spde}) the
  Banach-Alaoglu and Lions-Aubin theorems are used to identify
  limits of approximate solutions.  In the stochastic setting the
  solutions are random variables taking values in Banach spaces and
  the deterministic arguments are augmented with the Prokhorov theorem
  to obtain convergence of laws.

\item[3.] {\bf Convergence:} Concepts of weak and strong solutions are
  used in both, the deterministic and stochastic setting to specify in
  what sense a function $u$ is a solution of the equation. While the
  meaning of a weak solution is very different in each setting, it has
  the same purpose; it extends the concept of a solution to
  accommodate situations where strong (or classical) solutions may not
  exist. In this work, the concepts of a weak solution for the
  deterministic and stochastic setting are combined to construct
  weak martingale solutions as a limit of solutions to discrete
  approximations of \eqnref{:spde}.
\end{enumerate}

Bounds upon the numerical approximations establish stability of the
numerical schemes. For consistent (Galerkin) approximations of the
parabolic problems under consideration we show
$$
\text{stability} \quad \Rightarrow \quad
\text{stability \& compactness} \quad \Rightarrow \quad
\text{convergence},
$$
so that the Lax Richtmeyer theorem is realized.  The goal of this
article is to present these ideas in a context accessible to audiences
from either numerical PDE's or stochastic analysis. To achieve this, key
results required from each area will be stated and their role
explained prior to their use. 

To reduce the technical overhead, we first consider the situation
where $A:U \rightarrow U'$ is linear, and the Wiener process is
scalar-valued. Extensions to include nonlinear drift operators $A$ and
spatial noise will be considered in subsequent sections. These
extensions are mainly technical in the sense that once the additional
definitions, concepts, and properties are acquired, it becomes clear
that the ideas and proofs in the simplified setting extend directly to
the more general situation.

We finish this section with a terse review of the essential concepts
from numerical PDE's and stochastic processes required for the development
of weak martingale solutions to equation \eqnref{:spde}. 

\subsection{Numerical Partial Differential Equations}
This section reviews the abstract setting where tools from
functional analysis can be applied to solve PDE's. Solutions are sought in a Banach space $U$, and a pivot
space construction is used to characterize the partial differential
operator under consideration. Specifically, $U$ is assumed to be
densely embedded in a Hilbert space $H$, and when $H$ is identified
with its dual by the Riesz theorem we have $U \embed H \embed U'$.
Then $u \in U$ is identified with the dual element $\iota(u) \in
U'$ by
$$
\iota(u)(v) = (u,v)_H, \qquad v \in U.
$$
If $f \in U'$ we frequently write $(f,v)=f(v)$ so that $(f,v)=(f,v)_H$
when $f \in H$.

Solutions of time dependent problems are viewed as (strongly
measurable) functions from the interval $[0,T]$ to various Banach
spaces. The Bochner spaces are the natural Banach spaces that arise
in this context; for example,
\begin{align*}
\LtwoU &= \{u:[0,T] \rightarrow U \sst \int_0^T \normU{u(t)}^2 \, dt < \infty \}, \\
\LinfH &= \{u:[0,T] \rightarrow H \sst 
\esssup_{0 \leq t \leq T} \normH{u(t)} < \infty \}.
\end{align*}
Similar notation is used for the continuous functions, $C[0,T;U]$,
and H\"older continuous functions, $C^{0,\theta}[0,T;U]$, from $[0,T]$ to
a Banach space $U$.

The space $U$ is constructed so that the partial differential
operator, $A$, in equation \eqnref{:spde} maps $U$ to $U'$. In 
this situation it is possible to define $a:U \times U \rightarrow \Re$
by 
$$
a(u,v) = A(u)(v) , \qquad u, v \in U.
$$
The canonical example of this construction is the Laplacian; $A(u) = -\Delta
u$ on a bounded domain $D \subset \Re^d$ with homogeneous boundary
data.  Letting $H = \Ltwo$ and $U$ be the Sobolev space
$$
U = \Honeo \equiv \{u \in \Ltwo \sst \nabla u \in \Ltwo^d, 
\,\, u|_{\partial \Omega} = 0 \}, 
$$
then
$$
A(u)(v)
= (-\Delta u, v)
\equiv \int_D \nabla u . \nabla v\, dx
= a(u,v),
\qquad u,v \in U.
$$
In this setting, a weak solution of the (stationary) PDE $Au = f$ satisfies
\begin{equation}\label{f_three}
u \in U \qquad a(u,v) = f(v), \qquad v \in U.
\end{equation}
The solution of this second order PDE is ``weak'' in the sense that it is
only required to have one square integrable derivative and the datum $f
\in U'$ need not be regular.  For linear problems the following
theorem establishes existence of weak solutions in many situations.

\begin{theorem}[Lax Milgram 1954] Let $U$ be a Hilbert space and
$a:U \times U \rightarrow \Re$ be bilinear and suppose that
there exist constants $C_a$, $c_a > 0$ such that
$$
|a(u,v)| \leq C_a \normU{u} \normU{v}, 
\quad \text{ and } \quad
a(u,u) \geq c_a \normU{u}^2, \qquad u, v \in U.
$$
Then for each $f \in U'$ there exists a unique $u \in U$ such that
$$
a(u,v) = f(v), \qquad v \in U.
$$
Moreover, $\normU{u} \leq \norm{f}_{U'}/c_a$.
\end{theorem}

Given $f:(0,T) \rightarrow U'$ and $u^0 \in H$,
a weak solution of the evolution equation $\partial_t u + Au = f$ on
$(0,T)$ with $u(0) = u^0$ is a function $u:[0,T] \rightarrow U$
satisfying
\begin{equation} \label{eqn:weakHeat}
(u(t), v)_H + \int_0^t a(u,v)\, ds 
= (u^0,v)_H + \int_0^t (f,v)\, ds, 
\qquad v \in U, 
\,\, t \in [0,T].
\end{equation}
The pivot space construction is used to characterize
$\partial_t u(t) \in U'$; for almost every $t \in (0,T)$
$$
(\partial_t u(t), v) = 
\lim_{h \rightarrow 0} \frac{(u(t+h) - u(t), v)_H}{h},
\qquad v \in U.
$$

If $U_h \subset U$ is a finite dimensional subspace, a natural
numerical scheme to approximate weak solutions of the stationary
problem $Au = f$ is obtained by seeking a function $u_h \in U_h$ which
satisfies the weak statement (\ref{f_three}) for each ``test
function'' $v_h \in U_h$.  To obtain a fully discrete scheme for the
evolution equation $\partial_t u + Au = f$ it is necessary to also
approximate the time derivative. If $N \in \bbN$ and $\tau = T/N$ is a
time step, the implicit Euler scheme computes approximations
$\{u^n_{h\tau}\}_{n=1}^N \subset U_h$ of $\{u(t^n)\}_{n=1}^N$ on a
uniform partition $\{t^n\}^N_{n=0}$ of $[0,T]$ as solutions of
\begin{equation}\label{euler-intro}
(u^n_{h\tau} - u^{n-1}_{h\tau}, v_h)_H
+ \tau a(u^n_{h\tau}, v_h) = \tau (f^n_{h\tau},v_h), 
\qquad v_h \in U_h, \quad n = 1,2, \ldots N,
\end{equation}
with $u^0_{h\tau}$, and $f^n_{h\tau}$ approximations of $u^0$ and
$f(t^n)$.  The finite element methodology \cite{BrSc08} provides a systematic method
to construct finite dimensional subspaces of the function space $U$.
These subspaces consist of piecewise polynomial functions on a
partition of the domain $D \subset {\mathbb R}^d$; the index $h > 0$
denotes the maximal diameter of a partition (the mesh size). If $a:U
\times U \rightarrow \Re$ satisfies the hypotheses of the Lax Milgram
theorem, then so too does
$$
a_\tau(u_h, v_h) \equiv (u_h,v_h)_H + \tau a(u_h, v_h),
\qquad u_h, v_h \in U_h,
$$
which ensures the existence of a unique solution to the implicit Euler scheme
(\ref{euler-intro}).

Compactness properties of various Banach spaces are required to
obtain and identify limits of numerical solutions. For the parabolic
problems under consideration the space $U$ will always be compactly
embedded into the pivot space $H$; we write $U \cembed H$. For the
evolution problem with $U \cembed H \cembed U'$ a typical compactness
result for the associated Bochner spaces is the following 
\cite[Theorem 5]{Si87}.

\begin{theorem} \label{thm:cpt1}
  Let $U \cembed B \embed U'$ be embeddings of Banach
  spaces and $1 \leq p \leq \infty$. Then
  $L^p[0,T; U] \cap C^{0,\theta}[0,T;U'] \cembed L^p[0,T;B]$
  (and in $C[0,T;B]$ if $p=\infty$).
\end{theorem}

For the pivot space, we use $\normH{u}^2 \leq \normU{u} \normUp{u}$ and
\cite[Theorem 7]{Si87},
$$
C^{0,\theta}[0,T;U'] \cap L^1[0,T;U] \cembed L^2[0,T; H].
$$

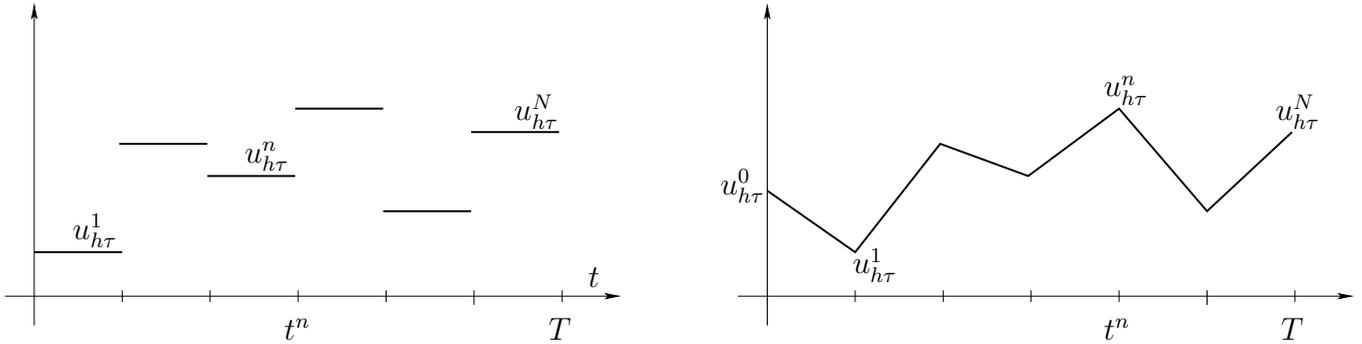
\begin{figure}
\begin{center}
\includegraphics[height=1.7in]{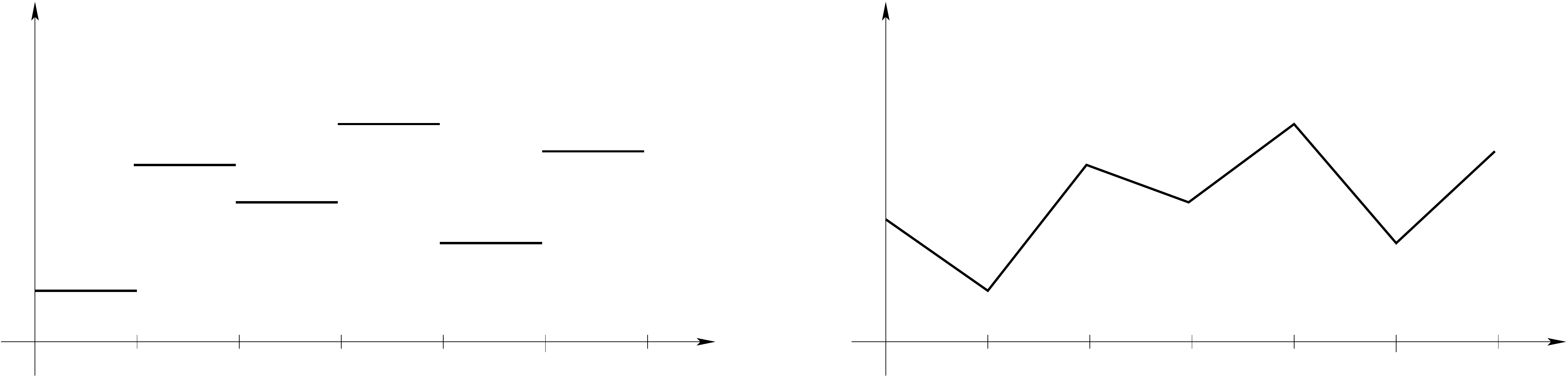}
\put(-405,-5){$t^n$}
\put(-290,15){$t$}
\put(-305,-5){$T$}
\put(-485, 33){$u_{h\tau}^1$}
\put(-420, 62){$u_{h\tau}^n$}
\put(-318, 78){$u_{h\tau}^N$}
\put(-28,-5){$T$}
\put(-240, 50){$u_{h\tau}^0$}
\put(-190, 20){$u_{h\tau}^1$}
\put(-95,-5){$t^n$}
\put(-95, 87){$u_{h\tau}^n$}
\put(-30, 78){$u_{h\tau}^N$}
\caption{Piecewise constant $u_{h\tau}$ and piecewise
affine interpolation $\uhat_{h\tau}$ of $\{u^n_{h\tau}\}_{n=0}^N$.}
\label{fig:uhatu}
\end{center}
\end{figure}

\subsubsection{Skorokhod Space}
The implicit Euler scheme (\ref{euler-intro}) gives a sequence
$\{u^n\}_{n=0}^N$ which can be interpolated to give either a
piecewise affine function $\uhat_{h\tau}$ or a piecewise constant
function $u_{h\tau}$ (see Figure \ref{fig:uhatu}), which satisfy the
equation
$$
\dbyd{\uhat_{h\tau}}{t} + A(u_{h\tau}) = f_{h\tau}, 
\qquad \text{ in } U'.
$$
Typically bounds are obtained by multiplying this equation by
$u_{h\tau}$, so discontinuous trial and test functions should be
admissible; however, it is desirable to retain some of the continuity
properties of $\uhat_{h\tau}$.  For this reason it is convenient to
pose the problem in the Skorokhod--type space\footnote{Functions in
  the Skorokhod space $D[0,T;U']$ are continuous from the right
  (\`a droite). For parabolic problems the initial data is less regular
  than the solution at later times, so it is natural to consider
  functions $G[0,T;U']$ continuous from the left (\`a gauche).}  (see
Figure \ref{fig:caglad}),
$$
G[0,T; U'] = \big\{u:[0,T] \rightarrow U' \sst 
u(t) = \lim_{s \rightarrow t^-} u(s) \text{ and }
\lim_{s \rightarrow t^+} u(s) \text{ exists} \big\},
\quad \text{(i.e. caglad functions)}.
$$
Developing a general theory in this context is extremely useful for
applications since stochastic solutions are not very
regular in time. Consistency errors of the form $A(u_{h\tau}) -
A(\uhat_{h\tau})$ would arise if test functions were required to
be continuous in time, and frequently these may not vanish as 
$(h,\tau) \rightarrow (0,0)$.

\begin{figure}
\begin{center}
\includegraphics[height=1.7in]{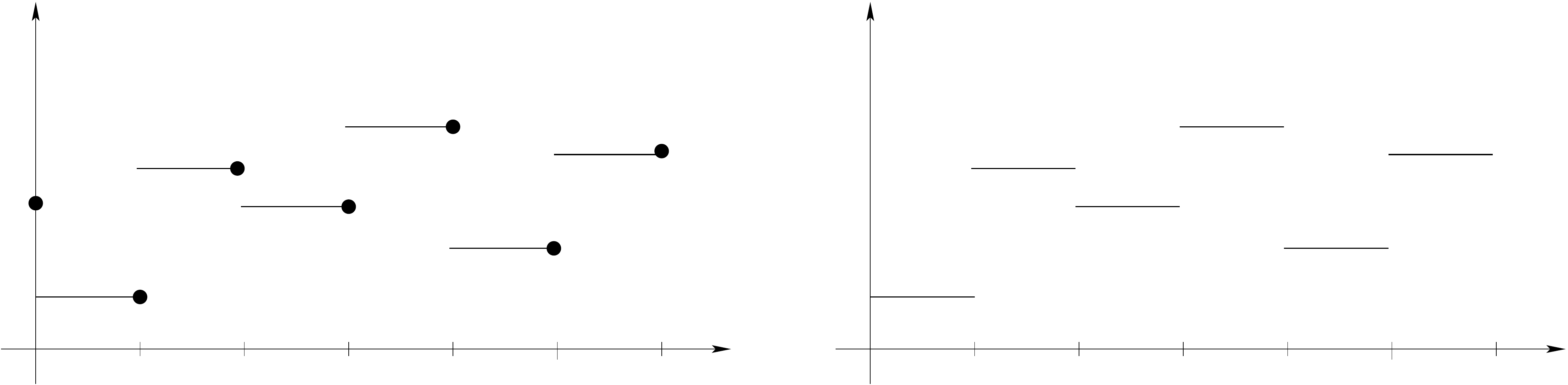}
\put(-395,-5){$t^m$}
\put(-410,15){$s$}
\put(-327,-5){$t^n$}
\put(-340,15){$t$}
\put(-293,-5){$T$}
\put(-480, 33){$u_{h\tau}^1$}
\put(-413, 62){$u_{h\tau}^m$}
\put(-348, 50){$u_{h\tau}^n$}
\put(-315, 79){$u_{h\tau}^N$}
\put(-510, 57){$u_{h\tau}^0$}
\put(-130,-5){$t^{n-1}$}
\put(-90,-5){$t^n$}
\put(-28,-5){$T$}
\put(-210, 33){$f_{h\tau}^1$}
\put(-110, 87){$f_{h\tau}^n$}
\put(-43, 79){$f_{h\tau}^N$}
\caption{Indexing of piecewise constant caglad functions $u_{h\tau}$
  and  Bochner functions $f_{h\tau}$.}\label{fig:caglad}
\end{center}
\end{figure}
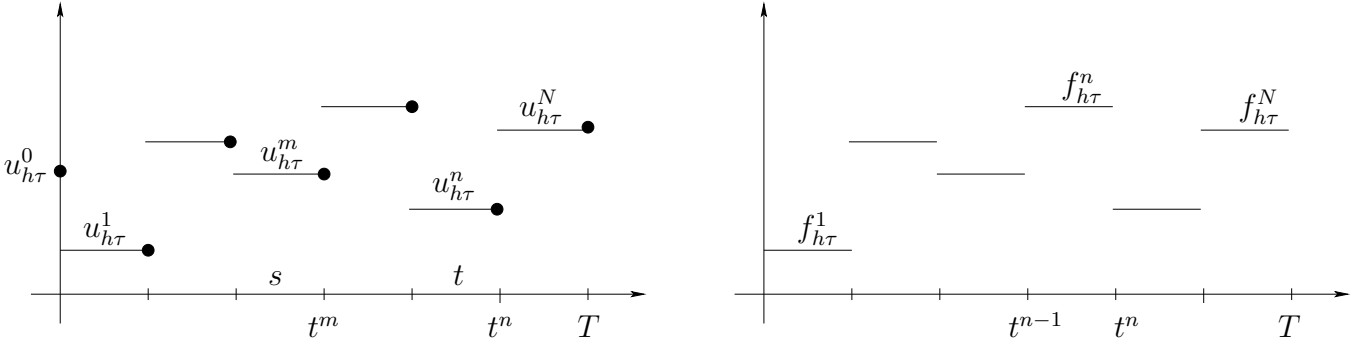

The construction of the Skorokhod metric is technical, and for
completeness we present it here; however, the explicit formula will
not be needed. Let $\Lambda$ be the set of strictly
increasing functions $\lambda:[0,T] \rightarrow [0,T]$ satisfying
$\lambda(0)=0$ and $\lambda(T) = T$, and set
$$
\gamma(\lambda) = \sup_{0 \leq s < t \leq T} \left|
\ln \left(\frac{\lambda(t)-\lambda(s)}{t-s} \right) \right| .
$$
The Skorokhod metric is
$$
d_G(u,v) = \inf_{\lambda \in \Lambda} 
\max\left(\gamma(\lambda), \linfUp{u - v \circ \lambda} \right).
$$
The following lemma contains the properties of $G[0,T; U']$ required
in the sequel.

\begin{lemma} \label{lem:skorohod} 
  Let $U$ be a Banach space and $G[0,T; U']$ denote the caglad
  functions on $[0,T]$ taking values in $U'$ endowed with the Skorokhod
  metric, $d_G(.,.)$.
  \begin{enumerate}
  \item $G[0,T; U']$ is complete and is separable when $U'$ is separable.

  \item The following embeddings are continuous,
    $$
    C[0,T; U'] \embed G[0,T; U'] \embed L^s[0,T;U'], 
    \qquad 1 \leq s < \infty,
    $$
    and $\linfUp{u} \leq d_G(0,u)$ so $G[0,T;U'] \subset
    \LinfUp$. However, the inclusion is not an embedding since convergence
    in $G[0,T; U']$ does not imply uniform convergence.  

  \item If $d_G(u,u_n) \rightarrow 0$, then $u_n(t) \rightarrow u(t)$
    for $t = 0$, $t=T$, and at every time $t \in (0,T)$ where $u$ is
    continuous.  In particular,
    \begin{itemize}
    \item $u_n(t) \rightarrow u(t)$ for almost every $t \in [0,T]$
      since there is at most a countable set of times $t \in [0,T]$ at
      which a function in $G[0,T;U']$ is discontinuous.

    \item If the limit $u$ is continuous, then $u_n(t) \rightarrow u(t)$
      for every $t \in [0,T]$.
    \end{itemize}
    
  \item If $0 = t^0 < t^1 < \ldots < t^N = T$, then the linear function
    $\phi:C^{0,\theta}[0,T;U'] \rightarrow G[0,T;U']$ for which
    $\phi(u)$ is the piecewise constant caglad interpolant of
    $\{u(t_i)\}_{i=0}^N$ is continuous, and 
    $$
    d_G(\phi(u), \phi(v)) \leq \norm{u-v}_{C[0,T;U']}
    \quad \text{ and } \quad
    d_G(\phi(u), u) 
    \leq \norm{u}_{C^{0,\theta}[0,T; U']}
    \left(\max_{1 \leq n \leq N} (t^n-t^{n-1}) \right)^\theta.
    $$
  \end{enumerate}
\end{lemma}

\subsection{Stochastic Processes}\label{sp1}
All of the random variables we consider will be measurable mappings
from a probability space $(\Omega,\calF, \bbP)$ to a topological space
$\bbX$ equipped with the Borel $\sigma$-algebra $\calB(\bbX)$, and we
adopt the terminology that a {\em (stochastic) process} is a function
from a time interval $[0,T]$ to a set of random variables.  Implicit
in the statement of equation \eqnref{:spde} is the presence of a
filtration $\{ \calF(t)\}_{0 \leq t \leq T}$ on $(\Omega, \calF, \bbP)$.  In
order to apply standard results from probability all filtrations are
assumed to satisfy the ``usual conditions'' \cite{KaSh98}, namely,
\begin{enumerate}
\item $\calF(0)$ contains all the null sets.
\item $\calF(t) = \cap_{s>t} \calF(s)$.
\end{enumerate}
An analogous terminology is utilized for discrete filtrations,
$\calF^0 \subset \calF^1 \subset \ldots \subset \calF^N$.

The probability of a measurable set $B \in \calF$ is denoted by
$\bbP[B]$, and the expected value of a random variable $X$ by
$\bbE[X]$.  The conditional expectation of a random variable $u$ with
respect to a sub-$\sigma$-algebra $\calG$ of $\calF$ is denoted by
$\bbE[u|\calG]$. 

In order to exploit arguments from both functional and stochastic
analysis it is convenient to view a solution of equation
\eqnref{:spde} as both a random variable with values in a Bochner
space (for example, $u \in L^2(\Omega, L^\infty[0,T;H])$), and as a
stochastic process (for example, $u \in L^2[0,T; L^p(\Omega,U')]$).
While both may be viewed as Bochner spaces, a much richer theory is
available for the subspace of stochastic processes {\em adapted} to a
filtration; that is, when $u(t)$ is $\calF(t)$-measurable for all $0 \leq t \leq T$.  

To construct the stochastic integral of a random variable with values
in equivalence classes of functions, such as $G:\Omega\to L^2[0,T;H]$,
a jointly measurable adapted representation $g:[0,T]\times\Omega\to H$
is required for which $g(\cdot, \omega)\in G(\omega)$ for every
$\omega\in\Omega$. Specifically, the stochastic integral is correctly
defined only for jointly measurable adapted processes with paths in
$L^2[0,T;H]$ almost surely. Such $g$ exists if and only if $\Omega\to
L^2[0,T;H]:\omega\mapsto\mathbf 1_{[0,t]}G(\omega)$ is
$\calF(t)$-measurable for every $t\in[0,T]$ in which case an appropriate
selection is the ``precise representative'' \cite{EvGa92} given by
$$
g(t,\omega)=\lim_{n\to\infty}n\int_{(t-1/n)_+}^tG(\omega)\,ds,
\qquad \text{ if the limit exists,}
$$
and $g(t,\omega)=0$ otherwise.  This representative is actually
predictable; that is, measurable with respect to the $\sigma$-algebra
generated by left continuous adapted processes. When identifying a
random variable taking values in a Bochner space with a process we will
tacitly assume that a jointly measurable element of the equivalence
class is taken so that the stochastic calculus is available.

\subsubsection{Martingales}
An important class of adapted processes is the martingales. Given a
filtration $\{ \calF(t)\}_{t \geq 0}$ on a probability space, an
adapted process $\{u(t)\}_{t \geq 0}$ with values in a Banach space $U$ is
an $\{ \calF(t)\}_{t \geq 0}$--martingale if at each time it is
integrable, $\bbE[\norm{u(t)}_U] < \infty$, and if it has conditionally
independent increments, $\bbE[u(t)-u(s) | \calF(s)] = 0$ when $s \leq
t$.  In particular, $\bbE[u(t)|\calF(s)] = \bbE[u(s)|\calF(s)] =
u(s)$; the second equality following since $u$ is adapted.  For $T >
0$ and $H$ a Hilbert space, we denote by ${\mathcal M}_T^2(H)$ the set
of all $H$-valued, square integrable martingales with continuous
paths.  If indistinguishable processes are considered as one process,
this is a Banach space when endowed with the norm
$$ 
\Vert X \Vert_{{\mathcal M}^2_T(H)} 
= \norm{X}_{L^2(\Omega, L^\infty[0,T;H])}
\equiv \Bigl( {\mathbb E}\bigl[
\sup_{t \in [0,T]} \Vert X(t) \Vert_{H}^2\bigr]\Bigr)^{1/2} .
$$ 
Note that the time at which the supremum is taken depends upon $\omega
\in \Omega$ and two processes $X$ and $Y$ are {\em indistinguishable}
if there exists a set $A \subset \Omega$ with $\bbP[A] = 1$ for which
$X(t,\omega) = Y(t,\omega)$ for all $\omega \in A$ and $t \in [0,T]$.

The quadratic variation $\langle X \rangle$ of a process $X \in {\mathcal M}_T^2(H)$,
defined next, plays a central role in the subsequent theory.

\begin{definition}\label{simn} 
  Let $T > 0$, $H$ be a separable Hilbert space, and $( \Omega, \calF,
  \{ \calF(t)\}_{t \geq 0}, \bbP)$ be a filtered probability space.
  The quadratic variation $\langle X \rangle$ of $X \in {\mathcal M}_T^2(H)$ is a
  symmetric, non--negative bilinear process $\langle X(t)
  \rangle : H \times H \rightarrow \Re$ satisfying:
  \begin{enumerate}
  \item (Adaptedness) For each $t \in [0,T]$ the real-valued
    random variable $\langle X(t) \rangle(u,v)$ is $\calF(t)$-measurable.

  \item (Continuity) $t \mapsto \langle X(t) \rangle(u,v)$ is
    continuous for every $\omega \in \Omega$ and $u,v\in H$.
    
  \item (Normalization) $\langle X(0) \rangle = 0$.

  \item (Monotonicity) $\langle X(t) \rangle (u,u) \geq \langle X(s)
    \rangle (u,u)$ for every $u \in H$ and $0 \leq s \leq t \leq T$.
    
  \item (Variation) The function $t \mapsto (X(t), u)_H (X(t), v)_H -
    \langle X(t) \rangle (u,v)$ is a continuous real--valued martingale
    for each pair $u,v \in H$.
  \end{enumerate}
\end{definition} 

Note that for each time $\langle X(t) \rangle$ is a semi--inner
product so is characterized by $\{\langle X(t) \rangle(u,u) \sst u \in
H\}$, or by the Riesz maps $L(t):H \rightarrow H$ for which
$(L(t)(u),v)_H = \langle X(t) \rangle (u,v)$. A ``standard Wiener
process'' (or Brownian motion) is a real-valued martingale $W \in
\calM^2_T(\Re)$ satisfying $W(0) = 0$, with $\bbE[W(t)] = 0$, and
quadratic variation $\langle W(t) \rangle = t$ for all $0 \leq t \leq T$.

The quadratic variation process appears in the isometry for Ito
integrals, and the statement of the Burkholder--Davis--Gundy (BDG)
inequalities.  The construction of the Ito integral, and a proof of the
BDG inequalities involve significant technical developments; however,
numerical schemes considered here involve processes taking values at
discrete times which eliminates much of the technical overhead.  Let
$\{t^n\}_{n=0}^N$ be a uniform partition of $[0,T]$ with time step
$\tau = T/N$, and $\{ {\mathcal F}^n\}_{n=0}^N$ be a (discrete)
filtration of $(\Omega, {\mathcal F}, {\mathbb P})$.  In this context
discrete Ito integrals take the form
\begin{equation} \label{eqn:discreteIto}
X^n_{\tau} = \sum_{m=1}^n g^{m-1}_{\tau} \xi^m_{\tau}, 
\qquad n=1,2, \ldots,N
\quad \text{ and } \quad X^0_{\tau} \equiv 0,
\end{equation}
where $g^{m-1}_\tau$ is an $\calF^{m-1}$--measurable random variable
with values in a Hilbert space $H$, and for each $m = 1,2, \ldots, N$
the increments $\{\xi^m_{\tau}\}_{n=1}^N$ are real--valued
random variables which satisfy the following standing assumptions. 

\begin{assumption} \label{ass:spde0} (with parameter $p \geq
  2$)  For each $N \in \bbN$ let $\{t^n\}_{n=0}^N$ be the uniform
  partition of $[0,T]$ with time step $\tau = T/N$. Then $\bigl(
  \Omega, \calF, \{\calF^n\}_{n=0}^N, \bbP \bigr)$ is a (discretely)
  filtered probability space and the real-valued random variables $\{
  \xi^n_\tau\}_{n=1}^N$ satisfy
  \begin{enumerate}
  \item (Zero average) ${\mathbb E}[\xi^n_\tau] = 0$.

  \item (Variance) ${\mathbb E}[|\xi^n_{\tau}|^2] = \tau \equiv T/N$.

  \item (Bounds) $\xi^n_{\tau} \in L^{p}(\Omega)$, and there exists a
    constant $C > 0$ such that ${\mathbb E}[|\xi^n_{\tau}|^{p}] \leq C
    \tau^{p/2}$.

  \item (Independence) $\xi^n_\tau$ is $\calF^n$-measurable and
    independent of $\{\calF^m\, \vert \, 0 \leq m \leq n-1\}$.
  \end{enumerate}
\end{assumption}

Increments of the form $\xi^n_\tau = W(t^n) - W(t^{n-1})$ with $W$ a
standard Wiener process on a filtration of $\bigl(
\Omega, \calF, \bbP \bigr)$ satisfy the above assumptions but lack
practical realization. In a numerical context, discrete random
variables $\{\xi^n_{\tau}\}_{n=1}^N$ taking values $\pm \sqrt{\tau}$
with the same probability of $1/2$, and $\calF^n$ the $\sigma$-algebra
generated by $\{\xi^m_\tau\}_{m=1}^n$ are a practical, convenient, and
admissible  choice satisfying Assumption~\ref{ass:spde0}. Setting
\begin{equation} \label{eqn:AB1}
W^0_{\tau} = 0
\qquad \mbox{and} \qquad
W^n_{\tau} = \sum_{m=1}^n \xi_{\tau}^m, 
\qquad n = 1,2, \ldots, N,
\end{equation}
the piecewise linear interpolant of $\{W^n_{\tau}\}_{n=0}^N$ is the
discrete Ito integral with $g^{m-1}_\tau \equiv 1$ and plays the role of a
standard Wiener process in the discrete setting.

Under Assumption~\ref{ass:spde0}, the
process $\{X^n_{\tau} \}_{n={0}}^N$ of equation \eqnref{:discreteIto} is
adapted to $\{\calF^n\}_{n=0}^N$, and
$$
\bbE[X^n_{\tau} - X^{n-1}_{\tau} | \calF^{n-1}] 
= \bbE[g^{n-1}_{\tau} \xi^n_{\tau}  | \calF^{n-1}] 
= g^{n-1}_{\tau} \bbE[\xi^n_{\tau}  | \calF^{n-1}] 
= 0.
$$
This shows that $\{X^n_{\tau}\}_{n={0}}^N$ is a (discrete) martingale;
the discrete Ito isometry is then immediate,
\begin{eqnarray*}
\bbE[ \normH{X^n_{\tau}}^2]
&=& \bbE \Bigl[\sum_{k,m = 1}^n 
(g^{k-1}_{\tau}, g^{m-1}_{\tau})_H \xi^k \xi^m_{\tau} \Bigr] \\
&=& \sum_{m = 1}^n 
\bbE\bigl[ \normH{g^{m-1}_{\tau}}^2 |\xi^m_{\tau}|^2 \bigr] +
2 \sum_{k < m} 
\bbE\bigl[ (g^{k-1}_{\tau}, g^{m-1}_{\tau})_H \xi^k_{\tau} \xi^m_{\tau} \bigr] \\
&=& 
\sum_{m = 1}^n  \bbE[ \normH{g^{m-1}_{\tau}}^2 ] \tau.
\end{eqnarray*}
The last line follows from Assumption \ref{ass:spde0}$_4$,
$$
\bbE[\normH{g^{m-1}_{\tau}}^2 |\xi^m_{\tau}|^2]
= \bbE[\normH{g^{m-1}_{\tau}}^2] \, \bbE[|\xi^m_{\tau}|^2]
= \bbE[\normH{g^{m-1}_{\tau}}^2] \tau,
$$
and when $k < m$ the cross terms vanish,
$$
\bbE\left[(g^{k-1}_{\tau}, g^{m-1}_{\tau})_H \xi^k_{\tau} \xi^m_{\tau} \right]
=\bbE\left[(g^{k-1}_{\tau}, g^{m-1}_{\tau})_H \xi^k_{\tau}\right] \, \bbE[\xi^m_{\tau} ]
=\bbE\left[(g^{k-1}_{\tau}, g^{m-1}_{\tau})_H \xi^k_{\tau} \right] \cdot 0. 
$$
A similar calculation shows that its discrete quadratic variation is
$$
\langle X^n \rangle (u,v) = \sum_{m=1}^n \tau (g^{m-1},u)_H (g^{m-1},v)_H,
\qquad n \geq 1.
$$
Note that in the discrete setting $\langle X^n \rangle (u,v)$ must be
$\calF^{n-1}$--measurable (predictable). The following
theorem shows that the quadratic and cross variations characterize
the Ito integral.

\begin{theorem} \label{thm:mglRepresentation}
  Let $U$ be a separable Banach space and $H$ a Hilbert space
  with $U \embed H \embed U'$ dense inclusions. Let $(\Omega, \calF,
  \{\calF(t)\}_{0 \leq t \leq T}, \bbP)$ be a filtered probability
  space and $X$, $g$, and $W$ be $U'$, $H$ and real--valued process
  respectively, with $X$ and $W$ continuous.
  Suppose that for each $v \in U$ the processes
  $$
  W(t), \quad
  W^2(t) - t, \quad
  (X(t), v), \quad
  (X(t),v)^2 - \int_0^t (g(s),v)_H^2 \, ds, \quad
  (X(t),v) W(t) - \int_0^t (g(s),v)_H \, ds,
  $$
  are all real-valued martingales.
  Then $W$ is a standard Wiener process and
  $$
  (X(t), v) = \int_0^t (g(s), v)_H \, dW(s), 
  \qquad v \in U.
  $$
\end{theorem}

\begin{proof} (sketch)
  We show that the quadratic variation of $(X(t) - \int_0^t g \, dW, v)$
  vanishes using the following calculus for the quadratic variations
  of real-valued martingales $X$ and $Y$:
  \begin{itemize}
  \item $\langle X + Y \rangle = \langle X \rangle + 2 \langle X, Y
    \rangle + \langle Y \rangle$, where the cross variation
    $\langle X, Y \rangle$ is determined from the parallelogram law,
    $
    4 \langle X, Y \rangle = \langle X + Y \rangle - \langle X - Y \rangle.
    $

  \item If $Y(t) = \int_0^t g \, dW$ then $\langle X, Y
      \rangle(t) = \int_0^t g\,d \langle X,W \rangle(t)$.
  \end{itemize}
  Using this calculus for the adapted process $(X(t) - \int_0^t g \,
  dW, v)$ gives the result.
  \begin{eqnarray*}
    \lefteqn{\left\langle(X,v)
        -\int_0^\cdot(g,v)_H\,dW \right\rangle(t)} \\
    &=&\left\langle(X,v)\right\rangle(t)-2\left\langle(X,v),
      \int_0^\cdot(g,v)_H\,dW \right\rangle(t)
    +\left\langle\int_0^\cdot(g,v)_H\,dW\right\rangle(t)
    \\
    &=&\left\langle(X,v)\right\rangle(t)-2\int_0^t(g,v)_H \,
    d\left\langle(X,v),W\right\rangle
    +\left\langle\int_0^\cdot(g,v)_H\,dW\right\rangle(t)
    \\
    &=&\int_0^t(g(s),v)_H^2\,ds-2\int_0^t(g(s),v)_H^2\,ds
    +\int_0^t(g(s),v)_H^2\,ds=0.
  \end{eqnarray*}
  The middle term takes the form shown 
  since
    \begin{align*}
      \langle (X,v) + W \rangle(t) 
      &= \int_0^t (g,v)^2_H \, dt + 2 \int_0^t (g,v)_H \, dt + t ,
      \qquad \text{ and }\\
      \langle (X,v) - W \rangle(t) 
      &= \int_0^t (g,v)^2_H \, dt - 2 \int_0^t (g,v)_H \, dt + t .
    \end{align*}
    Then $\langle (X,v), W \rangle (t) = \int_0^t (g,v)_H \, dt$, so that
    $d \langle (X,v), W \rangle (t) = (g(t),v)_H \, dt.$
  \end{proof}

The (discrete) BDG inequality, stated next, shows that moments of a
discrete $\{\calF^n\}_{n=0}^N$-martingale taking values in a Hilbert
space may be bounded by their quadratic variations, \cite[Remark
3.3]{Pisier_75}.

\begin{theorem}[Burkholder-Davis-Gundy (BDG)] \label{thm:BDG} Let
  $(\Omega, \calF, \bbP)$ be a probability space with (discrete)
  filtration $\{\calF^n\}_{n={0}}^N$ and let $\{X^n_{\tau}\}_{n=0}^N$
  with $X^0_{\tau}\equiv 0$ be a (discrete)
  $\{\calF^n\}_{n={0}}^N$-martingale taking values in a separable
  Hilbert space $H$.  Then for each $p \geq 1$ there exist constants
  $0 < c_p < C_p$ such that
  $$
  c_p \bbE\left[
    \left(\sum_{n=1}^N \normH{X^n_{\tau}-X^{n-1}_{\tau}}^2 \right)^{p/2} \right]
  \leq \bbE\left[ {\max_{0 \leq n \leq N}} \normH{X^n_{\tau}}^p \right]
  \leq C_p \bbE\left[
    \left(\sum_{n=1}^N \normH{X^n_{\tau}-X^{n-1}_{\tau}}^2 \right)^{p/2} \right].
   $$
\end{theorem}

When the martingale is the discrete Ito integral
\eqnref{:discreteIto}, the moments of the quadratic variation can be
bounded by the Bochner norms of {$\{g^n_{\tau}\}_{n=0}^{N-1}$}, which
is the content of the following lemma.

\begin{lemma} \label{lem:mglVariation} Let $(\Omega, \calF, \bbP)$ be
  a probability space with (discrete) filtration $\{\calF^n\}_{n=0}^N$
  and let $\{X^n_{\tau}\}_{n=0}^N$ be the discrete ($\{\calF^n\}_{n=0}^N$-adapted) Ito integral in
  \eqnref{:discreteIto} taking values in a Hilbert space $H$ with data
  $\{g^n_\tau\}_{n=0}^{N-1} \subset L^p(\Omega,H)$ and increments
  $\{\xi^n_\tau\}_{n=1}^N$ satisfying Assumption \ref{ass:spde0}, with
  $2 \leq p$.
  Then
  $$
   \bbE\left[
     \left(\sum_{m=1}^n \normH{g^{m-1}_{\tau} \xi^m_\tau}^2 
     \right)^{p/2} \right]
   \leq C_p (n \tau)^{p/2 - 1} 
   \sum_{m=1}^n \tau \norm{g^{m-1}_{\tau}}_{L^p(\Omega,H)}^p
   \qquad {n=1,\ldots,N},
   $$
   where $C_p>0$ is a constant depending upon $p$ and
   the constant in Assumption \ref{ass:spde0}$_3$.
\end{lemma}

\begin{proof} (sketch) 
The discrete process with $X^0_{\tau} = 0$ and $X^n_{\tau} = \sum_{k=1}^n
g_\tau^{k-1} \xi^k_\tau$ for $n = 1,2,\ldots$ is a (discrete) martingale,
and the Burkholder--Rosenthal inequality \cite[Theorem 5.50]{Pi16}
bounds the middle term in the BDG inequality as
$$
\Bbb E\,\left[\max_{1\le k\le n}
\normH{\sum_{m=1}^k g^{m-1}_\tau \xi^m_\tau}^p \right]
\leq \beta_p \bbE \left[\sum_{k=1}^n 
\bbE\left[\normH{g^{k-1} \tau\xi^k_\tau}^2 \sst
  \calF^{k-1}\right]\right]^{p/2}                 
+\beta_p \bbE\left[\max_{1\le k\le n}
\normH{g^{k-1}_\tau \xi^k_\tau}^p\right],
$$
where $\beta_p$ is a constant depending only upon $p \geq 2$.
Since $g_\tau^{k-1}$ is $\calF^{k-1}$-measurable and $\xi^k_\tau$
is independent of $\calF^{k-1}$ it follows that
\begin{align*}                                                                  
\bbE\left[\sum_{k=1}^n 
\bbE\left[\normH{g^{k-1}_\tau \xi^k_\tau}^2 \sst \calF^{k-1}\right] 
\right]^{p/2}
&= \bbE\left[ \sum_{k=1}^n \normH{g^{k-1}_\tau}^2
\bbE\left[(\xi^k_\tau)^2\right] \right]^{p/2} \\ 
&\leq C \tau^{p/2} 
\bbE\left[\sum_{k=1}^n \normH{g^{k-1}_\tau}^2\right]^{p/2} \\
&\leq C \tau^{p/2}n^{p/2-1}
\bbE\left[\sum_{k=1}^n \normH{g^{k-1}_\tau}^p\right],
\end{align*}
where $C$ is the constant  in Assumption \ref{ass:spde0}$_3$.
The bound on second term is direct,
\begin{align*}                                                                  
\bbE\left[\max_{1\le k\le n} 
\normH{g^{k-1}_\tau\xi^k_\tau}^p\right]
\leq \bbE\left[\sum_{k=1}^n 
\normH{g^{k-1}_\tau\xi^k_\tau}^p \right] 
=\sum_{k=1}^n\Bbb E\left[\normH{g^{k-1}_\tau}^p \right]
\bbE\left[(\xi^k_\tau)^p \right] 
\leq C\tau^{p/2} \sum_{k=1}^n\bbE\left[ \normH{g^{k-1}_\tau}^p \right].
\end{align*}
\end{proof}

\subsubsection{Convergence in Law}
Below we construct numerical schemes whose solutions converge in law
to a limit.  For this purpose,
it is necessary to show that
solutions of a discrete approximation of the equation (\ref{eqn:spde}) will pass to
solutions of the SPDE (\ref{eqn:spde}) with this mode of convergence. In the deterministic case,
the following two properties are used ubiquitously to identify limits:
\begin{itemize}
\item Norm bounded subsets of reflexive Banach spaces are weakly
  sequentially compact. That is, if $A \subset U$ is a norm bounded
  set of a reflexive Banach space $U$, then there exist a sequence
  $\{u_n\}_{n=1}^\infty \subset A$ and $u$ in the closed convex hull
  of $A$ such that $u_n \weak u$.

\item Continuous convex functions $\psi:U \rightarrow \Re$ are
  sequentially weakly lower semi--continuous. That is, if
  $\{u_n\}_{n=1}^\infty \subset U$ and $u_n \weak u$ then $\psi(u)
  \leq \liminf_{n\rightarrow \infty} \psi(u_n)$.
\end{itemize}
We present analogous results for random variables with convergence in
law in place of weak convergence. 

If $(\Omega, \calF, \bbP)$ is a probability space and $X:(\Omega,\calF)
\rightarrow (\bbX, \calB(\bbX))$ is a random variable with values in
the topological space $\bbX$ with its Borel $\sigma$-algebra
$\calB(\bbX)$, then the law of $X$ on $\bbX$ is the measure
$$
\calL(X)[B] = \bbP[\omega \in \Omega \sst X(\omega) \in B],
\qquad B \in \calB(\bbX).
$$
If $\{X_n\}_{n=1}^\infty$ is a sequence of such random variables, the
laws converge (weakly) to the measure $\bbPt$ on $(\bbX, \calB(\bbX))$,
and we write $\calL(X_n) \Rightarrow \bbPt$, iff
$$
\bbEt[\psi]
\equiv \int_{\bbX} \psi(x) \, d\bbPt(x)
= \lim_{n \rightarrow \infty} \bbE[\psi \circ X_n]
\qquad \psi \in C_b(\bbX),
$$
where $C_b(\bbX)$ denotes the set of bounded continuous real-valued
functions on $\bbX$. In the current context $\bbX$ will typically be a
product of spaces; for example, 
\begin{align*}
\bbX &= G[0,T;U'] \cap \LrU_{weak} \times \LqpUp,  \\
\text{ or } \qquad 
\bbX &= G[0,T;U'] \cap \LrU_{weak} \times \LqpUp_{weak},
\end{align*}
where $\LrU_{weak}$ and $\LqpUp_{weak}$ denote the spaces $\LrU$ and
$\LqpUp$ endowed with the weak topology. 
The space $C_b(\bbX)$ has insufficient functions to identify the
limits when the factor spaces have the weak topology; however, the
lemma below shows that in many situations a larger class of test
functions is available when the sequence of laws are tight.

\begin{definition} \label{def:tight} Let $\bbX$ be a topological space
  and $\calB(\bbX)$ denote its Borel $\sigma$-algebra. 
  \begin{itemize}
  \item A sequence of probability measures $\{\bbP_n\}_{n=1}^\infty$
    on $(\bbX, \calB(\bbX))$ is tight if for every $\epsilon > 0$
    there exists a compact set $K_\epsilon \subset \bbX$ for which
    $\bbP_n[K_\epsilon] \geq 1 - \epsilon$ for all $n = 1, 2, \ldots$.

  \item A sequence of random variables $\{X_n\}_{n=1}^\infty$ taking
    values in $\bbX$ is tight if their laws $\{\calL(X_n)\}_{n=1}^\infty$
    are tight.
  \end{itemize}
\end{definition}

Tight subsets of probability measures on separable metric spaces 
play a similar role to norm bounded sequences in reflexive Banach spaces
in the sense that they are both weakly sequentially compact.

\begin{lemma}\label{lem:mapping_theorem}
  Let $\mathbb X$ be a topological space with a countable sequence of
  continuous functions separating points and $\{\bbP_k\}_{k=1}^\infty$
  be tight on $\mathbb X$ and $\bbP_k\Rightarrow \bbP$.
\begin{enumerate}
\item Let $\zeta_k,\zeta:\mathbb X\to\Bbb R$ be Borel measurable for
  $k\in\Bbb N$. Define
  $$
  N=\{x\in{\mathbb X} \sst \exists\{x_k\},\,x_k\to x\text{ in }\mathbb
  X\text{ such that }\{\zeta_k(x_k)\}\text{ does not converge to
  }\zeta(x)\}
  $$
  and assume that $\bbP^*[N]=0$, i.e.
  $
  \inf\,\{\bbP[B]: N\subseteq B\in \calB(\mathbb X)\}=0.
  $
  
  Then $\bbP_k[\zeta_k\in\cdot]\Rightarrow \bbP[\zeta\in\cdot]$ and if
  $$
  \lim_{R\to\infty}\left[\sup_k\int_{[|\zeta_k|>R]}|\zeta_k|\,d\bbP_k\right] 
  = 0
  \qquad \text{ then } \qquad
  \lim_{k\to\infty}\int_{\mathbb X}\zeta_k\,d\bbP_k
  = \int_{\mathbb X}\zeta\,d\bbP.
  $$
  In particular, if $\epsilon  > 0$ and
  $$
  \sup_k\int_{\mathbb X}|\zeta_k|^{1+\varepsilon}\,d\bbP_k < \infty
  \qquad \text{ then } \qquad
  \lim_{k\to\infty}\int_{\mathbb X}\zeta_k\,d\bbP_k
  = \int_{\mathbb X}\zeta\,d\bbP.
  $$

\item Let $\zeta:\mathbb X\to[0,\infty]$ be such that $[\zeta\le t]
  \equiv \{x \in \mathbb X \sst \zeta(x) \leq t\}$ is sequentially closed
  for every $t\ge 0$. Then $\zeta$ is $\bbP$-measurable as well as
  $\bbP_k$-measurable for every $k\ge 1$ and
  $$
  \int_{\mathbb X}\zeta\,d\bbP\le\liminf_{k\to\infty}\int_{\mathbb
    X}\zeta\,d\bbP_k.
  $$
\end{enumerate}
\end{lemma}

This lemma may be viewed as an extension of the classical Portmanteau
theorem and is similar to the mapping theorem in \cite[Theorem
2.7]{Billingsley99}. We provide a proof of this result in the
Appendix.  The following corollary uses this lemma to show that
sequentially continuous test functions are available in the current
setting. The class of weakly sequentially continuous functions is
substantially larger than the weakly continuous functions since weakly
convergent sequences are norm bounded while neighborhoods in the weak
topology are not.

\begin{corollary} \label{cor:Law} Let $(\Omega, \calF, \bbP)$ be a
  probability space, $1 < p < \infty$, and $\calU$ be a separable
  reflexive Banach space, and let $\calU_{weak}$ denote
  $\calU$ endowed with the weak topology.
  \begin{itemize}
  \item Let $\psi:\calU \rightarrow \Re$ be weakly sequentially continuous. If
    the laws of $\{u_n\}_{n=1}^\infty$ converge on $\bbX = \calU_{weak}$
    to a measure $\bbPt$ and
    $\{\psi(u_n)\}_{n=1}^\infty$ is bounded in $L^p(\Omega)$, then
    $\psi(u)$ is integrable on $(\bbX, \calB(\bbX), \bbPt)$ and
    $$
    \bbEt\left[ \psi(u) \vph\right] 
    = \lim_{n\rightarrow \infty} \bbE\left[\psi(u_n) \vph\right].
    $$

  \item Let $\psi:\calU \rightarrow \Re$ be continuous, convex, and
    bounded below.  If the laws of $\{u_n\}_{n=1}^\infty$ converge on
    $\bbX = \calU_{weak}$ to a measure $\bbPt$, then $\psi(u)$ is
    measurable on $(\bbX, \calB(\bbX))$ and
    $$
    \bbEt\left[ \psi(u) \vph\right] 
    \leq \liminf_{n\rightarrow \infty} \bbE\left[ \psi(u_n) \vph\right].
    $$
  \end{itemize}
\end{corollary}

\begin{proof} (sketch) The first result will follow from the first
  statement of the lemma.  Since the Borel $\sigma$-algebras for $\calU$
  and $\calU_{weak}$ coincide $\psi$ is Borel measurable. In addition,
  since $\psi$ is weakly sequentially continuous it follows that the
  set $N$ in the lemma is empty.

  The final result follows from the second statement of the lemma and
  Mazur's theorem which states that continuous convex functions on a
  Banach space are weakly lower semi--continuous.
\end{proof}

The following example illustrates the use of these results to identify
and bound  initial and final values for the evolution problems
under consideration.

\begin{example} \label{eg:ic}
  Let $U$ be a separable Banach space, $H$ a Hilbert space, and
  $U \embed H \embed U'$ be dense embeddings.
  Suppose that $\{u_n\}_{n=1}^\infty$ are random variables on
  $(\Omega, \calF, \bbP)$ taking values in $G[0,T;U']$, and
  $\calL(u_n) \Rightarrow \bbPt$.

  For $t \in [0,T]$ fixed, the mapping $u \in G[0,T;U'] \mapsto u(t)
  \in U'$ is Borel, and if $p \geq 1$ the function
  $\zeta:U' \mapsto [,\infty]$ given by
  $$
  \zeta(u) = \left\{ \begin{array}{cc}
      \normH{u}^p & u \in H, \\
      \infty & \text{ otherwise},
    \end{array} \right.
  $$
  is convex and lower semi--continuous. It follows from
  the second statement of Lemma \ref{lem:mapping_theorem} that
  $$
  \bbEt\left[\normH{u(t)}^p \right]
  \leq \liminf_{n \rightarrow \infty} \bbE\left[\normH{u_n(t)}^p \right].
  $$

  Next, suppose that $u_n(0)$ converges to a limit in $L^p(\Omega,H)$.
  Then the laws of $(u_n(0), u_n)$ are tight on $H \times G[0,T; U']$,
  so passing to a subsequence we may assume their laws converge to a
  limit, $\calL(u_n(0), u_n) \Rightarrow \bbQ$, on $H \times G[0,T;
  U']$. If $f \in C_b(G[0,T; U'])$ then
  $$
  \int_{H \times G[0,T;U']} f(u) \, d\bbQ(u^0, u)
  = \bbE^{\mathbb Q}[f(u)] 
  = \lim_{n \rightarrow \infty} \bbE[f(u_n)]
  = \int_{G[0,T; U']} f(u) \, d\bbPt(u),
  $$
  shows $\bbPt$ is the second marginale of $\bbQ$. 

  Assume that
  $\normH{u^0_n}$ and $\normUp{u_n}$, and hence $\norm{(u^0_n,u_n)}_{H
    \times G[0,T;U']}$, have bounded moments of order $p > 1$, and fix
  $v \in U$. Then the mapping $(u^0, u) \mapsto |(u^0 -
  u(0),v)|$ is continuous on $H \times G[0,T;U']$, and it follows from
  the first statement of the lemma that
  $$
  \bbE^{\mathbb Q}\left[\big|(u^0 - u(0),v)\big| \right] 
  = \lim_{n \rightarrow \infty} 
  \bbE\left[\big|(u_n(0) - u_n(0),v)\big| \right] 
  = 0,
  $$
  whence $u(0)=u^0$  $\bbQ$-almost surely.
  From the Tonelli theorem we then conclude
  \begin{eqnarray*}
    \bbEt\left[\normH{u(0)} \right]
    &=& \int_{G[0,T;U']} \normH{u(0)} \, d\bbPt(u) \\
    &=& \int_{H \times G[0,T;U']} \normH{u(0)} \, d\bbQ(u^0,u) \\
    &=& \int_{H \times G[0,T;U']} \normH{u^0} \, d\bbQ(u^0,u) \\
    &=& \lim_{n \rightarrow \infty} \bbE\left[\normH{u^0_n} \right],
  \end{eqnarray*}
  the last line following since $(u^0,u) \mapsto \normH{u^0}$ is
  continuous on $H \times G[0,T; U']$. Similarly, if $\normH{u^0_n}$
  has moments of order $p > 1$ then
  $$
  \bbEt\left[\normH{u(0)}^s \right] 
  = \lim_{n \rightarrow \infty} \bbE\left[\normH{u^0_n}^s \right],
  \qquad 1 \leq s < p.
  $$
  \end{example}

\subsection{Stochastic Partial Differential Equations} \label{sec:stochPDE's} 
Combining the ideas from the previous section provides a formulation
of the stochastic evolution equation \eqnref{:spde} amenable to
analysis by results from functional analysis and probability theory.
Letting $U \embed H \embed U'$ be dense embeddings and writing $a(u,v)
= (A(u),v)$, a solution of \eqnref{:spde} may be viewed as a process
taking values in $U$ which at each time $t \in [0,T]$ satisfies
\begin{equation} \label{eqn:Spde}
(u(t),v)_H + \int_0^t a(u,v)\, ds 
= (u^0, v)_H 
+  \int_0^t ( f,v )\, ds
+ \int_0^t (g,v)_H \, dW,
\qquad
v \in U.
\end{equation}
The last integral in this equation is the Ito integral corresponding
to a Wiener process $W$ defined on the filtered probability space.
The distinction between a (stochastically) weak and strong solution
of (\ref{eqn:Spde}) is as follows:

\begin{itemize}
\item For a stochastically strong solution of \eqnref{:spde}, a
  filtered probability space $(\Omega, \calF, \{\calF(t)\}_{0 \leq t
    \leq T}, \bbP)$ and random variables $f$, $g$, $W$, and $u^0$
  {\em are specified}, and the
  solution $u:[0,T] \rightarrow U$ is a process adapted to $\{
  \calF(t)\}_{0 \leq t\leq T}$ which satisfies \eqnref{:Spde}.

\item For a stochastically weak solution of \eqnref{:spde}, {\em laws}
  $\bbP_f$, $\bbP_g$ and $\bbP_0$ of the data
  are specified, and a solution consists of a probability space
  $(\Omegat, \calFt, \{\calFt(t)\}_{0 \leq t \leq T}, \bbPt)$ and
  adapted processes $u$, $f$, $g$, and $W$, which satisfy
  \begin{itemize}
  \item $\calL(f) = \bbP_f$,
  \item $\calL(g) = \bbP_g$,
  \item $\calL(W)$ is an instance of the standard Wiener measure,
  \item $\calL(u(0)) = \bbP_0$,
  \end{itemize}
  and $(u,f,g,W)$ satisfy \eqnref{:Spde} $\bbPt$ almost surely.
\end{itemize}

Clearly a strong solution is also a weak solution, the major
distinction between the two concepts is that the construction of a
filtered probability space is a part of the solution process for weak
solutions. Since filtered probability spaces and Wiener processes are
not available in a computational context, only weak solutions are
commutable in practice.

\begin{definition} \label{def:mglSoln} Let $T>0$ and $U \embed H$ be a
  dense embedding of the Banach space $U$ into a Hilbert space $H$ so
  that $U \embed H \embed U'$.  Then $(\Omegat, \calFt, \{
  \calFt(t)\}_{0 \leq t\leq T}, \bbPt)$ and random variables
  $u$, $f$, $g$, and $W$ on this space are a weak martingale solution of
  \eqnref{:Spde} if
  \begin{itemize}
  \item[(i)] $(\Omega, \calF, \{ \calF(t)\}_{0 \leq t\leq T}, \bbP)$ is a
    filtered probability space satisfying the usual conditions, 
    $f$ and $g$ are adapted, and $u^0$ is $\calF(0)$-measurable.

  \item[(ii)] $W = \{ W(t) \, \vert \, 0 \leq t \leq T\}$ is a
    standard real-valued Wiener process on
    $(\Omegat, \calFt, \{ \calFt(t)\}_{0 \leq t\leq T}, \bbPt)$.

  \item[(iii)]$u: [0,T] \times \Omega \rightarrow U$ is adapted to
    $\{ \calFt(t)\}_{0 \leq t\leq T}$, and
    \begin{enumerate}
    \item[(a)] $u \in C[0,T; U']$ $\bbPt$--a.s.,
    \item[(b)] equation \eqnref{:Spde} holds $\bbPt$-a.s., for
      every $v \in U$ and every $0 \leq t \leq T$.
    \end{enumerate}
  \end{itemize}
\end{definition}

The remainder of this manuscript considers the numerical approximation
of weak martingale solutions using (pseudo) random number generators
to simulate the role of noise in \eqnref{:Spde}.  For simplicity of
presentation we will consider a real-valued Wiener process; extensions
to infinite-dimensional and cylindrical noise are outlined in Section
\ref{ssect:cylinder}.

\subsubsection{Ito's Formula}
A version of Ito's formula is available for weak martingale solutions
of stochastic PDE's taking values in a Banach space
\cite{Kr13,KR1,PrRo07}. The Ito formula stated next considers
weak martingale solutions of the equation $du = F \, dt + g \, dW$ with
$F$ taking values in $U'$ and $g$ taking values in the pivot space $H$.
Writing equation \eqnref{:spde} as
$$
du = (f - A(u)) \, dt + g \, dW \equiv F \, dt + g \, dW,
$$
shows that it takes the form assumed in the theorem.

\begin{theorem} \label{thm:Ito} Let $(\Omega,\calF,
  \{\calF_t\}_{t=0}^T, \bbP)$ be a filtered probability space
  satisfying the usual conditions, $U$ be a separable Banach space,
  $H$ a Hilbert space, and $U \embed H \embed U'$ be dense embeddings.
  With $1 < q < \infty$, let $F \in L^{q'}(\Omega, \LqpUp)$ and $g \in
  L^2(\Omega, \LtwoH)$ be jointly measurable (as functions of
  $(t,\omega)$) adapted processes, and $W$ be a standard Wiener
  process.  If $u^0 \in L^2(\Omega,H)$, and a process $u \in
  L^q(\Omega, \LqU)$ with $(u,g)_H \in L^2(\Omega \times (0,T))$
  satisfies
  $$
  (u(t), v) = (u^0,v) + \int_0^t (F(s),v) \, ds + \int_0^t (g(s),v)_H \, dW(s),
  \qquad v \in U,
  $$
  then there is an adapted version of $u$ with values in $C[0,T;H]$
  for which
  $$
  \bbE \left[ \sup_{0 \leq t \leq T} \normH{u(t)}^2 \right] < \infty,
  $$
  and
  $$
  \bbE\left[(1/2) \Vert {u(t)}\Vert^2_H\right] 
  = \bbE\left[(1/2) \normH{u^0}^2 + \int_0^t (F(s), u(s)) 
    + (1/2) \normH{g(s)}^2 \, ds \right].
  $$
\end{theorem}

\subsubsection{Uniqueness of Solutions} \label{sec:unique}
This section shows that if the solution of the deterministic equation
is unique then the laws of weak martingale solutions of the
corresponding SPDE with additive noise will also be unique. Writing
equation \eqnref{:spde} as
$$
du = (f \, dt + g \, dW) - A(u) \, dt \equiv dV - A(u) \, dt,
$$
then (the law of) $V$ depends upon (laws of) the data $(f,g,W)$. Theorem
\ref{thm:uniqueLaw} below shows that the law of a solution $u$ to an equation
of this form will depend only upon the law of $V$ when $A(.)$ satisfies
the following assumption. 

\begin{assumption} \label{ass:A} If $\lambda>0$ and $u_1, u_2 \in C[0,T;U] \cap
  L^r[0,T;U']$ satisfy $A(u_1), A(u_2) \in L^1[0,T;U']$ and
  $$
  (u_2(t)-u_1(t), w)_H 
  + \int_0^t \lambda \left( A(u_2(s))- A(u_1(s)), v \vph\right) \,ds = 0,
  \quad t\in[0,T], \quad v \in U,
  $$
  then $u_1=u_2$. (Note that if this holds for some $T>0$ then it holds for
  all $T>0$.)  
\end{assumption}

This assumption will always be considered in the context where $U$ is
a separable Banach space, $H$ is a Hilbert space, the embeddings $U
\embed H \embed U'$ are dense, and $A:U \rightarrow U'$.

\begin{definition} 
  Let $(\Omega,\calF,\bbP)$ be a probability space and $\bbX_1 =
  C[0,T;U'] \cap L^r[0,T;U']_{weak}$ with $1 < r < \infty$, and $A:U
  \rightarrow U'$.  Then a pair of random variables $(u,V)$ taking
  values in $\bbX_1 \times C[0,T;U']$ satisfy
  \begin{equation}\label{eqdet}
    du=dV - A(u)\,dt,
  \end{equation}
  if $ \bbP\,[u\in S]=1 $ for some $\sigma$-compact set $S$ in
  $\bbX_1$, $A(u) \in L^1[0,T;U']$ almost surely, and
  $$
  \bbP\left[(u(t),v)_H = (V(t),v) - 
    \int_0^t (A(u(s)),v) \,ds\right]=1,\qquad t\in[0,T], \quad v \in U.
  $$
\end{definition}

The following theorems establish uniqueness when the partial
differential operators satisfying Assumption \ref{ass:A}, and may be
viewed as extensions of the classical Yamada-Watanabe theory to the
situation where the data $f$ and $g$ are random.

\begin{theorem}[Joint Uniqueness in Law] \label{thm:uniqueLaw}
  Let Assumption \ref{ass:A} hold. If $(u^i,V^i)$ satisfy
  \eqref{eqdet} on a probability space $(\Omega^i,\calF^i,\Bbb
  P^i)$ and $\calL(V^0)=\calL(V^1)$, then $\calL(u^0,V^0)=\calL(u^1,V^1)$.
\end{theorem}

\begin{theorem}[Strong Existence]\label{thm:SE} 
  Let Assumption \ref{ass:A} hold and let there exist a solution
  $(\tilde u,\tilde V)$ of \eqref{eqdet} on some probability space. If
  $(\Omega,\calF,\bbP)$ is a probability space, $V$ is a
  $C[0,T;U^\prime]$-valued random variable with $\calL(V)=\calL(\tilde
  V)$ then there exists a unique $\Bbb X_1$-valued random variable $u$
  with a $\sigma$-compact range such that $(u,V)$ is a solution of
  \eqref{eqdet}. Moreover, $u$ is $(\calF_t^{V,0})$-adapted where
  $(\calF_t^{V,0})$ denotes the $\bbP$-augmentation of the filtration
  generated by $V$.
\end{theorem}

The proofs of these two theorems are presented in the Appendix.

\section{Numerical Approximation of SPDE's} \label{sec:introSpde} To
construct numerical approximations of the weak statement
\eqnref{:Spde} let $U_h \subset U$ be a (finite-dimensional) subspace,
and $\{ t^n\}^N_{n=0}$ be a uniform partition of $[0,T]$ with time
step $\tau = T/N > 0$. A (pseudo) random number generator is used to
generate sampled random variables $\xi^n_{\tau}(\omega) \in \Re$
satisfying Assumptions \ref{ass:spde0}. Then $u^n_{h\tau} \equiv
u^n_{h\tau}(\omega) \in U_h$ is a solution of
\begin{equation} \label{eqn:Spdeh}
(u^n_{h\tau}-u^{n-1}_{h\tau}, v_h)_H + \tau a(u^n_{h\tau}, v_h) 
= \tau ( f^n_{h\tau}, v_h) + (g^{n-1}_{h\tau},v_h)_H \xi^n_{\tau},
\quad
v_h \in U_h,
\quad 1 \leq n \leq N.
\end{equation}
In this equation,  $f^n_{h\tau}$ is  a $U'$-valued approximation of $f$, $g^n_{h\tau}$ is a $H$-valued 
approximation of $g$, and $u^0_{h\tau}$ is a $U_h$-valued
approximation of $u^0$; for example,
\begin{equation}\label{eqn:AA1}
f^n_{h\tau} = \frac{1}{\tau} \int_{t^{n-1}}^{t^n} f(s) \, ds 
\quad \text{ and } \quad
g^n_{h\tau} = \frac{1}{\tau} \int_{t^{n-1}}^{t^n} g(s) \, ds,
\end{equation}
and $u^0_{h\tau}$ is the orthogonal projection of $u^0$ onto $U_h \subset H$.
In general, $f$ and $g$ may depend upon $u$ ($u_{h \tau}$ in the
discrete case), so both $h$ and $\tau$ are included in the notation
$f^n_{h\tau}$ and $g^n_{h\tau}$.

The specific bounds available for solutions of a particular equation
\eqnref{:spde} depend in an essential fashion upon the structure of
the operator $A$.  For this reason a passage to the limit in this term
in a numerical scheme is problem dependent. In contrast, there is a
commonality of the structure in the temporal terms which facilitates a
convergence theory for implicit Euler approximations of this class of
problems provided bounds upon the solution are available.

Writing $F(t) = f(t) - A(u(t))$, the spatial dependence
of the equation is characterized by a single process taking values
in $U'$. With this notation the implicit Euler scheme \eqnref{:Spdeh}
becomes: Find $u^n_{h\tau}(\omega) \in U_h$ such that
\begin{equation} \label{eqn:eulerF}
(u^n_{h\tau},v_h)_H 
= (u^{n-1}_{h\tau},v_h)_H 
+ \tau ( F^n_{h\tau}, v_h )
+ (g^{n-1}_{h\tau}, v_h)_H \xi^n_{\tau},
\qquad v_h \in U_h, 
\,\, 1 \leq n \leq N,
\end{equation}
with the $U'$-valued $F_{h\tau}^n$ defined by
$F_{h\tau}^n(v) = f_{h\tau}^n(v) - a(u_{h\tau}^n, v)$.

Theorem \ref{thm:main} below establishes conditions under which
solutions of this abstract difference scheme will converge to a weak
martingale solution. Assumption \ref{ass:spde0} on the stochastic
increments $\{ \xi_\tau^n\}_{n=1}^N$, and the following assumptions on
the data and discrete spaces will be assumed throughout.

\begin{assumption} \label{ass:spde1} $U \embed H$ is a dense embedding
  of a Banach space $U$ into a Hilbert space $H$. The discrete subspace 
  $U_h \subset U$, and data of the numerical scheme \eqnref{:eulerF}
  with time step $\tau = T/N$ with $N \in \bbN$ and $t^n \equiv n
  \tau$ satisfy:
  \begin{enumerate}
  \item $(\Omega,\calF,\{\calF^n\}_{n=0}^N, \bbP)$ is a
    (discretely) filtered probability space satisfying the usual
    assumptions.

  \item $\{F^n_{h\tau}\}_{n=1}^N$ 
    is adapted to $\{\calF^n\}_{n=1}^N$ with values in $U'$.

  \item $\{g^{n}_{h\tau}\}_{n=0}^{N-1}$
    is adapted to $\{\calF^n\}_{n=0}^{N-1}$ with values in $H$.

  \item The initial datum $u^0_{h\tau}$ is an $H$-valued random variable
    that is $\calF^0$-measurable.

  \item \label{it:hDense} For each $v \in U$, there exists a sequence
    $\{v_h\}_{h > 0}\subset U_h$ such that $\lim_{h \rightarrow 0} v_h
    \rightarrow v$.

  \item The restrictions of the orthogonal projections $P_h:H
    \rightarrow U_h$ to $U$ are stable in the sense that there exists
    a constant $C > 0$ independent of $h > 0$ such that
    $\normU{P_h(v)} \leq C \normU{v}$.
  \end{enumerate}
\end{assumption}

The last two conditions are density and stability conditions on the
spatial discretizations and, in a finite element context, are
satisfied under mild restrictions on the triangulations of the domain
\cite{CrTh87}.

We make frequent use of the following notation. Piecewise constant
temporal interpolants of 
$\{F^n_{h\tau}\}_{n=1}^N$, and $\{g^{n-1}_{h\tau}\}_{n=1}^N$ are
denoted by $F_{h\tau}$, and $g_{h\tau}$ respectively. 
With $\{u^n_{h\tau}\}_{n=0}^N$ taking values in $U_h$ and
$\{W^n_{\tau}\}_{n=0}^N$ as in \eqnref{:AB1}, $\uhat_{h\tau}$ and
$\What_{\tau}$ denote the piecewise linear interpolants respectively,
and $u_{h\tau}$ will denote the piecewise constant caglad interpolant;
see Figure \ref{fig:caglad}. Accordingly, we denote by $F_{h\tau}$ and 
$g_{h\tau}$  the piecewise constant caglad interpolants
of $\{F^n_{h\tau}\}_{n=1}^N$ and $\{g_{h\tau}^n\}_{n=0}^{N-1}$.
In Section \ref{sec:mglSoln} we establish the following theorem which
is the main result of this manuscript.

\begin{theorem} \label{thm:main} Let $T >0$, $(\Omega, \calF, \bbP)$
  be a probability space, $U$ be a separable reflexive Banach space,
  $H$ a Hilbert space, and $U \cembed H \cembed U'$ be compact, dense
  embeddings.  For every pair of numerical parameters $(\tau, h)$
  with $\tau = T/N \in \bbN$ let Assumptions \ref{ass:spde1} and
  \ref{ass:spde0} hold with parameter $p > 2$, and let $\{
  u_{h\tau}^n\}_{n=0}^N$ be a solution of \eqnref{:eulerF} with data
  $(u^0_{h\tau},F_{h\tau},g_{h\tau})$.  Assume for some $1 < q,r <
  \infty$ that
  \begin{enumerate}
  \item $\{\norm{u_{h\tau}}_{L^p(\Omega,\LrU)}\}_{h, \tau>0}$ is bounded.

  \item \label{hyp:Fhtau} $\{\norm{F_{h\tau}}_{L^p(\Omega,
      \LqpUp)}\}_{h,\tau>0}$ is bounded.

  \item \label{hyp:FuBounded} $\{F_{h\tau}(u_{h\tau})\}_{h,\tau > 0}$
    is bounded in $L^{p/2}(\Omega, L^1(0,T))$.

  \item $\{\norm{g_{h\tau}}_{L^p(\Omega, L^p[0,T;H])}\}_{h,\tau>0}$ 
    is bounded.

  \item The initial data $\{u^0_{h\tau}\}_{h,\tau>0}$ are bounded in
    $L^p(\Omega, H)$ and converge in $L^2(\Omega,H)$
    as $(h, \tau) \rightarrow (0,0)$.
  \end{enumerate}
  Then the following properties hold.
  \begin{enumerate}
  \item $\{\norm{u_{h\tau}}_{L^p(\Omega,\LinfH)}\}_{h,\tau>0}$ and
    $\{\norm{\uhat_{h\tau}}_{L^p(\Omega,
      C^{0,\theta}[0,T;U'])}\}_{h,\tau >0}$ with $0 < \theta <
    \min(1/2-1/p,1/q)$ are bounded.

   \item There exist a probability space $(\Omegat, \calFt, \bbPt)$
     and a random variable $(u,F,g,W)$ on $\Omegat$ with values in
     $$
     \bbX \equiv G[0,T;U'] \cap \LrU_{weak}
     \times \LqpUp_{weak}
     \times \LtwoH_{weak}
     \times C[0,T],
     $$
     and a subsequence $(h_k, \tau_k) \rightarrow (0,0)$ for which the
     laws of $\bigl\{(u_{h_k \tau_k}, F_{h_k \tau_k}, g_{h_k
       \tau_k}, \hat{W}_{\tau_k})\bigr\}_{k=1}^\infty$ converge to the
     law of $(u,F,g,W)$,
     $$
     \calL(u_{h_k\tau_k}, F_{h_k\tau_k}, g_{h_k\tau_k}, \What_{\tau_k})
     \ \Rightarrow \calL(u,F,g,W),
     $$
     with $\bbPt\big[u \in C[0,T;U'] \cap \LinfH \big] = 1$.
     Here $\LrU_{weak}$ and $\LqpUp_{weak}$ denote the
     spaces $\LrU$ and $\LqpUp$ endowed with the weak
     topology.

   \item If, in addition, the laws of $\{g_{h\tau}\}_{h,\tau > 0}$ are
     tight on $\LtwoH$ (which, for example, is the case when
     $g_{h\tau}$ converges in $L^2(\Omega,\LtwoH)$) then the laws converge
     on
     $$
     \bbX \equiv G[0,T;U'] \cap \LrU_{weak}
     \times \LqpUp_{weak}
     \times \LtwoH
     \times C[0,T],
     $$
     and there exists a filtration $\{\calFt(t)\}_{0 \leq t
       \leq T}$ satisfying the usual conditions for which $F$ is
     adapted, $g$ has a predictable representative in $L^2((0,T)
     \times \Omegat; U')$, and $W$ is a real-valued Wiener process,
     such that for all $0 \leq t \leq T$
     \begin{equation}\label{eqn-7}
       (u(t), v)_H = (u^0,v)_H 
       + \int_0^t ( F, v )\, ds
       + \int_0^t {( g, v )} \, dW,
       \qquad v \in U.
     \end{equation}

   \item \label{it:LsV} If additionally $V \embed {U'}$ is a separable
     reflexive Banach space and $\{u_{h\tau}\}_{h,\tau > 0}$ is
     bounded in $L^p(\Omega,L^s[0,T;V])$ for some $1 < s < \infty$,
     then the laws converge on
     $$
     \bbX \equiv G[0,T;U']
     \cap \LrU_{weak}
     \cap L^s[0,T;V]_{weak} 
     \times \LqpUp_{weak}
     \times \LtwoH
     \times C[0,T].
     $$
     If $U \cembed V$ is compact and $1 \leq \shat < s$, then the laws
     converge in
     $$
     \bbX \equiv G[0,T;U']
     \cap \LrU_{weak}
     \cap L^{\shat}[0,T;V]
     \times \LqpUp_{weak}
     \times \LtwoH
     \times C[0,T].
     $$
     
   \item If $U_0 \subset U$ is a subspace and if Assumption
     $\text{\ref{ass:spde1}}_\text{\ref{it:hDense}}$ is weakened to:
     \begin{itemize}
     \item[(\ref{it:hDense}')] For each $v \in U_0$, there exists a
       sequence $\{v_h\}_{h > 0} \subset U_h$ such that $v_h \rightarrow
       v$ for $h \rightarrow 0$.
     \end{itemize}
     the above still hold except that 
     $$
     (u(t), v)_H = ({u^0},v)_H 
     + \int_0^t (F(s), v ) \, ds
     + \int_0^t (g, v ) \, dW,
     \qquad v \in U_0.
     $$

   \item \label{it:Fsum} If $F_{h\tau} = \sum_{\ell=1}^{L}
     F^{(\ell)}_{h\tau}$ and each summand is bounded as in Hypothesis
     \ref{hyp:Fhtau}, then the above holds mutatis mutandis with
     $$
     \bbX \equiv G[0,T;U'] \cap \LrU_{weak}
     \times \LqpUp_{weak}^L
     \times \LtwoH
     \times C[0,T],
     $$
     and 
     $$
     (u(t), v)_H = (u^0,v)_H 
     + \int_0^t \big(\sum_{\ell=1}^L F^{(\ell)}(s), v \big) \, ds
     + \int_0^t {(g, v )} \, dW,
     \qquad v \in U.
     $$
     If $F^{(\ell)}_{h\tau}$ converges strongly in
     $L^p(\Omega,\LqpUp)$ for an index $1 \leq \ell \leq L$, then
     the laws converge when the corresponding factor space of
     $\LqpUp^L$ has the strong topology.
   \end{enumerate}
 \end{theorem}

This theorem can be viewed as an instance of the Lax--Richtmeyer
equivalence theorem or an infinite dimensional version of Donsker's
theorem with random walk in $U'$. The stability hypothesis of the
Lax--Richtmeyer theorem is identified with the bounds assumed upon
$\{(u_{h\tau}, F_{h\tau}, g_{h\tau})\}_{h, \tau > 0}$, and the
convergence is as stated.  The analysis of numerical schemes for each
of the examples introduced at the beginning of Section \ref{s-intro}
all included the following steps.
\begin{enumerate}
\item Bounds upon the approximate solution were first derived which always
  contained the hypotheses of Theorem \ref{thm:main} as a subset.  The
  implicit Euler scheme has been used ubiquitously in both the
  deterministic (PDE) and probabilistic (SODE) setting and
  bounds for stochastic PDE's follow upon integrating the ideas from
  these two disciplines.

\item The ideas introduced in Section \ref{sec:mglSoln} below for
  the proof of Theorem \ref{thm:main} were utilized to establish
  convergence to a weak martingale solution. In addition to 
  those introduced in the previous two sections, these include
  appropriate versions of the Kolmogorov-Centsov theorem to
  establish pathwise continuity, and the theorems of Prokhorov
  and Lions-Aubin  to establish compactness.

\item Compactness properties were developed in order to show that the
  limit $F$ took the form $F = f - A(u)$ (and $g \equiv \gamma(u)$ if $g_{h\tau}
  \equiv \gamma_{h\tau}(u_{h\tau})$).  This involves an interchange of limits;
  the numerical scheme will be ``consistent'' if $F_{h\tau} \equiv
  F(u_{h\tau}) \Rightarrow F(u)$.

  Frequently this step, which involves the spatial terms, was not
  well-delineated from the previous step which establishes convergence
  of the time stepping scheme. In the deterministic setting
  consistency is usually direct once the compactness is established;
  however, in the stochastic setting additional arguments are
  required.  In the next section we illustrate how convergence in law
  is used to establish consistency. Note that if additional bounds are
  available for a specific problem (as in Statement \ref{it:LsV} of the
  theorem) more test functions are available when the solutions
  converge in law, and these can be used to show consistency.
\end{enumerate}

\subsection{Consistency of the Spatial Terms} \label{se-consist}
Theorem \ref{thm:main} shows that the implicit Euler scheme
\eqnref{:eulerF} is consistent in the sense that (along a subsequence)
the laws of the discrete solution $(u_{h\tau}, F_{h\tau}, g_{h\tau},
\What_{\tau})$ converge to the laws of a limit $\bbPt =
\calL(u,F,g,W)$ satisfying (\ref{eqn-7}).  In order to recover a
solution of \eqnref{:Spde} it is necessary to show that $F = f - A(u)$
on the support of $\bbPt$, and, if the diffusion term depends upon the
solution, $g_{h\tau} = G(u_{h\tau})$, that $g = G(u)$.  Convergence in
law will be used to show this; recall that this mode of convergence
guarantees that
$$
\bbE[\phi(u_{h\tau}, F_{h\tau}, g_{h\tau}, \What_{\tau})]
\rightarrow \int_\bbX \phi(u,F,g, {W}) \, d \bbPt(u,F,g,W),
\quad \text{ for all } \quad \phi \in C_b(\bbX).
$$
A judicious selection of test functions in Lemma
\ref{lem:mapping_theorem} and Corollary \ref{cor:Law} is made to
establish consistency.

In all of the examples $F = f - A(u) = F^{(1)} + F^{(2)}$ is a sum, 
and Statement \ref{it:Fsum} in Theorem \ref{thm:main} shows that it is sufficient to
consider consistency of each term separately. Specifically, with
$$
\bbX \equiv G[0,T;U'] \cap \LrU_{weak}
\times \LqpUp_{weak}^2
\times \LtwoH
\times C[0,T],
$$
we have
$$
\calL(u_{h_k\tau_k}, (f_{h_k\tau_k}, A(u_{h\tau})), 
g_{h_k\tau_k}, \What_{\tau_k}) \Rightarrow \calL(u,(f,a),g,W) \equiv \bbPt, 
\qquad \text{ on } \bbX.
$$
Typically, the data $\{f_{h\tau}\}_{h,\tau > 0}$ are an approximation of a
specified random variable with law $\bbP_f$, and the discrete
approximations are constructed so that $\calL(f_{h\tau}) \Rightarrow
\bbP_f$ on $\LqpUp$. This will be the case if, for example,
$f_{h\tau}$ converges to a limit in $L^p(\Omega,\LqpUp)$.  It is then
immediate that $\calL(f) = \bbP_f$.

When ${\mathcal L}(u_{h\tau}) \Rightarrow {\mathcal L}(u)$ and ${\mathcal L}(A(u_{h\tau})) \Rightarrow {\mathcal L}(a)$ it is
necessary to show $a = A(u)$ on the support of $\bbPt$. The next
example shows that this is easily verified when $A$ is linear, and the
following example uses Corollary \ref{cor:Law} to establish this for a
nonlinear problem.

\begin{example}[linear equations] \label{eg:1} Let $A:U 
  \rightarrow U'$ be linear and continuous, $\normUp{A(u)} \leq C_a
  \normU{u}$. For $v \in \LtwoU$ fixed, the mapping
  $u \mapsto A(u)(v)$ is linear and continuous on $L^2[0,T;U']$, hence
  weakly continuous, so
  $$
  \phi(u, (f,a),g,W) 
  = \left| \int_0^T (a - A(u), v)  \, ds\right| \wedge 1
  $$
  is continuous on $\bbX = L^2[0,T;U]_{weak} \times
  L^2[0,T;U']_{weak}^2 \times \LtwoH \times C[0,T]$ and bounded.
  Consistency is then immediate,
  $$
  \bbEt \left[ \left| \int_0^T (a - A(u), v) \, ds\right| \wedge 1 \right]
  = \lim_{(h_k,\tau_k) \rightarrow (0,0)}
  \bbE\left[ \left| \int_0^T (A(u_{h_k \tau_k}) - A(u_{h_k\tau_k}), v)
      \, ds\right| \wedge 1 \right]
  = 0.
  $$
\end{example}

The next example considers the common situation where
the spatial operator is a compact perturbation of a linear operator.

\begin{example}[stochastic Navier Stokes equation] \label{eg:2}
  Solutions of the stochastic Navier-Stokes equation take values in
  the divergence free Sobolev space $U_0 = \{u \in \Honeo^3 \sst
  div(u) = 0\}$, with $D \subset \Re^3$ bounded and Lipschitz.
  However, numerical solutions are computed in the larger space $U =
  \Honeo^3$, and the spatial operator $A:U \rightarrow U'$ is
  \begin{eqnarray}
    (A(u),v) &=& 
    (1/2) \left((u.\nabla)u, v \vph \right) 
    - (1/2) \left(u, (u.\nabla) v \vph \right) 
    + \left(2 \mu D(u), \nabla v \vph\right) \label{eqn:Ans} \\
    &\equiv& \sum_{ij=1}^d \int_D 
    (1/2)\left(u_j \dbydp{u_i}{x_j} v_i - u_i u_j \dbydp{v_i}{x_j} \right)
    + \mu \left(\dbydp{u_i}{x_j} + \dbydp{u_j}{x_i} \right) \dbydp{v_i}{x_j},
    \qquad v \in U, \nonumber
  \end{eqnarray}
  where $D(u) = 1/2 (\nabla u + \nabla u^T)$.
  The last term on the right is bilinear and continuous and is
  accommodated as in the prior example. Appropriate exponents
  for this example are $r=2$, $q = 8$, $q'=8/7$.
  
  Let $\Ahat:U \rightarrow U'$ denote the operator
  $$
  (\Ahat(u),v) = (1/2) \left((u.\nabla)u, v \vph \right) 
    - (1/2) \left(u, (u.\nabla) v \vph \right).
  $$
  For $v \in L^4[0,T; U]$ fixed, we show that
  $$
  \phi(u,\Fhat) = \left|\int_0^T (\Fhat - \Ahat(u), v) \, ds\right|
  $$
  is sequentially continuous on
  $\bbXt \equiv G[0,T; U'] \cap L^2[0,T;U]_{weak} \times L^{8/7}[0,T;U']_{weak}$
  and has a finite moment of order $p$ when the solution has
  moments of order $2p$. Thus if $\calL(u_{h\tau},\Fhat_{h\tau})
  \Rightarrow \bbPhat$ on $\bbXhat$ with
  $\Fhat_{h\tau} \equiv \Ahat(u_{h\tau})$, then $\phi(
  u_{h\tau},\Fhat_{h\tau}) \equiv 0$ and from
  Corollary \ref{cor:Law} we conclude that
  $$
  \bbEhat \left[ \left| \int_0^T (\Fhat - \Ahat(u),v)\, ds\right| \right]
  = \lim_{(h,\tau) \rightarrow (0,0)}
  \bbE\left[ \left| \int_0^T 
      (\Fhat_{h\tau} - \Ahat(u_{h\tau}), v)  \, ds\right| \right]
  = 0,
  $$
  whence $\bbPhat[\Fhat = \Ahat(u)] = 1$.
  
  Since the mapping $\Fhat \in L^{8/7}[0,T; U']_{weak} \mapsto \int_0^T
  (\Fhat,v)$ is continuous it suffices to show that
  $$     
  u \in G[0,T; U'] \cap L^2[0,T; U]_{weak} \mapsto \int_0^T (\Ahat(u),v)\, ds
  $$
  is sequentially continuous. We sketch a proof of this; a
  detailed discussion of this operator is available in every text on
  the Navier-Stokes equations \cite{Ga94,GiRa79,Te77}.  

  A calculation using H\"older's inequality and the Sobolev embedding
  theorem, $U \embed L^6(D)$ in three dimensions, shows
  $$
  |(\Ahat(u_2)-\Ahat(u_1),v)|
  \leq C \norm{u_2-u_1}_{L^3(D)} 
  \left(\normU{u_1} + \normU{u_2} \vph\right) \normU{v}.
  $$
  Integration by parts for functions with homogeneous boundary data is
  used to obtain a bound without any derivatives on the difference $u_2 -
  u_1$.
  Using the interpolation estimate $\ltwo{u} \leq \normU{u}^{1/2} 
  \normUp{u}^{1/2}$ it follows that 
  $$
  \norm{u}_{L^3(D)} 
  \leq \ltwo{u}^{1/2} \norm{u}_{L^6(D)}^{1/2}
  \leq C \normUp{u}^{1/4} \normU{u}^{3/4},
  $$
  so
  $$
  |(\Ahat(u_2)-\Ahat(u_1),v)|
  \leq C \normUp{u_2-u_1}^{1/4} \left(\normU{u_1}^{7/4} 
    + \normU{u_2}^{7/4}\right) \normU{v}.
  $$
  In particular, setting $u_1=u$ and $u_2=0$ and integrating in time,
  it follows that
  $$
  \Big|\int_0^T (\Ahat(u),v) \, ds\Big|
  \leq \linfUp{u}^{1/4} \ltwoU{u}^{7/4} \norm{v}_{L^8[0,T;U]},
  $$
  so $\Ahat$ maps bounded sets in $G[0,T;U'] \cap \LtwoU$ to bounded
  sets in $L^{8/7}[0,T; U']$, and H\"older's inequality (with $s=8$,
  $s'=8/7$) shows
  $$
  \bbE\left[\Big|\int_0^T (\Ahat(u),v) \, ds\Big|^p \right]
  \leq \bbE\left[\linfUp{u}^{2p} \right]^{1/8}
  \bbE\left[\ltwoU{u}^{2p} \right]^{7/8} \norm{v}_{L^8[0,T;U]},
  $$
  so has moments of order $p > 1$ if the solution has moments greater
  than $2$.

  If $u_n \rightarrow u$ in $G[0,T;U'] \cap \LtwoU_{weak}$, then
  $\{u_n\}_{n=1}^\infty$ converges in $L^2[0,T;U']$ and is bounded in
  $\LtwoU$.  An application of H\"older's inequality then shows
  $$
  \int_0^T \left| \big(\Ahat(u_n) - \Ahat(u), v \big) \right| \, ds 
  \leq C \ltwoUp{u_n - u}^{1/4}
  \left(\ltwoU{u_n}^{7/4} + \ltwoU{u}^{7/4} \right) 
  \norm{v}_{C[0,T;U]} \rightarrow 0.
  $$
  Since the embedding $C[0,T;U'] \embed L^8[0,T,U]$ is dense it
  follows that $\Ahat(u_n) \weak \Ahat(u)$ in $L^{8/7}[0,T,U']$.

  The fully implicit approximation of the nonlinear term has
  $\Fhat_{h\tau} = \Ahat(u_{h\tau})$; semi--implicit schemes
  approximate the convective term with the operator
  $$
  (\Fhat^n_{h\tau},v) 
  = (1/2) \left((u_{h\tau}^{n-1}.\nabla)u_{h\tau}^n, v \vph \right) 
    - (1/2) \left(u_{h\tau}^n, (u_{h\tau}^{n-1}.\nabla) v \vph \right),
  $$
  so that each time step only requires the solution of a linear system.
  The choice preserves skew symmetry, $(\Fhat^n_{h\tau}, u^n_{h\tau}) = 0$, and
  using the embedding theorems as above shows
  $$
  |(\Fhat^n_{h\tau} - \Ahat(u^n_{h\tau}),v)|
  \leq C \norm{\uhat_{h\tau}}_{C^{0,\theta}[0,T;U']}^{1/4}
  \left(\normU{u_{h\tau}^{n-1}}^{7/4} + \normU{u_{h\tau}^n}^{7/4} \vph\right)
  \normU{v} \tau^{\theta/4},
  $$
  and 
  $$
  \bbE \left[
    \int_0^T |(\Fhat_{h\tau} - \Ahat(u_{h\tau}),v)|^p \, ds\right]
  \leq C \bbE\left[\norm{\uhat_{h\tau}}_{C^{0,\theta}[0,T;U']}^{2p}\right]^{1/8}
  \bbE\left[\ltwoU{u_{h\tau}}^{2p}\right]^{7/8}
  \norm{v}_{L^8[0,T;U]} \tau^{\theta/4}.
  $$
  Theorem \ref{thm:main} bounds the H\"older norm
  $L^p(\Omega,C^{0,\theta}[0,T;U'])$, so this term vanishes as $\tau
  \rightarrow 0$, and consistency of this approximation of the
  nonlinear term follows.
\end{example}

The stochastic equation \eqnref{:spde} is said to have ``additive
noise'' if the law of the function $g$ in equation \eqnref{:spde} is
specified. In this case $\{g_{h\tau}\}_{h,\tau > 0}$ is an
approximation of a specified random variable with law $\bbP_g$, and
the discrete approximations are constructed so that $\calL(g_{h\tau})
\Rightarrow \bbP_g$ on $\LqpUp$.  When the stochastic term depends upon
the solution, we write $g = \gamma(u)$, the equation is said to have
``multiplicative noise'', and it is necessary to verify that this
equation holds in the limit. The following elementary lemma is useful
in this context.

\begin{lemma} Let $U$ be a separable Banach space, $H$ be a Hilbert space
and $U \embed H$ be continuous embeddings.
If $\gamma:L^s[0,T;H] \cap \LrU_{weak} \rightarrow \LtwoH$
is sequentially continuous then $\gamma$ maps tight sequences to tight
sequences.
\end{lemma}

\begin{proof} Compact subsets of $L^2[0,T;H] \cap \LrU_{weak}$ are
metrizable, so $\gamma$ maps compact subsets to compact subsets. Thus
if $\epsilon > 0$, $K_\epsilon \subset L^s[0,T;H] \cap \LrU_{weak}$
is compact, and $\bbP[u_k \in K_\epsilon] \geq 1-\epsilon$, then
$$
\bbP[\gamma(u_k) \in \gamma(K_\epsilon)] 
= \bbP[u_k \in \gamma^{-1}(\gamma(K_\epsilon))] 
\geq \bbP[u_k \in K_\epsilon] 
\geq 1 - \epsilon.
$$
\end{proof}

\begin{example} Let $D \subset \Re^3$ be a bounded Lipschitz domain,
  $U \subset \Hone$ and $H = \Ltwo$.  Suppose that
  $ \gamma:(0,T) \times D \times
  \Re \rightarrow \Re$ is Caratheodory \cite{Sh97}; that is,
  $\gamma(t,x,u)$ is measurable in $(t,x)$ with $u$ fixed, and continuous in
  $u$ with $(t,x)$ fixed, and suppose that
  $$
  |\gamma(t,x,u)| \leq C |u|^{3/2} + k(t,x),
  \quad a.e. \, x \in (0,T) \times D, \quad u \in \Re,
  $$
  and $k \in L^2[0,T; L^2(D)]$. Letting
  $g(t,x,u) = \gamma(t,x, u(t,x))$ also denote the realization of $\gamma$ on
  the Lebesgue spaces. Under these assumptions $\gamma:L^3[0,T;L^3(D)]
  \rightarrow L^2[0,T;\Ltwo]$ is continuous.

  We show $\gamma:L^6[0,T;H] \cap L^2[0,T;U]_{weak} \rightarrow
  \LtwoH$ is sequentially continuous. For this purpose, recall that Statement
  \ref{it:LsV} of Theorem \ref{thm:main} shows that
  $\{u_{h\tau}\}_{h,\tau \geq 0}$ is tight in $L^s[0,T;H]$ for all $s
  > 1$.

  The first step is to note that the Sobolev embedding theorem
  shows $U \embed L^6(D)$, and since $1/3 = \theta/2 + (1-\theta)/6$
  when $\theta = 1/2$ it follows that
  $$
  \norm{u}_{L^3(D)} 
  \leq \ltwo{u}^{1/2} \norm{u}_{L^6(D)}^{1/2}
  \leq C \normH{u}^{1/2} \normU{u}^{1/2}.
  $$
  Integrating in time and H\"older's inequality (with $s=4$ and $s'=4/3$) shows
  $$
  \norm{u}_{L^3[0,T;L^3(D)]} 
  \leq C \norm{u}^{1/2}_{L^6[0,T;H]}
  \ltwoU{u}^{1/2}.
  $$
  Since weakly convergent sequences in $\LtwoU$ are bounded, and $\gamma$
  is continuous from $L^3[0,T;L^3(D)]$ to $\LtwoH$, sequential
  continuity of $\gamma:L^6[0,T;H] \cap L^2[0,T;U]_{weak} \rightarrow
  \LtwoH$ follows.

  Finally, note that
  \begin{eqnarray*}
  \bbE[\ltwoH{\gamma(u)}^p]
  &\leq& C \bbE\left[ \norm{u}^{p/2}_{L^6[0,T;H]} \ltwoU{u}^{p/2} 
    + \norm{k}_{L^2[0,T;\Ltwo]} \vph\right] \\
  &\leq& C \left(
    \bbE[\norm{u}^p_{L^6[0,T;H]}]^{1/2} \, \bbE[\ltwoU{u}^p]^{1/2} + 1 
    \vph\right),
  \end{eqnarray*}
  so $\gamma(u)$ inherits moment bounds from $u$. From Corollary \ref{cor:Law}
  it follows that if $\calL(u_{h\tau}) \Rightarrow \calL(u)$ in 
  $L^6[0,T;H] \cap L^2[0,T;U]_{weak}$ then $\calL(\gamma(u_{h\tau}))
  \Rightarrow \calL(\gamma(u))$ on $\LtwoH$.
\end{example}

\subsection{Computational Model}
Strong solutions are never realized in a computational context since
this would require a filtered probability space to be input as part of
the problem specification.  Instead a random number generator is
seeded and then ittereated to generate a sequence
$\{b_p(\omega)\}_{p=1}^\infty$ which exhibit the satistics of a
sequence of real valued i.i.d. variables $\{b_p\}_{p=1}^\infty$ sampled
at a point $\omega \in \Omega$ determined by the seed.  Typcially
their law is the uniform (Lebesgue) measure on $(0,1)$. Given laws of
the data, $\calL(f,g,W)$, the random numbers then used to engineer
samples $(f^n_{h\tau}(\omega), g^n_{h\tau}(\omega),
\xi^n_\tau(\omega))$ of random variables with laws $\calL(f_{h\tau},
g_{h\tau}, \What_{h\tau}) \Rightarrow \calL(f,g,W)$.

\begin{example} If $\calL(b_n)$ is Lebesgue measure on $(0,1)$ and
$\xi^n_\tau(\omega) = \sqrt{12 \tau} (b_n(\omega) - 1/2)$ then
$$
\bbE[\xi_\tau^n] = 0, \qquad
\bbE[(\xi_\tau^n)^2] = \tau,
\qquad \text{ and } \qquad
\bbE[|\xi_\tau^n|^p] = \mbox{$\frac{(3 \tau)^{p/2}}{(p+1)}$}.
$$
It follows that $\{\xi^n\}_{n=1}^N$ will satisfy Assumption \ref{ass:spde0}.
In addition, if 
$$
f^n_{h\tau}(x,\omega) = \Phi^n_{h\tau}(x, b_1(\omega), \ldots, b_n(\omega))
\qquad \text{ with } \qquad
\Phi^n_{h\tau} \in C(D \times \Re^n; U_h),
$$
then $f_{h\tau}$ will be adapted to $\calF^n_{h\tau} \equiv
\sigma(b_1, \dots, b_n)$.
\end{example}

If the law $\calL(u_{h\tau})$ of a solution of the implicit Euler
scheme \eqnref{:Spdeh} depends only upon the laws of the data
$\calL(f_{h\tau},g_{h\tau},W_{h\tau})$ (and the law of the initial data if
not deterministic), then for $(h,\tau)$ fixed, solutions
$\{(f^{(p)}_{h\tau}(\omega), g^{(p)}_{h\tau}(\omega),
W^{(p)}_{h\tau}(\omega))\}_{p=1}^\infty$ of the implicit Euler scheme
computed using distinct subsets of the random numbers will be i.i.d.  In
this context Monte-Carlo quadrature can be used to compute the
statistics of a solution guaranteed by Theorem \ref{thm:main}. If
$\bbPt$ is the measure and $\{(h_k, \tau_k)\}_{k=1}^\infty$
is the subsequence guaranteed by Theorem \ref{thm:main}, then
$$
\bbEt[\phi(u)] =
\lim_{h_k,\tau_k \rightarrow 0} \bbE[\phi(u_{h\tau})]
=  \lim_{h_k,\tau_k \rightarrow 0} \Big(
\lim_{P \rightarrow \infty} \frac{1}{P} 
\sum_{p=1}^P \phi(u^{(p)}_{h_k\tau_k}(\omega)) \Big),
\qquad \text{ almost surely,}
$$
for any function $\phi:G[0,T;U'] \times
L^r[0,T;U]_{weak} \rightarrow \Re$ satisfying the hypotheses of
Lemma \ref{lem:mapping_theorem}. 

When the law $\calL(u)$ of the solution to \eqnref{:spde} is uniquely
determined by the law $\calL(f,g,W)$ of the data, it is unnecessary to
pass to a subsequence provided $\calL(f_{h\tau}, g_{h\tau},
\What_{h\tau}) \Rightarrow \calL(f,g,W)$.  This is typically achieved
by constructing $(f^n_{h\tau}(\omega),g^n_{h\tau}(\omega))$ to be
projections or interpolants of specified functions onto the discrete
spaces (e.g. as in equation \eqnref{:AA1}) to give a Cauchy sequence
in $L^p(\Omega;\LqpUp] \times L^p(\Omega; L^p[0,T;H])$. In the
examples below it is assumed that $\{(f_{h\tau},g_{h\tau})\}_{h\tau >
  0}$ converges in $L^p(\Omega;\LqpUp) \times L^p(\Omega; L^p[0,T;H])$
whenever we wish to assert uniqueness.

\section{The stochastic heat equation}
In this section, we construct a weak martingale solution of the
stochastic heat equation.  While (stochastically) strong solutions
exist for this problem \cite{PrRo07,DPZ1}, we choose this simplified
framework to eliminate many technical issues that would otherwise
obfuscate the essential structure; more general nonlinear SPDE's are
presented in Section~\ref{sec:examples}.

Let $D \subset \Re^d$ be a bounded Lipschitz domain, and $[0,T]$ be a
time interval. Adopting the notation commonly used in stochastic
analysis, the heat equation with a stochastic source takes the form:
find a filtered probability space $(\Omega, \calF, \{\calF(t)\}_{0 \leq t
  \leq T}, \bbP)$ satisfying the usual conditions, an adapted process
$u:[0,T] \times D \times \Omega \rightarrow \Re$, and a standard Wiener
process $W:[0,T] \times \Omega\rightarrow \Re$ such that
\begin{equation} \label{eqn:heat}
d u - \Delta u \, dt = f \, dt + g \, d W
\qquad
u|_{t=0} = u^0,
\qquad
u|_{\partial D} = 0,
\end{equation}
with data $f, g:[0,T] \times D \times \Omega \rightarrow \Re$ that are
adapted to $\{ \calF(t)\}_{0 \leq t \leq T}$ and $u^0$ measurable on
$\calF(0)$.  Multiplying the heat equation by a test function $v$
vanishing on the boundary and integrating by parts shows
\begin{equation} \label{eqn:heatWeak}
\int_D u(t) v\, dx + \int_0^t \! \int_D \nabla u . \nabla v\, dxds
= \int_D u^0 v\, dx +
\int_0^t \! f  v \, dxds
+ \int_0^t \! \left( \int_D g v \, dx\right) \,dW,
\quad 0 \leq t \leq T.
\end{equation}
Setting $H = \Ltwo$, $U = \Honeo$ and defining $a: U \times U
\rightarrow \Re$ by
$$
a(u,v) = \int_D \nabla u . \nabla v\, dx,
$$
it follows that a solution of the heat equation with stochastic source
is an instance of the stochastic evolution equation exhibited in
equation \eqnref{:Spde}.  Convergence of the discrete scheme
\eqnref{:Spdeh} with these operators will be established under the
following hypotheses.

\begin{assumption} \label{ass:heat}
Let
  $U \embed H \embed U'$ be a dense embedding of separable Hilbert
  spaces, and the operators and data for equation \eqnref{:Spde} satisfy
  \begin{enumerate}
  \item $a:U \times U \rightarrow \Re$ is bilinear, continuous, and
    coercive. Specifically, there exist constants $c_a$, $C_a > 0$
    such that
    $$
    |a(u,v)| \leq C_a \normU{u} \normU{v},
    \quad \text{ and } \quad
    a(u,u) \geq c_a \normU{u}^2, 
    \qquad u,v \in U.
    $$
    
  \item For every $h>0$, $U_h$ is a finite
    dimensional subspace of $U$, and $\{t^n\}_{n=0}^N$ is a uniform
    partition of $[0,T]$ with time-step $\tau = T/N$.

  \item For each pair of parameters $(h,\tau)$, $\calF^0$ is generated
    by $u^0$ and $\{\calF^n\}_{n=1}^N$ is the discrete filtration with
    $\calF^n = \sigma\left( \{(u^m_{h\tau}, f^m_{h\tau}, g^m_{h\tau},
      \xi^m_{\tau}) \}_{m=0}^n \vph\right)$.
  \end{enumerate}
\end{assumption}

Granted Assumptions \ref{ass:heat} and \ref{ass:spde0} with $p \geq 2$,
the existence to the discrete scheme \eqnref{:Spdeh} is
direct; fix $\omega \in \Omega$ and write equation \eqnref{:Spdeh}
as $u^n_{h\tau}(\omega) \in U_h$,
$$
\left(u^n_{h\tau}(\omega), v_h\right)_H 
+ \tau a\left(u^n_{h\tau}(\omega), v_h\right) 
= \left(u^{n-1}_{h\tau}(\omega), v_h\right)_H 
+ \tau \left( f^n_{h\tau}(\omega), v_h \right) 
+ \left(g^{n-1}_{h\tau}(\omega),v_h\right) \xi^n_{\tau}(\omega),
\quad
v_h \in U_h.
$$
Upon selecting a basis for $U_h$ this becomes a system of linear
equations, $\bbA \bfu(\omega) = \bfb (\omega)$, with as many equations
as unknowns; moreover,
$$
{\bf v}^\top \bbA {\bf v} 
= \bigl(v_h, v_h \bigr)_H + \tau a\bigl(v_h, v_h\bigr) 
\geq \normH{v_h}^2 + \tau c_a \normU{v_h}^2, \qquad v_h \in U_h,
$$
so $\bbA$ is nonsingular and $u^n_{h\tau}$ is a continuous function of
the data $(u^{n-1}_{h\tau}, f^n_{h\tau}, g^{n-1}_{h\tau},
\xi^n_{\tau})$.  Since measurability of random variables is always
with respect to the Borel $\sigma$-algebra on the target space,
continuity of the solution operator guarantees that $u^n_{h\tau}$ is
$\calF^n$--measurable whence the sequence $\{u^n_{h\tau}\}_{n=0}^N$ is
adapted to $\{\calF^n\}_{n=0}^N$.

\subsection{Bounds} \label{bounds-1}
We begin by recalling bounds satisfied by the deterministic equation
$$
u \in U, \qquad
(\partial_t u,v)_H + a(u,v) = (f,v),
\qquad v \in U,
$$
with the bilinear function satisfying Assumption \ref{ass:heat}.
The fundamental estimate is found upon selecting $v=u$ to get
$$
(1/2) \dbyd{}{t} \normH{u}^2 + c_a \normU{u}^2 
\leq ( f,u )
\leq \normUp{f} \normU{u}.
$$
Integration in time then shows
$$
\linfH{u}^2 + c_a\ltwoU{u}^2
\leq \normU{u(0)}^2 + (1/2c_a) \ltwoUp{f}^2.
$$
The analogous statement for the discrete scheme
(\ref{euler-intro}) is obtained upon selecting the test function
$v_h = u^n_{h\tau}$, and the corresponding estimate is
$$
\max_{1 \leq n \leq N} \normH{u_{h\tau}^n}^2 
+ \sum_{m=1}^N \normH{u_{h\tau}^m - u_{h\tau}^{m-1}}^2
+ c_a\sum_{m=1}^N \tau \normU{u_{h\tau}^m}^2 
\leq C \left( 
\normH{u_{h\tau}^0}^2
+ \sum_{m=1}^N \tau \normUp{f^m_\tau}^2
\right).
$$
The second term on the left is an additional dissipative term
inherent to the implicit Euler scheme which arises when completing
the square of the approximate time derivative,
\begin{equation} \label{eqn:square}
(u - v, u)_H
= (1/2) \normH{u}^2 + (1/2) \normH{u-v}^2 - (1/2) \normH{v}^2.
\end{equation}
Consider next the discrete scheme \eqnref{:Spdeh} with bilinear form
satisfying Assumption \ref{ass:heat}.  To bound its solution,
independence of the increments and the dissipative term in the Euler
scheme are used in an essential fashion. With $\omega \in \Omega$
fixed, selecting the test function in equation \eqnref{:Spdeh} to be
$v = u^n_{h\tau}(\omega)$
gives
\begin{equation} \label{eqn:energyEst}
\normH{u^n_{h\tau}}^2 
+ \normH{u_{h\tau}^n-u_{h\tau}^{n-1}}^2
+ 2 c_a \tau \normU{u_{h\tau}^n}^2
\leq \normH{u_{h\tau}^{n-1}}^2 
+ 2 \tau ( f^n_{h\tau} , u_{h\tau}^n )
+ 2 (g^{n-1}_{h\tau}, u_{h\tau}^n) \xi^n_{\tau}.
\end{equation}
To bound the last term 
properties of the stochastic increments from Assumption~\ref{ass:spde0}
are utilized.
Writing this term as
$$
(g^{n-1}_{h\tau}, u_{h\tau}^n) \xi^n_{\tau}
= (g^{n-1}_{h\tau}, u_{h\tau}^n-u_{h\tau}^{n-1}) \xi^n_{\tau}
+ (g^{n-1}_{h\tau}, u_{h\tau}^{n-1}) \xi^n_\tau,
$$
and taking the expected value we have
\begin{itemize}
\item $(g_{h\tau}^{n-1}, u_{h\tau}^{n-1})_H$ is $\calF^{n-1}$-measurable, so is independent of $\xi^n_\tau$, and since the average of
  $\xi^n_\tau$ vanishes it follows that
  $$
  \bbE \left[ (g_{h\tau}^{n-1}, u_{h\tau}^{n-1}) \xi^n_\tau \right]
  = \bbE [ (g_{h\tau}^{n-1}, u_{h\tau}^{n-1})] \, \bbE[\xi^n_\tau]
  = 0.
  $$

\item $\normH{g_{h\tau}^{n-1}}$ and $|\xi^n_\tau|$ are also independent, so
  an application of the Cauchy-Schwarz inequality gives
  \begin{eqnarray*}
    \bbE \left[ (g_{h\tau}^{n-1}, u_{h\tau}^n-u_{h\tau}^{n-1}) 
      \xi^n_\tau \right]
    &\leq& \left( \bbE[\normH{u_{h\tau}^n-u_{h\tau}^{n-1}}^2] \right)^{1/2}
    \left( \bbE[\normH{g_{h\tau}^{n-1}}^2 |\xi^n_\tau|^2] \right)^{1/2} \\
    &=& \left( \bbE[\normH{u_{h\tau}^n-u_{h\tau}^{n-1}}^2] \right)^{1/2}
    \left( \bbE[\normH{g_{h\tau}^{n-1}}^2] \bbE[|\xi^n_\tau|^2] \right)^{1/2} \\
    &=& \left( \bbE[\normH{u_{h\tau}^n-u_{h\tau}^{n-1}}^2] \right)^{1/2}
    \left( \tau \bbE[\normH{g_{h\tau}^{n-1}}^2] \right)^{1/2}.
  \end{eqnarray*}
\end{itemize}
Taking the expected value of both sides of equation \eqnref{:energyEst},
this bound is used to estimate the stochastic term,
\begin{multline} \label{eqn:est0}
\norm{u_{h\tau}}^2_{L^\infty[0,T; L^2(\Omega,H)]} 
+ \norm{u_{h\tau}}^2_{L^2[0,T; L^2(\Omega,U)]} \\
\leq C
\left(
  \norm{u_{h\tau}^0}^2_{L^2(\Omega,H)} 
  + \norm{f_{h\tau}}^2_{L^2[0,T;L^2(\Omega,U')]}
  + \norm{g_{h\tau}}^2_{L^2[0,T; L^2(\Omega,H)]} 
\right).
\end{multline}

This estimate bounds $u_{h\tau}$ in the Bochner space $L^\infty[0,T;
L^2(\Omega,H)]$; however, we also wish to identify $u_{h\tau}$ as a
random variable taking values in the Bochner space $\LinfH$. For any
Banach space $U$, the canonical correspondences
$$
L^2[0,T; L^2(\Omega, U)] 
\simeq L^2((0,T) \times \Omega, U) 
\simeq L^2(\Omega, L^2[0,T;U]),
$$
allow functions in these spaces to be identified as a random variable
with values in $L^2[0,T;U]$. In general it is not possible to identify
$L^\infty[0,T; L^2(\Omega,H)]$ with $L^2(\Omega, L^\infty[0,T;U])$;
however, the BDG inequality shows that the
norms on these two spaces are equivalent on the subspace of
martingales. The following lemma uses the property that the stochastic
term in \eqnref{:Spdeh} is an martingale to bound the solution in
$L^2(\Omega, L^\infty[0,T;H])$.

\begin{lemma} \label{lem:bound}
  Let Assumptions \ref{ass:spde0} and \ref{ass:heat} with $p \geq 2$
  hold and $u_{h\tau}$ be a solution of the implicit Euler scheme
  \eqnref{:Spdeh} with 
  initial condition $u_{h\tau}^0 \in L^p(\Omega,H)$,
  and data $f_{h\tau} \in L^p(\Omega, \LtwoUp)$, and $g_{h\tau} \in
  L^p(\Omega, \LpH)$. Then there exists a constant $C = C(p) > 0$ such
  that
  \begin{eqnarray}
    \lefteqn{
      \norm{u_{h\tau}}_{L^p(\Omega, \LinfH)}
      + \norm{u_{h\tau}}_{L^p(\Omega, \LtwoU)}
      + \bbE \left[ \left( \sum_{m=1}^N 
          \normH{u_{h\tau}^m - u_{h\tau}^{m-1}}^2 \right)^{p/2} \right]^{1/p}
      \nonumber } \\
    &\leq& 
    C \left(
      \norm{u_{h\tau}^0}_{L^p(\Omega,H)} 
      + \norm{f_\tau}_{L^p(\Omega, \LtwoUp)}
      + T^{1/2-1/p} \norm{g_\tau}_{L^p(\Omega, \LpH)} \right). 
    \label{eqn:bound}
  \end{eqnarray}
\end{lemma}

\begin{proof}
Sum equation \eqnref{:energyEst} to obtain
\begin{eqnarray*}
  \lefteqn{ \normH{u_{h\tau}^n}^2 
    + \sum_{m=1}^n \normH{u_{h\tau}^m - u_{h\tau}^{m-1}}^2 
    + 2c_a \sum_{m=1}^n \tau \normU{u_{h\tau}^m}^2 } \\
  &\leq& 
  \normH{u_{h\tau}^0}^2 
  + 2 \sum_{m=1}^n \tau ( f_{h\tau}^{m}, u_{h\tau}^m )
    + 2 \sum_{m=1}^n (g_{h\tau}^{m-1}, u_{h\tau}^m)_H \xi^m_\tau  \\
  &\leq& 
  \normH{u_{h\tau}^0}^2 
  + 2 \sum_{m=1}^n \tau \normUp{f_{h\tau}^{m}} \normU{u_{h\tau}^m}
    +  2 \sum_{m=1}^n \normH{g_{h\tau}^{m-1}} 
    \normH{u_{h\tau}^m - u_{h\tau}^{m-1}} |\xi^m_\tau| \\
    &&
  + 2 \sum_{m=1}^n(g_{h\tau}^{m-1}, u_{h\tau}^{m-1})_H \xi_\tau^m,
\end{eqnarray*}
and use the Cauchy-Schwarz and Young inequalities to get
\begin{eqnarray*}
  \lefteqn{ \normH{u_{h\tau}^n}^2 
    + \sum_{m=1}^n \normH{u_{h\tau}^m - u_{h\tau}^{m-1}}^2 
    + \sum_{m=1}^n \tau \normU{u_{h\tau}^m}^2 } \\
  &\leq& 
  C \Bigl(\normH{u_{h\tau}^0}^2 
  +  \sum_{m=1}^n \tau \normUp{f_{h\tau}^{m}}^2
  +  \sum_{m=1}^n \normH{g_{h\tau}^{m-1}}^2 |\xi^m_\tau|^2
  + \Bigl\vert \sum_{m=1}^n(g_{h\tau}^{m-1}, 
  u_{h\tau}^{m-1})_H \xi^m_\tau \Bigr\vert \Bigr).
\end{eqnarray*}
Raising each side to the power $p/2$ and using Assumption
\ref{ass:spde0}$_3$ shows
\begin{eqnarray*}
  \lefteqn{ \normH{u_{h\tau}^n}^p 
+ \left( \sum_{m=1}^n \normH{u_{h\tau}^m - u_{h\tau}^{m-1}}^2 \right)^{p/2}
+ \left( \sum_{m=1}^n \tau \normU{u_{h\tau}^m}^2 \right)^{p/2} } \\
&\leq&  C \Bigl(\normH{u_{h\tau}^0}^p   
+ \ltwoUp{f_{h\tau}}^p
+ \left( \sum_{m=1}^n \normH{g^{m-1}_{h\tau}}^2 |\xi^m_\tau|^2 \right)^{p/2}
+ |\sum_{m=1}^n (g_{h\tau}^{m-1}, u_{h\tau}^{m-1})_H \xi^m_\tau |^{p/2}\Bigr) \\
&\leq&  C \Bigl(\normH{u_{h\tau}^0}^p   
+ \ltwoUp{f_{h\tau}}^p
+ n^{p/2-1} \sum_{m=1}^n \normH{g_\tau^{m-1}}^p |\xi^m_\tau|^p 
+ |\sum_{m=1}^n (g_{h\tau}^{m-1}, u_{h\tau}^{m-1})_H \xi_\tau^m |^{p/2}\Bigr).
\end{eqnarray*}
Taking the maximum  over $1 \leq n = n(\omega) \leq N$ 
and using the property that 
$$
\bbE\left[n^{p/2-1} \normH{g_{h\tau}^{m-1}}^p |\xi_\tau^m|^p \vph \right] 
\leq C(p) N^{p/2-1} \bbE\left[\normH{g_{h\tau}^{m-1}}^p \vph\right] \tau^{p/2}
\leq C(p) T^{p/2-1} \bbE\left[\tau \normH{g_{h\tau}^{m-1}}^p \vph\right],
$$ 
shows
\small
\begin{eqnarray}
  \lefteqn{ 
    \bbE \left[ \max_{1 \leq n \leq N} \normH{u_{h\tau}^n}^p 
      + \left( \sum_{m=1}^N \normH{u_{h\tau}^m - u_{h\tau}^{m-1}}^2 \right)^{p/2}
      + \ltwoU{u_{h\tau}}^p \right]}
    \label{eqn:est1}  \\
  &\leq&  C 
  \Bigl(\norm{u_{h\tau}^0}_{L^p(\Omega, H)}^p   
  + \norm{f_{h\tau}}_{L^p(\Omega, \LtwoUp)}^p
  + T^{p/2-1} \norm{g_{h\tau}}_{L^p(\Omega, L^p[0,T;H])}^p \nonumber \\
  && + \bbE \left[ \max_{1 \leq n \leq N} 
    |\sum_{m=1}^n (g_{h\tau}^{m-1}, u_{h\tau}^{m-1})_H \xi^m_{\tau}|^{p/2} \right]\Bigr). 
  \nonumber
\end{eqnarray}
\normalsize

The last term is a discrete Ito integral (c.f. equation
\eqnref{:discreteIto}),
$$
X^n_{\tau} = \sum_{m=1}^n (g^{m-1}_{h\tau}, u_{h\tau}^{m-1})_H \xi^m_\tau,
$$
and is bounded using the discrete BDG inequality
(Theorem \ref{thm:BDG}) and Lemma \ref{lem:mglVariation}.
With $\epsilon >0$ to be selected below,
\begin{eqnarray*}
  \bbE\left[ \max_{0 \leq n \leq N} |X^n_{\tau}|^{p/2} \right]
  &\leq& C
  \sum_{m=1}^N \tau 
  \norm{(g_{h\tau}^{m-1}, u_{h\tau}^{m-1})_H}_{L^{p/2}(\Omega)}^{p/2}
  T^{p/4-1} \\
  &\leq& C
  \sum_{m=1}^N \tau 
  \norm{g_{h\tau}^{m-1}}_{L^p(\Omega,H)}^{p/2} 
  \norm{u_{h\tau}^{m-1}}_{L^p(\Omega,H)}^{p/2}
  T^{p/4-1} \\
  &\leq& C
  \left(\max_{0 \leq m \leq N-1} 
    \norm{u_{h\tau}^m}_{L^p(\Omega,H)}^{p/2}\right)
  \sum_{m=1}^N \tau 
  \norm{g_{h\tau}^{m-1}}_{L^p(\Omega,H)}^{p/2} 
  T^{p/4-1} \\
  &\leq&
  \epsilon \left(\max_{0 \leq m \leq N-1} 
    \norm{u_{h\tau}^m}_{L^p(\Omega,H)}^p\right)
  + (2 C^2/\epsilon)
  \left(\sum_{m=1}^N \tau \norm{g_{h\tau}^{m-1}}_{L^p(\Omega,H)}^{p/2} \right)^2
  T^{p/2-2} \\
  &\leq& 
  \epsilon \norm{u_{h\tau}^0}_{L^p(\Omega,H)}^p
  + \epsilon \max_{1 \leq m \leq N} \norm{u_{h\tau}^m}_{L^p(\Omega,H)}^p 
  +  (2 C^2/\epsilon) T^{p/2-1} \norm{g_{h\tau}}_{L^p(\Omega, \LpH)}^p.
\end{eqnarray*}
The proof now follows since the middle term, with an appropriate
choice of $\epsilon > 0$, can be absorbed into
the left-hand side of equation \eqnref{:est1}.
\end{proof}

\subsection{Passage to the Limit}\label{passage-1}
Setting $( F, v ) = (f, v) - a(u, v)$, the weak statement of the
stochastic heat equation \eqnref{:heatWeak} is an instance of the
abstract problem (\ref{eqn-7}), and its discretization is of the form
(\ref{eqn:eulerF}), which is considered in Theorem \ref{thm:main}. The
bounds in Lemma \ref{lem:bound} are sufficient to verify the
hypotheses of Theorem \ref{thm:main}, and convergence of the discrete
scheme to a weak martingale solution of \eqnref{:heat} follows.

\begin{theorem}
  Let $U$ be a separable reflexive Banach space, $H$ a Hilbert space,
  $U \cembed H$ be a compact, dense embedding, and let $(\Omega,
  \calF, \bbP)$ be a probability space.  Let the operators of the
  abstract difference scheme and data satisfy Assumptions
  \ref{ass:heat} and \ref{ass:spde1} respectively, and let the
  stochastic increments satisfy Assumptions \ref{ass:spde0} with $p
  \in (2, \infty)$.  Denote the discrete Wiener process with increments
  $\{\xi^m_{\tau}\}_{m=1}^N$ by $\What^n_{\tau}$, and let
  $\{u_{h\tau}\}_{h,\tau >0}$ be a sequence of solutions of the
  corresponding implicit Euler scheme \eqnref{:Spdeh} with data
  satisfying:
  \begin{enumerate}
  \item $\{u^0_{h\tau}\}$ is bounded in $L^p(\Omega,H)$ and
    converges to a limit $u^0$ in $L^2(\Omega, H)$
    as $h \rightarrow 0$.

  \item $\{f_{h\tau}\}$ is bounded in $L^p(\Omega, \LtwoUp)$ and
    converges in $L^2(\Omega, \LtwoUp)$
    as $\tau,h \rightarrow 0$.

  \item $\{g_{h\tau}\}$ is bounded in $L^p(\Omega, \LpH)$ and
    converges in $L^2(\Omega, \LtwoH)$ as $\tau,h \rightarrow 0$.
  \end{enumerate}
  Let 
  $$
  \bbX = G[0,T;U'] \cap L^2[0,T;U]_{weak} \times \LtwoUp_{weak} \times
  \LtwoH \times C[0,T]\, .
  $$
  Then there exist a probability space $(\Omegat, \calFt, \bbPt)$ and
  a random variable $(u,F,g,W)$ on $\Omegat$ with values in $(\bbX,
  \calB(\bbX))$
  for which the laws of $\bigl\{ (u_{h \tau}, (f_{h
    \tau},A(u_{h \tau})), g_{h \tau},
  \hat{W}_{\tau})\bigr\}_{h,\tau > 0}$ converge to the law of
  $(u,(f,A(u)),g,W)$,
  $$
  \calL(u_{h\tau}, (f_{h\tau},A(u_{h \tau})), 
  g_{h\tau}, \What_{\tau})
  \ \Rightarrow \ \calL(u,(f,A(u)),g,W).
  $$
  In addition, there exists a filtration $\{\calFt(t)\}_{0 \leq t \leq
    T}$ satisfying the usual conditions for which $(u,f,g,W)$ is adapted
  and $W$ is a real-valued Wiener process for which
  $$
  (u(t), v)_H   + \int_0^t a(u,v) \, ds
  = (u^0,v)_H 
  + \int_0^t ( f, v ) \, ds
  + \int_0^t {( g, v )} \, dW,
  \qquad v \in U.
  $$
\end{theorem}

\begin{proof}
  Under the assumptions of the theorem solutions of the implicit Euler
  scheme satisfy the bounds stated in Lemma \ref{lem:bound}; in
  particular, $\{ u_{h\tau}\}_{h,\tau >0}$ is bounded in $L^p(\Omega,
  L^2[0,T;U])$.  With
  $$
  F_{h\tau}^n(v_h) = (f^n_{h\tau} ,v_h) - a(u^n_{h\tau}, v_h),
  $$
  it is immediate that $F^n_{h \tau}$ is $\calF^n$-measurable, and
  since $a:U \times U \rightarrow \Re$ is bilinear and continuous,
  $\{F_{h\tau}\}_{h,\tau > 0}$ is bounded in $L^p(\Omega,
  L^2[0,T;U'])$, and it follows from the Cauchy-Schwarz inequality that
  $\{F_{h\tau}(u_{h\tau})\}_{h\tau > 0}$ is bounded in
  $L^{p/2}(\Omega, L^1(0,T))$. This establishes the hypotheses of
  Theorem \ref{thm:main} (with $r=q=2$) which guarantees the existence
  of a filtered probability space $(\Omegat, \calFt, \{\calFt(t)\}_{0
    \leq t \leq T}, \tilde{\bbP})$, and a subsequence $(h_k, \tau_k)
  \rightarrow (0,0)$ and a limit $(u,f,g,W)$ for which the laws of
  $(u_{h_k\tau_k}, F_{h_k\tau_k}, g_{h_k\tau_k}, \What_{\tau_k})$
  convergence as asserted in the theorem and
  $$
  (u(t), v)_H = (u^0,v)_H 
  + \int_0^t (F, v ) \, ds
  + \int_0^t (g, v )_H \, dW,
  \quad v \in U.
  $$
  To verify that $(F,v) = (f,v) - a(u,v)$ for $v \in U$ note that the
  mapping $u \mapsto f - a(u,.)$ is affine so is continuous from
  $\LtwoU$ to $\LtwoUp$ with both the weak and strong topologies.
  This is the setting of Example \ref{eg:1} where it was shown that
  $F$ takes the required form. Finally, $A$ satisfies Assumption
  \ref{ass:A} since solutions of the deterministic heat equation
  are unique so  Theorem \ref{thm:uniqueLaw} is applicable. It follows that
  $\calL(u)$ is uniquely determined by $\calL(f,g,W)$; in particular,
  passing to a subsequence was unnecessary. 
\end{proof}

\section{Construction of a martingale solution} \label{sec:mglSoln}

This section is devoted to the proof of Theorem \ref{thm:main}.
Throughout $U$ will denote a Banach space densely embedded in a
Hilbert space $H$ so that $U \embed H \embed U'$, and $U_h$
will denote a (finite dimensional) subspace of $U$, and $\tau = T/N$ the time step for
the implicit Euler scheme \eqnref{:eulerF}. 

\subsection{Bounds and Pathwise Continuity}
The following lemma is essentially a restatement of Lemma
\ref{lem:bound} adapted to the current setting where bounds upon the
solution are assumed.

\begin{lemma} \label{lem:boundF} Let $1 \leq q \leq \infty$, and 
  $\{u^n_{h\tau}\}_{n=1}^N$ be a $U_h$-valued solution of the implicit
  Euler scheme \eqnref{:eulerF} with increments and data satisfying
  Assumptions \ref{ass:spde0} with $p \geq 2$ and \ref{ass:spde1}
  respectively.  If $u^0_{h\tau} \in L^p(\Omega,U_h)$, $F_{h\tau} \in
  L^p(\Omega, L^{q'}[0,T; U_h'])$, and $g_{h\tau} \in L^p(\Omega,
  \LpH)$, then there exists a constant $C = C(p) > 0$ such that the
  piecewise constant interpolant $u_{h\tau}$ satisfies
  \begin{eqnarray*}
    \lefteqn{
      \norm{u_{h\tau}}_{L^p(\Omega, \LinfH)}
      + \bbE \left[ \left( 
          \sum_{m=1}^N \normH{u^m_{h\tau} - u^{m-1}_{h\tau}}^2 \right)^{p/2}
      \right]^{1/p} \nonumber } \\
    &\leq& 
    C \left(
      \norm{u^0_{h\tau}}_{L^p(\Omega,H)} 
      + \norm{F_{h\tau}(u_{h\tau})}_{L^{p/2}(\Omega,L^1(0,T))}^{1/2}
      + T^{1/2-1/p} \norm{g_{h\tau}}_{L^p(\Omega, \LpH)} \right). 
  \end{eqnarray*}
\end{lemma}

Pathwise continuity is an essential property of martingale solutions;
that is, for almost every $\omega \in \Omega$ the map $t \mapsto
u(\omega,t)$ is continuous. Solutions of nonlinear PDE's may not be
pathwise continuous into the pivot space $H$; however, continuity into
the dual space $U'$ follow from standard arguments.  Specifically,
H\"older continuity into $U'$ is established by showing that solutions of
the numerical scheme satisfy the hypothesis the following theorem
\cite[Theorem 3.3]{DaZa92}.

\begin{theorem}\label{kc1} {\bf (Kolmogorov-Centsov)}
  Let $(\Omega, \calF, \bbP)$ be a probability space, $\mathcal X$ be
  a Banach space, and $u \in L^1(\Omega, L^p[0,T;{\mathcal X}])$. If
  for some $0 < \theta \leq 1$ there exists $\Chat > 0$ such that for
  all $0 \leq \delta < T$
  $$
  \bbE \Bigl[\int_\delta^T \norm{u(t) - u(t-\delta)}_{\mathcal X}^p \, dt\Bigr]
  \leq C^p \delta^{1+\theta p},
  $$
  then there exists a modification of $u$ on a null set of $(0,T)$
  such that $u(\omega) \in C^{0,\theta'}[0,T;{\mathcal X}]$ for almost
  every $\omega \in \Omega$ and all $0 < \theta' < \theta$; in
  particular
  $\bbE\left[\norm{u}^p_{C^{0,\theta'}[0,T;{\mathcal X}]}\right] < C$.
\end{theorem}

Piecewise linear interpolants $\uhat_{h \tau}$ of numerical schemes are
Lipschitz (in the time variable), so no modification is required; the bound
on the H\"older norm is the essential content.

The following theorem bounds translates of solutions of the difference
scheme \eqnref{:eulerF} appearing in the Kolmogorov--Centsov
theorem.  The spatial discretization plays no role in this lemma; $U$
is an arbitrary Banach space. Setting $U = U_h$ establishes H\"older
continuity of the discrete solution for almost all paths in the dual
space $U_h'$ which has norm
$$
\norm{u}_{U_h'} = \sup_{v_h \in U_h} \frac{(u, v_h)_H}{\normU{v_h}}.
$$
This is a norm on $U_h$ and a semi--norm on $U$ with $\norm{u}_{U_h'}
\leq C \normUp{u}$. If $P_h:H \rightarrow U_h$ denotes the orthogonal
projection and $u_h \in U_h$ then
$$
\normUp{u_h} 
= \sup_{v \in U} \frac{(u_h, v)_H}{\normU{v}}
= \sup_{v \in U} \frac{(u_h, P_h(v))_H}{\normU{P_h(v)}}
\frac{\normU{P_h(v)}}{\normU{v}}
\leq \norm{u_h}_{U'_h}
\left(\sup_{v \in U} \frac{\normU{P_h(v)}}{\normU{v}} \right).
$$
In a finite element context the supremum on the right is bounded
independently of $h$ under mild conditions on the underlying mesh
\cite{CrTh87}. In this situation a function taking values in $U_h$ is
bounded in $U'$ when it is bounded in $U'_h$.

\begin{theorem} \label{thm:holderF} 
  Let $1 \leq q \leq \infty$, and $U \embed H$ be
  an embedding of a Banach space into the Hilbert space $H$ so that $U
  \embed H \embed U'$.  Let $0 = t^0 < t^1 < \ldots < t^N = T$ be a
  uniform partition of $[0,T]$ with time step $\tau$ and $(\Omega,
  \calF,\{ \calF^n\}_{n=0}^N, \bbP)$ be a (discretely) filtered
  probability space.  Let $\{u^n_\tau\}_{n=0}^N$ be an adapted process
  taking values in $U$, satisfying the difference scheme
  $$
  (u^n_\tau - u^{n-1}_\tau, v)_H 
  = \tau ( F^n_\tau ,v) 
  + ( g_\tau^{n-1}, v)_H \xi^n_\tau,
  \qquad v \in U,
  $$
  with
  \begin{itemize}
  \item $\{\xi^n_\tau\}_{n=1}^N$ satisfying Assumption \ref{ass:spde0} 
    with $p > 2$.

  \item $u^0_\tau \in L^p(\Omega,U')$.

  \item $F_\tau \in L^p(\Omega, \LqpUp)$, and $F_{\tau}^n$ is
    $\calF^n$-measurable for $1 \leq n \leq N$,

  \item $g_\tau \in L^p(\Omega, \LpH)$, and $g^{n-1}_\tau$ is
    $\calF^{n-1}$-measurable for $1 \leq n \leq N$,
  \end{itemize}
  where $F_\tau(t)=F^n_\tau$ and $g_\tau(t) = g^{n-1}_\tau$ on
  $(t^{n-1}, t^n)$ denote the piecewise constant functions. Then the
  piecewise linear interpolant $\uhat_\tau$ of $\{u^n_\tau\}_{n=1}^N$
  satisfies
  $$
  \bbE \Bigl[\int_\delta^T \norm{\uhat_\tau(t) 
    - \uhat_\tau(t-\delta)}^p_{U'} \, dt\Bigr]
  \leq C \left(\norm{F_\tau}^p_{L^p(\Omega, \LqpUp)} 
    + \norm{g_\tau}_{L^p(\Omega, L^p[0,T;U'])}^p \right) \delta^{1+\theta p},
  \quad 0 < \delta < T,
  $$
  with $\theta = \min(1/2-1/p, 1/q)$. In particular, $\uhat_\tau$ is
  bounded in $L^p(\Omega, C^{0,\theta'}[0,T;U'])$ for all $0 < \theta'
  < \theta$, and the difference between the piecewise constant
  interpolant $u_\tau$ and $\uhat_\tau$ is bounded by
  $$
  \norm{u_\tau - \uhat_\tau}_{L^p(\Omega,L^\infty[0,T; U'])}
  \leq C \tau^{\theta'}
  \left(\norm{F_\tau}_{L^p(\Omega, \LqpUp)}
    + \norm{g_\tau}_{L^p(\Omega, L^p[0,T;U'])} \right).
  $$
\end{theorem}

For piecewise linear functions on a uniform partition of $[0,T]$ it
suffices to bound translates by integer multiples, $m \tau$, of the
time step. To verify this, write $\delta = m \tau + s$ with
$0 \leq m < N$ and $0 \leq s \leq \tau$. From the triangle
inequality
$$
\normUp{\uhat_\tau(t) - \uhat_\tau(t-\delta)}
\leq \normUp{\uhat_\tau(t)-\uhat_\tau(t-m\tau)} 
+ \normUp{\uhat_\tau(\tilde{t}) - \uhat_\tau(\tilde{t}-s)},
\quad \text{ where } \quad \tilde{t} = t-m\tau.
$$
{\em Granted} that for all $1 \leq m \leq N$
\begin{equation} \label{eqn:translateM}
\bbE \Bigl[\int_{m\tau}^T \normUp{\uhat_\tau(t)-\uhat_\tau(t-m\tau)}^p \, dt\Bigr]
\leq C(p) \bbE \Bigl[\sum_{n=m}^N \tau \normUp{u^n_\tau - u^{n-m}_\tau}^p\Bigr]
\leq C(p,f,g) (m\tau)^{1+\theta p},
\end{equation}
it suffices to show that translates of size $0 < s \leq \tau$ can be
bounded by $s^{1+\theta p}$.
\begin{itemize}
\item If $t \in (t^n+s, t^{n+1})$ then 
  $$
  \normUp{\uhat_\tau(t) - \uhat_\tau(t-s)} 
  = (s/\tau) \normUp{u^{n+1}_\tau - u^n_\tau}.
  $$

\item If $t \in (t^n, t^n+s)$ use the triangle inequality to write
  $$
  \normUp{\uhat_\tau(t) - \uhat_\tau(t-s)}
  \leq \normUp{\uhat_\tau(t) - u^n_\tau} + \normUp{u^n_\tau - \uhat_\tau(t-s)}.
  $$
  Explicit formulas for the piecewise linear interpolants on each
  interval show 
  $$
  \normUp{\uhat_\tau(t) - u^n_\tau} + \normUp{u^n_\tau - \uhat_\tau(t-s)}
  \leq \bigl(\normUp{u^{n-1}_\tau - u^n_\tau}
    + \normUp{u^{n+1}_\tau - u^n_\tau} \bigr) (s/\tau)
  \qquad t^n \leq t \leq t^n + s.
  $$
\end{itemize}
Inequality in \eqnref{:translateM} with $m=1$ then gives
\begin{eqnarray*}
\bbE \Bigl[\int_\tau^T \normUp{\uhat(t)-\uhat(t-s)}^p \, dt\Bigr]
&\leq& C(p) (s/\tau)^p \bbE \Bigl[\sum_{n=1}^N \tau 
\normUp{u^n_\tau - u^{n-1}_\tau}^p \Bigr] \\
&\leq& C(p,f,g) \tau^{1+\theta p} (s/\tau)^p 
\leq C(p,f,g) s^{1+\theta p},
\end{eqnarray*}
where the last inequality holds since $1 + \theta p \leq p$ when
$p \geq 2$, and
$s / \tau \leq 1$. The following lemma shows that solutions of the implicit
Euler scheme do satisfy inequality \eqnref{:translateM}.

\begin{lemma} \label{4.4}
  Under the hypotheses of Theorem \ref{thm:holderF},
  $$
  \bbE \Bigl[\sum_{n=m}^N \tau \normUp{u_\tau^n - u_\tau^{n-m}}^p\Bigr]
  \leq C \left(\norm{F_\tau}^p_{L^p(\Omega, \LqpUp)} 
    (m\tau)^{1+p/q}
    + \norm{g_\tau}_{L^p(\Omega, L^p[0,T;U'])}^p  (m\tau)^{p/2} \right),
  $$
  for all $ 1 \leq m \leq N$ when $1/p + 1/q \leq 1$; otherwise, $p
  \leq q'$ and
  $$
  \bbE \Bigl[\sum_{n=m}^N \tau \normUp{u_\tau^n - u_\tau^{n-m}}^p\Bigr]
  \leq C \left(\norm{F_\tau}^p_{L^p(\Omega, L^p[0,T;U'])} (m\tau)^p
    + \norm{g_\tau}_{L^p(\Omega, L^p[0,T;U'])}^p  (m\tau)^{p/2} \right).
  $$
\end{lemma}

\begin{proof} Let $v \in U$ and sum the difference scheme
  \eqnref{:eulerF} from $n-m+1$ to $n$ to obtain
  \begin{eqnarray*}
    (u_\tau^n - u_\tau^{n-m}, v)_H
    &=& \sum_{k=n-m+1}^n \tau ( F^k_\tau, v)
    + \sum_{k=n-m+1}^n (g_\tau^{k-1} \xi^k_{\tau} , v)_H \\
    &\leq& \left( \sum_{k=n-m+1}^n \tau \normUp{F^k_\tau} 
      + \normUp{ \sum_{k=n-m+1}^n g_\tau^{k-1} \xi^k_{\tau}} \right)
    \normU{v}.
  \end{eqnarray*}
  Taking the supremum on the left over $v \in U$ with
  $\normU{v} = 1$, raising both sides to the power $p$, and summing
  shows
  $$
  \bbE \Bigl[\sum_{n=m}^N \tau \normUp{u_\tau^n - u_\tau^{n-m}}^p \Bigr]
  \leq C
  \bbE \Bigl[\sum_{n=m}^N \tau 
  \left\{ \left(\sum_{k=n-m+1}^n \tau \normUp{F^k_\tau} \right)^p
    + \normUp{\sum_{k=n-m+1}^n g_\tau^{k-1} \xi^k_{\tau}}^p \right\}\Bigr].
  $$
  The first term on the right is bounded using  H\"older's inequality. If
  $q' \leq p$ then
  \begin{eqnarray*}
    \bbE \Bigl[\sum_{n=m}^N \tau 
    \left(\sum_{k=n-m+1}^n \tau \normUp{F^k_\tau} \right)^p\Bigr]
    &\leq&
    \bbE \Bigl[\sum_{n=m}^N \tau 
    \left(\sum_{k=n-m+1}^n \tau \normUp{F^k_\tau}^{q'} \right)^{p/q'} (m \tau)^{p/q} \Bigr]\\
    &\leq&
    \bbE \Bigl[\sum_{n=m}^N \tau 
    \left( \sum_{k=n-m+1}^n \tau \normUp{F^k_\tau}^{q'} \right)
    \norm{F_\tau}_{\LqpUp}^{p-q'} (m \tau)^{p/q}\Bigr] \\
    &\leq&
    \bbE \left[ 
      \norm{F_\tau}_{\LqpUp}^p \right] (m \tau)^{1+p/q}.
  \end{eqnarray*}
  When $q' = p$ the exponent in the last term is $1+p/q = 1+ p(1-1/q') = p$.

  The second term is a (discrete) Ito integral and is bounded using
  the discrete BDG inequality, Theorem
  \ref{thm:BDG}, and Lemma \ref{lem:mglVariation},
  \begin{eqnarray*}
    \bbE \Bigl[\sum_{n=m}^N \tau 
    \left( \normUp{\sum_{k=n-m+1}^n 
        g_\tau^{k-1} \xi^k_{\tau}} \right)^p\Bigr]
    &\leq& C
    \sum_{n=m}^N \tau \left( 
      \sum_{k=n-m+1}^n \tau \norm{g_\tau^{k-1}}_{L^p(\Omega,U')}^p \right)  
    (m \tau)^{p/2 - 1}\Bigr] \\
    &\leq&
    C \sum_{k=1}^N \tau \norm{g_\tau^{k-1}}_{L^p(\Omega,U')}^p 
    \, (m \tau)^{p/2}.
  \end{eqnarray*}
\end{proof}

\subsection{Compactness} \label{sec:compactness} 

The Prokhorov theorem, stated next, will be used to establish
convergence of the laws of the solutions to the implicit Euler
equation \eqnref{:eulerF}. The key hypothesis of this theorem requires
a sequence of probability measures $\{\bbP_n\}_{n=1}^\infty$ on a
topological space $\bbX$ endowed with its Borel $\sigma$-algebra
$\calB(\bbX)$ to be tight (see Definition \ref{def:tight}).

\begin{theorem}[Prokhorov]
  Let $\bbX$ be a topological space with the property that there
  exists a countable family of real-valued continuous functions which
  separates points of $\bbX$. Let $\{\tilde{\bbP}_n\}_{n=1}^\infty$ be
  a tight sequence of probability measures on
  $\bbX$ with its Borel $\sigma$-algebra $\calB(\bbX)$. Then there exist
  a subsequence $\{\tilde{\bbP}_{n_k}\}_{k \in 
  \mathbb N}$ and a probability measure $\tilde{\bbP}$ on $(\bbX,
  \calB(\bbX))$ for which $\tilde{\bbP}_{n_k} \Rightarrow
  \tilde{\bbP}$.
\end{theorem}

If the probability measures in this theorem are the laws of random
variables $\tilde{\bbP}_n = \calL(X_n)$ taking values in $\bbX$, and
$\calL(X_{n_k}) \Rightarrow \tilde{\bbP}$ we can write $\bbPt =
\calL(X)$ where $X:\bbX \rightarrow \bbX$ is the identity function
identified as a random variable on $\Omegat \equiv (\bbX, \calB(\bbX),
\tilde{\bbP})$.

Below we will set
$$
\bbX = G[0,T; U'] \cap L^r[0,T;U]_{weak}
\times \LqpUp_{weak} 
\times L^2[0,T;U'] \times C[0,T],
$$
and the probabilities in the Prokhorov theorem to be the laws of
$X_{h\tau} = (u_{h\tau}, F_{h\tau}, g_{h\tau}, \What_{\tau})$. Recall
that $\What_{\tau}$ denotes the piecewise linear interpolant of
\eqnref{:AB1}, and $L^r[0,T;U]_{weak}$ and $\LqpUp_{weak}$ the
indicated spaces endowed with the weak topology.  When $U$ is
separable and reflexive and $1 < r, q' < \infty$, classical results
from functional analysis can be used to exhibit a countable family of
real-valued functions on $\bbX$ which separate points.

A convenient way to establish compactness of piecewise constant functions
in $G[0,T; U']$ is to use the Arzela--Ascoli theorem to show that
their corresponding piecewise linear interpolants are compact in
$C[0,T; U']$.  The following lemma makes this precise, and also shows
that the laws concentrate on $C[0,T; U']$.

\begin{lemma} \label{lem:UandUhat} 
Let $U$ be a separable reflexive Banach space, $H$ a Hilbert space,
$U \embed H \embed U'$ be dense embeddings, and $(\Omega, \calF,
\bbP)$ be a probability space. For $n = 1, 2, \ldots$, let
$\{u^i_n\}_{i=0}^n$ be $U$-valued processes, and define their caglad and
piecewise linear interpolants on $[0,T]$ by
$$
u_n = u_n^0\mathbf 1_{\{0\}} 
+ \sum_{i=0}^{n-1}u^i_n\mathbf 1_{(t_n^i,t_n^{i+1}]}
\quad \text{ and } \quad
\uhat_n(t) = \frac{t_n^{i+1}-t}{t_n^{i+1}-t_n^i}u_n^i
+ \frac{t-t_n^i}{t_n^{i+1}-t_n^i}u_n^{i+1}
\quad\text{for}\quad t\in[t_n^i,t_n^{i+1}],
$$
where $t_n^i=iT/n$. 

If $\{\calL(\uhat_n)\}_{n=1}^\infty$ is tight on $C[0,T;U^\prime]$ and
$\{\calL(u_n)\}_{n=1}^\infty$ is tight on $L^r_{weak}[0,T;U]$, then
$\{\calL(u_n)\}_{n=1}^\infty$ is tight on $G[0,T;U^\prime]\cap
L^r_{weak}[0,T;U]$, and if $\mu$ is any accumulation point of
$\{\calL(u_n)\}_{n=1}^\infty$ on $G[0,T;U^\prime]\cap
L^r_{weak}[0,T;U]$ then
$$
\mu(C[0,T;U^\prime]\cap L^r[0,T;U])=1.
$$
\end{lemma}

Note that the Borel subsets of $\LrU_{weak}$ and $\LrU$ coincide since
$U$ is separable. This lemma, which we prove in the Appendix, is used
to establish tightness of solutions to the numerical scheme.

\begin{theorem} \label{thm:hLimits}
  Let $U$ be a separable reflexive Banach space, $H$ a Hilbert space,
  and $U \cembed H$ be a compact, dense embedding, and 
  $(\Omega, \calF, \bbP)$ be a probability space. Assume
  that the spaces, data, and increments, of the scheme \eqnref{:eulerF}
  satisfy Assumptions  \ref{ass:spde1} and  \ref{ass:spde0} with $p \in
  (2,\infty)$, and that the initial data $\{u^0_{h\tau}\}$ are bounded
  in $L^p(\Omega, H)$ and converge in $L^2(\Omega,H)$ to a limit
  $u^0$ as $(h,\tau) \rightarrow (0,0)$.

  Let $u_{h\tau}$ denote the piecewise
  constant caglad interpolant of $\{u^n_{h\tau}\}_{n=0}^N$ in
  time and assume for some $1 < q,r < \infty$ that
  \begin{enumerate}
  \item $\{\norm{u_{h\tau}}_{L^p(\Omega, \LrU)}\}_{h,\tau>0}$ is bounded.
    
  \item $\{\norm{F_{h\tau}}_{L^p(\Omega, \LqpUp)} \}_{h,\tau>0}$ 
    is bounded.

  \item $\{\norm{g_{h\tau}}_{L^p(\Omega, L^p[0,T;H])} \}_{h,\tau>0}$
    is bounded.
  \end{enumerate}
  Then the laws of $\{(u_{h\tau}, F_{h\tau}, g_{h \tau}, \What_{\tau})
  \}_{h,\tau > 0}$ are tight on
  $$
  \bbX = G[0,T; U'] \cap \LrU_{weak}
  \times \LqpUp_{weak} 
  \times \LtwoH_{weak} \times C[0,T].
  $$
  In addition,
  \begin{itemize}
  \item If $\{g_{h\tau}\}_{h,\tau > 0}$ is Cauchy in $L^p(\Omega,\LtwoH)$
    then the laws $\{\calL(g_{h\tau})\}_{h,\tau > 0}$
    are tight on $\LtwoH$.
    
  \item The piecewise linear interpolants $\{\uhat_{h\tau}\}_{h,\tau >
      0}$ are tight in $C[0,T; U'] \cap \LrU_{weak}$.

  \item If additionally $V \embed {U'}$ is a separable
     reflexive Banach space and $\{u_{h\tau}\}_{h,\tau > 0}$ is
     bounded in $L^p(\Omega,L^s[0,T;V])$ for some $1 < s < \infty$,
     then the laws $\{\calL(u_{h\tau}\}_{h,\tau > 0}$ are tight on
     $$
     G[0,T;U']
     \cap \LrU_{weak}
     \cap L^s[0,T;V]_{weak}. 
     $$
     If $U \cembed V$ is compact and $1 \leq \shat < s$, then the laws
     are tight on
     $$
     G[0,T;U']
     \cap \LrU_{weak}
     \cap L^{\shat}[0,T;V].
     $$
  \end{itemize}
\end{theorem}

\begin{proof}
  To establish tightness for the laws we exhibit large compact
  sets in each of the factor spaces of $\bbX$.
  \begin{itemize}
  \item If $U \cembed H$ then $U \cembed H \cembed U'$, and for
    $\theta > 0$ 
    $$
    C^{0,\theta}[0,T;U'] \cap L^r[0,T;U] \cembed C[0,T; U'].
    $$
    Fix $0 < \theta < \min(1/2-1/p, 1/q)$ and let
    $$
    K_\epsilon = \{\uhat \in C^{0,\theta}[0,T;U'] \sst
    \lrU{\uhat}^p \leq 1/\epsilon
    \,\, \text{ and } \,\, 
    \norm{\uhat}_{C^{0,\theta}[0,T;U']}^p \leq 1/\epsilon \},
    $$
    If $\tau = 1/N$ then
    \begin{eqnarray*}
      \calL(\uhat_{h \tau})[C[0,T; U'] \setminus K_\epsilon]
      &=& \bbP \left[\{\omega \in \Omega \sst 
        \uhat_{h\tau} \not\in K_\epsilon \} \right] \\
      &\leq& \bbP \left[\{\omega \in \Omega \sst 
      \lrU{\uhat_{h \tau}}^p > 1/\epsilon
      \,\, \text{ or } \,\, 
      |\uhat_{h \tau}|_{C^{0,\theta}[0,T;U']}^p > 1/\epsilon \} \right] \\
      &\leq& 
      C \left( 
        \norm{\uhat_{h \tau}}_{L^p(\Omega, \LrU)}^p + 
        \norm{\uhat_{h \tau}}_{L^p(\Omega, C^{0,\theta}[0,T;U'])}^p
      \right) \epsilon,
    \end{eqnarray*}
    where the last line follows from Chebyshev's inequality.  The
    hypotheses assumed upon the data and Theorem \ref{thm:holderF}
    bound the two norms in the last expression independently of $h$
    and $\tau$ which shows $\calL(\uhat_{h \tau})[K_\epsilon] \geq 1 -
    C \epsilon$, and tightness on $C[0,T; U']$ follows.

    If $U \cembed V$ then $L^r[0,T;U] \cap C^{0,\theta}[0,T; U']
    \cembed L^r[0,T; V]$ and the same argument shows that
    $\{\calL(\uhat_{h\tau})\}_{h,\tau > 0}$ are tight in $L^r[0,T;
    V]$. The mapping $\uhat_{h\tau} \mapsto u_{h\tau}$ is a bijective
    and continuous on $L^r[0,T;V]$, so maps compact sets to compact
    sets, so $\{\calL(u_{h\tau})\}_{h,\tau > 0}$ is also tight
    on $L^r[0,T; V]$.

    If $1 \leq r < s$ and $K \subset L^r[0,T;V]$ is compact, then $K
    \cap L^s[0,T;V]$ is compact in $L^{\shat}[0,T;V]$ for $1 \leq \shat <
    s$. Thus if $\{u_{h\tau}\}_{h,\tau > 0}$ is also bounded in
    $L^p(\Omega, L^s[0,T;V])$ the laws are also tight in
    $L^{\shat}[0,T;V]$.

  \item Since $U$ is reflexive and $1 < q < \infty$ the Banach-Alaoglu
    theorem shows
    $$
    \LqpUp_{strong} \cembed L^{q'}[0,T; U']_{weak}.
    $$
    Letting $K_\epsilon$  be the closed ball in $\LqpUp$
    centered at the origin with radius $1/\epsilon$, Chebyshev's
    inequality shows
    $$
    \calL(F_{h\tau})[\LqpUp \setminus K_\epsilon]
    = \bbP \left[ \{\omega \in \Omega \sst 
    \norm{F_{h\tau}}_{\LqpUp} > 1 / \epsilon \} \right]
    \leq \norm{F_{h\tau}}_{L^p(\Omega, \LqpUp)} \epsilon.
    $$
    Since closed balls in $\LqpUp$ are weakly compact,
    and $\norm{F_{h\tau}}_{L^p(\Omega, \LqpUp)}$ is bounded
    independently of $k$, it follows that
    $\{\calL(F_{h\tau})\}_{h,\tau > 0}$ is tight in $\LqpUp_{weak}$.

    The same argument shows $\{\calL(\uhat_{h\tau})\}_{h,\tau > 0}$ is
    tight in $\LrU_{weak}$ and is also tight in $L^s[0,T;V]_{weak}$
    when $\{u_{h\tau}\}_{h,\tau > 0}$ is bounded in $L^p(\Omega,
    L^s[0,T;V])$. The previous lemma then shows that the laws of
    $\{u_{h\tau}\}_{h,\tau > 0}$ are tight on $G[0,T;U'] \cap \LrU_{weak}$.
    
  \item The laws of a strongly convergent sequence in $L^p(\Omega; X)$
    are always tight and converge weakly to the limit. In particular,
    if $\{g_{h\tau}\}_{h,\tau > 0}$ is Cauchy in $L^2(\Omega,
    \LtwoUp)$, then $\{\calL(g_{h\tau})\}_{h,\tau > 0}$ are tight
    and converge weakly to $\calL(g)$.

  \item The discrete Wiener process $\hat{W}_\tau$ interpolating
    $\{W^n_{\tau} \}_{n=0}^N$ is H\"older continuous.  Briefly, from
    Lemma \ref{lem:mglVariation} (with $H= \Re$ and $g^m=1$) it
    follows that
    $$
    \bbE \Bigl[\sum_{n=m}^N \tau |W^n_{\tau} -W_{\tau}^{n-m}|^p\Bigr]
    = \bbE \Bigl[\sum_{n=m}^N 
    \tau \left|\sum_{k=n-m+1}^n \xi^k_{\tau}\right|^p\Bigr] \\
    \leq C \sum_{n=m}^N \tau (m\tau)^{p/2}.
    $$
    Since $p > 2$ the Kolmogorov-Centsov Theorem \ref{kc1} bounds the
    expected value of the H\"older norm of $\What_{\tau}$ with
    exponent $\theta < \min(1/2-1/p,1)$.  Tightness then follows from
    the Arzella-Ascolli theorem since $C^{0,\theta}[0,T]$ is compactly
    embedded in $C[0,T]$.
  \end{itemize}
\end{proof}

\subsection{Convergence: Proof of Theorem \ref{thm:main}}
This section establishes convergence along subsequences of solutions of
the numerical scheme \eqnref{:eulerF} to a weak martingale solution of
the problem \eqnref{:Spde} which we write as $du = F \, dt + g \, dW$.
Convergence is established using the Prokhorov theorem to construct a
measure $\bbPt$ on the product space
\begin{equation}\label{product_space}
\Omegat 
= G[0,T;U'] \cap \LrU
  \times \LqpUp
  \times \LtwoH
  \times C[0,T],
\end{equation}
and a filtration for which the projections
\begin{equation*}
u :\Omegat \rightarrow C[0,T;U'] \cap \LrU, \quad 
F :\Omegat \rightarrow \LqpUp, \quad 
g :\Omegat \rightarrow \LtwoUp, \quad 
W :\Omegat \rightarrow C[0,T],
\end{equation*}
defined by
\begin{equation} \label{eqn:uFgW}
u(\omegat) = \omegat_1, \quad
F(\omegat) = \omegat_2, \quad
g(\omegat) = \omegat_3, \quad
W(\omegat) = \omegat_4,
\quad \text{ with } \quad
\omegat = (\omegat_1,\omegat_2,\omegat_3,\omegat_4),
\end{equation}
are random variables satisfying \eqnref{:Spde}. 
To verify that these variables are a solution,
independence properties of the approximating scheme are used to show that
$$
X(t) \equiv u(t) - u^0 - \int_0^t F\, ds,
$$
is a martingale with respect to the filtration generated by
$(u,F,g,W)$. The final step is to verify that $W$ is a Wiener process
and $X(t) = \int_0^t g \, dW$.

The following lemma, which characterizes when one process is
independent of the filtration generated by another, is useful in this
context.

\begin{lemma} \label{lem:indep} Let $\{\bbX_t\}_{t=0}^T$ be
  topological spaces, $\{Y(t)\}_{t=0}^T$ be $\bbX_t$-valued Borel
  measurable random variables, and let $\{\calF_t\}_{t=0}^T$ be the
  filtration given by $\calF_t = \sigma(Y(s) : 0 \leq s \leq t)$. An
  integrable process $\{X(t)\}_{t=0}^T$ adapted to this filtration
  taking values in a separable Banach space $\bbX$ is a
  martingale with respect to the filtration if and only if
  $$
  \bbE\left[ \bigl(X(t)-X(s)\bigr) \,
    \prod_{i=1}^m \prod_{j=1}^n
    \phi_{ij}\bigl(\psi_{ij}(Y(s_j))\bigr) \right] = 0
  $$
  holds for all times $0 \leq s_1 < \ldots < s_m \leq s < t \leq T$, all
  $\phi_{ij}\in C_b(\Bbb R)$, and for all $\psi_{1j}, \ldots,
  \psi_{nj} \in \mathcal A_{s_j}$, where $\mathcal A_s$ is a subset of
  real-valued functions on $\bbX_s$ for which $\sigma(\calA_s) =
  \calB(\bbX_s)$ (the Borel $\sigma$-algebra on $\bbX_s$).
\end{lemma}

The Dynkin lemma shows that the criteria in this lemma is equivalent
to $\bbE[X(t)-X(s) \sst \calF_s] = 0$.

\begin{example} \label{ex:As} In the proof below 
$$
\Bbb X_t=G[0,t;U^\prime] \cap L^r[0,t;U]_{weak}
\times L^{q'}[0,t;U^\prime]_{weak} \times L^2[0,t;H] \times C[0,t].
$$
A set of continuous functions, $\mathcal A_t$, generating $\calB(\bbX_t)$
is 
$$
(u,F,g,W)\mapsto z_1 \int_a^b(u(r),v)\,dr 
+ z_2 \int_a^b(F(r),v)\,dr
+ z_3 \int_a^b(g(r),v)\,dr 
+ z_4\int_a^bW(r)\,dr,
$$
for $a<b$ in $[0,t]\cap\Bbb Q$, $v$ in a dense subset of $U$, and
$z_1,z_2,z_3,z_4 \in \{0,1\}$.
\end{example}

\begin{proof}(of Theorem \ref{thm:main})
Let
$$
\bbX 
= G[0,T;U'] \cap \LrU_{weak}
\times \LqpUp_{weak}
\times \LtwoH
\times C[0,T],
$$
and $(\Omegat,\calFt) = (\bbX, {\mathcal B}(\bbX))$ be the
corresponding measurable space endowed with the Borel sigma algebra.
Let $\bbPt_{hk}$ denote the law of $(u_{hk}, F_{hk}, g_{hk},
\hat{W}_\tau)$; that is
$$
\bbPt_{h\tau}[B_1 \times B_2 \times B_3 \times B_4]
= \bbP[
(u_{h\tau} \in B_1) \wedge
(F_{h\tau} \in B_2) \wedge
(g_{h\tau} \in B_3) \wedge
(\hat{W}_{\tau} \in B_4)]
$$
for 
$$
B_1 \in \calB(G[0,T;U'] \cap \LrU), \,
B_2 \in \calB(\LqpUp), \,
B_3 \in \calB(\LtwoH), \,
B_4 \in \calB(C[0,T]).
$$
Recall that the Borel subsets of $\LrU$ and $\LqpUp$ with the
weak and strong topologies coincide.

Theorem \ref{thm:hLimits} shows that the measures
$\{\bbPt_{h\tau}\}_{h,\tau > 0}$ form a tight family, so by Prokhorov
theorem we may pass to a subsequence $(h_k, \tau_k) \rightarrow (0,0)$
for which $\bbPt_{h_k\tau_k} \Rightarrow \bbPt$ and Lemma
\ref{lem:UandUhat} shows that 
$$
\bbPt\left[\left\{(u,F,g,W) 
    \sst u \in C[0,T;U'] \cap \LrU \right\} \vph\right] = 1.
$$
Below we write $\bbPt_k = \bbPt_{h_k \tau_k}$, $u_k = u_{h_k \tau_k}$
etc.

For $0 < t \leq T$, let $X(t):\bbX \rightarrow U'$ be the function
$X(t) = u(t) - u(0) - \int_0^t F \, ds$. We construct a filtration of
$(\Omegat, \calFt)$ for which $X(t)$ is a square integrable martingale.

For $v \in U$ and $0 \leq t \leq T$ fixed we first verify that
$(X(t),v)$ is square integrable on $(\Omegat, \calFt, \bbPt)$ and show
$$
\int_{\bbX} (X(t), v) \, d\bbPt
= \lim_{k \rightarrow \infty} 
\bbE\left[\left(u_k^{n_k} - u_k^0,v \vph\right)_H 
  - \int_0^t (F_k, v) \, ds\right],
$$
when $0 \leq n_k \tau_k - t < \tau$ with $n_k \in \bbN$, i.e.
$u_k(t) = u_k^{n_k}$ as in Figure \ref{fig:caglad}.

To do this, define $\zeta:\bbX \rightarrow \Re$ by
$
\zeta(u,F,g,W) = (X(t), v).
$
Since the mapping $u \mapsto u(t)$ is Borel on $G[0,T; U']$, and the
coordinate projections in equation \eqnref{:uFgW} are continuous, it
follows that $X(t)$, and hence $\zeta$ is Borel measurable. Set
$$
N = \left\{(u,F,g,W) \in \bbX \sst 
\exists (\ubar_k,\Fbar_k,\gbar_k,\Wbar_k) \rightarrow (u,F,g,W)
\text{ such that } \zeta(\ubar_k, \Fbar_k, \gbar_k, \Wbar_k) \not\rightarrow
\zeta(u,F,g,W) \vph\right\}.
$$
\begin{claim}
  $N$ has null outer $\bbPt$ measure, $\bbPt^*[N] = 0$.
\end{claim}
\begin{proof}
\begin{itemize}
\item If $\Fbar_k \rightarrow F$ in $\LqpUp_{weak}$ it is immediate that
  $\ds
  \int_0^t (\Fbar_k,v) \, ds \rightarrow \int_0^t (F,v)\, ds.
  $

\item If $\ubar_k \rightarrow u$ in $G[0,T;U'] \cap \LrU_{weak}$ then
  $$
  (\ubar_k(t), v)_H \not\rightarrow (u(t),v)_H
  \qquad \Rightarrow \qquad u \not\in C[0,T; U'].
  $$
\end{itemize}
From Lemma \ref{lem:UandUhat} we conclude
$
\bbPt^*[N] \leq
\bbPt^*\left[\left\{(u,F,g,W) 
\sst u \not\in  C[0,T; U'] \cap \LrU \vph\right\} \right] 
= 0.
$
\end{proof}
Since
$$
|\zeta(u_k, F_k, g_k, \What_k)|
\leq \left(2 \linfUp{u_k} + \norm{F_k}_{L^q[0,T,U']} \, (t-s)^{1/q} \vph\right)
\normU{v},
$$
Lemma \ref{lem:boundF} bounds the $p^{th}$ moment of the right-hand
side, with $p > 2$, so from Lemma \ref{lem:mapping_theorem} (with
$\zeta_k = \zeta$) we conclude that $(X(t),v)$ is square integrable and
$$
\int_{\Omegat} (X(t),v) \, d\bbPt
= \int_\bbX \zeta \, d\bbPt
= \lim_{k \rightarrow \infty} \int_\bbX \zeta \, d\bbPt_k
\equiv \lim_{k \rightarrow \infty} \bbE\left[
(u^{n_k}_k - u_k^0 , v)_H - \int_0^t (F_k,v) \, ds\right].
$$

Next, let $\{\calFt(t)\}_{t=0}^T$ be the coarsest filtration
on $\Omegat$ for which each of the mappings
$$
\Omegat \mapsto
G[0,t;U^\prime] \cap L^q[0,T;U]_{weak}
\times L^{q^\prime}[0,t;U^\prime]_{weak}
\times L^2[0,t;H]
\times C[0,t] \equiv \bbX_t,
$$
given by
$$
(u,F,g,W)\mapsto(u|_{[0,t]},F_{[0,t]},g_{[0,t]},W|_{[0,t]}),
\qquad 0 \leq t \leq T,
$$
is a measurable map $(\Omegat,\calFt) \to (\bbX_t, \calB(\bbX_t))$.

\begin{claim} For $v \in U$ fixed, the real-valued random variable
  $(X(t),v) = (u(t) - u(0),v)_H - \int_0^t (F,v)\, ds$ is a martingale on
  $(\Omegat, \calFt, \{\calFt(t)\}_{0 \leq t \leq T}, \bbPt)$.
\end{claim}
\begin{proof}
  Fix $0 \leq s_1 < \ldots < s_m \leq s < t \leq T$, $\phi_{ij} \in
  C_b(\Re)$ and let $\psi_{1j}, \ldots, \psi_{nj}$ be functions in the
  generating set of $\calB(\bbX_{s_j})$ given in Example \ref{ex:As}.
  Then let $\phi \in C_b(\bbX, \Re)$ be the function
$$
\phi(u,F,g,W) = \prod_{i=1}^m \prod_{j=1}^n
\phi_{ij}\bigl(\psi_{ij}(u,F,g,W)\bigr),
$$
and let $v_k \rightarrow v$ with  $v_k \in U_{h_k}$. Define
$\zeta, \zeta_k: \bbX \rightarrow \Re$ to be the functions
$$
\zeta(u,F,g,W) = \bigl(X(t)-X(s), v\bigr) \phi
\qquad \text{ and } \qquad
\zeta_k(u,F,g,W) = \bigl(X(t)-X(s), v_k \bigr) \phi.
$$
If
$$
N = \left\{(u,F,g,W) \in \bbX \sst 
\exists (\ubar_k,\Fbar_k,\gbar_k,\Wbar_k) \rightarrow (u,F,g,W)
\text{ such that } \zeta_k(\ubar_k, \Fbar_k, \gbar_k, \Wbar_k) \not\rightarrow
\zeta(u,F,g,W) \vph\right\},
$$
then, as above, $\bbPt^*[N] = 0$ and $\xi$ and $\xi_k$ have moments of
order $p > 2$. Lemma \ref{lem:mapping_theorem} then gives
\begin{eqnarray}
\lefteqn{\int_{\Omegat} (X(t)-X(s),v) \phi \,d\bbPt 
= \int_{\Omegat} \zeta \, d\bbPt 
= \lim_{k \rightarrow \infty} \int_{\Omegat} \zeta_k \, d\bbPt_k } \nonumber \\
&\equiv& \lim_{k \rightarrow \infty}
\bbE\left[(u_k^{n_k} - u_k^{m_k}
  - \int_s^t F_k\, dr, v_k)  \phi \right] \nonumber \\
&=& \lim_{k \rightarrow \infty}
\bbE\left[(u_k^{n_k} - u_k^{m_k} 
  - \int_{t^{m_k}}^{t^{n_k}} F_k\, dr, v_k) \phi \right]
- \lim_{k \rightarrow \infty}
\bbE\left[ \big( 
\int_s^{t^{m_k}} F_k \, dr - \int_t^{t^{n_k}} F_k \, dr, v_k\big) \phi
\right],  \label{eqn:Xinc} 
\end{eqnarray}
where $0 \leq n_k \tau_k - t < \tau_k$ and $0 \leq m_k \tau_k - s <
\tau_k$ since $u_k(t) = u_k(n_k \tau_k)$ and $u_k(s) = u_k(m_k
\tau_k)$; see Figure \ref{fig:caglad}.

We verify that each term on the right-hand side vanishes to conclude
from Lemma \ref{lem:indep} that increments of $X$ are independent and
$X$ is a martingale.

Summing the implicit Euler scheme \eqnref{:eulerF} shows
$$
(u_k^{n_k}, v_k) = (u_k^{m_k}, v_k) 
+ \tau \sum_{j=m_k+1}^{n_k} (F_k^j, v_k)
+ \sum_{j=m_k+1}^{n_k} (g_k^{j-1}, v_k)_H \xi_k^j.
$$
Multiplying this equation by $\phi$ and rearranging gives
$$
\bbE\left[(u_k^{n_k} - u_k^{m_k} 
  - \int_{t^{m_k}}^{t^{n_k}} F_k\, dr, v_k) \phi \right]
= \sum_{j=m_k+1}^{n_k} \bbE\left[(g_k^{j-1}, v_k)_H \phi \xi_k^j \right]
= \sum_{j=m_k+1}^{n_k} \bbE\left[(g_k^{j-1}, v_k)_H \phi\right]
\, \bbE[ \xi_k^j ]
= 0,
$$
where the last two steps follow since 
$\phi$ is $\calF(t^{m_k})$ measurable, $\calF(t^{m_k}) \subset 
\calF(t^{j-1})$ when $j \geq m_k+1$, and $\xi_k^j$ is independent
of $\calF(t^{j-1})$ with zero average.

The last term in \eqnref{:Xinc} vanishes since $F_k \in
L^1(\Omega,\LqpUp$) with $1 < q < \infty$,
$$
\bbE\left[\int_s^{t^{m_k}} \normUp{F_k} \, dr 
  + \int_t^{t^{n_k}} \normUp{F_k} \, dr\right]
\leq \bbE\left[\lqpUp{F_k} \right] \, 
\left(|t^{m_k}-s|^{1/q} + |t^{n_k}-t|^{1/q} \vph\right)
\rightarrow 0.
$$
\end{proof}

The arguments used here can be repeated to show that for each $v \in U$
the processes
$$
W(t), \quad
W^2(t) - t, \quad
(X(t),v)^2 - \int_0^t (g(s),v)^2 \, ds, \quad \text{ and } \quad
(X(t),v) W(t) - \int_0^t (g(s),v) \, ds,
$$
are also real-valued martingales on $(\Omegat, \calFt, \{\calFt(t)\}_{t=0}^T)$. 
The Martingale Representation Theorem \ref{thm:mglRepresentation}
then shows that $W$ is a real-valued Wiener process and
$$
(u(t),v) = (u^0, v) + \int_0^t(F(s),v)\,ds+\int_0^t(g,v)\,dW,
\qquad 0 \leq t \leq T,
$$
holds $\bbPt$--a.s. for every $v \in U$. 
Moreover, since paths of $W$
are continuous, $W$ is also a Wiener process for the augmentation of
$\{\calFt(t)\}_{0 \leq t \leq T}$ satisfying the usual conditions, so
$(u,F,g,W)$ is also a weak martingale solution with respect to the
augmented filtration.

Finally, since $u^0_k = u_k(0)$, the map $u \mapsto u(0)$ is
continuous on $G[0,T;U']$, and the initial data is assumed to converge
and has moments of order $p > 2$, it follows that $\calL(u^0_k)
\Rightarrow \calL(u(0))$ on $U'$.
\end{proof}

\subsubsection{Infinite--Dimensional Wiener Process}\label{ssect:cylinder}
If \eqref{eqn:Spde} is driven by an infinite-dimensional or cylindrical
Wiener process $W$ in a Hilbert space $K$, it would take the form
\begin{equation}\label{eqn:Spde_inf}
(u(t),v)_H + \int_0^t a(u,v)\, ds = (u^0, v)_H +  \int_0^t ( f,v )\, ds
+\sum_{j=1}^\infty\int_0^t (g_j,v)_H \, dW_j,
\quad
v \in U
\end{equation}
where $\{g_j\}_{j\in\Bbb N}$ are processes in $U^\prime$, and
$\{W_j\}_{j\in\Bbb N}$ are standard real-valued independent Wiener
processes. More precisely, $g_j(t,\omega):=g(t,\omega)e_j$ and
$W_j(t,\omega):=\langle W(t,\omega),e_j\rangle_K$, $j\in\Bbb N$ for an
orthonormal basis $\{e_j\}_{j\in\Bbb N}$ in the Hilbert space $K$,
$t\ge 0$ and $\omega\in\Omega$. The space $\bbX$ in Theorem
\eqref{thm:main} would have a form
$$
\bbX \equiv G[0,T;U'] \cap L^q[0,T;U]_{weak}
\times \LqpUp_{weak}
\times \LtwoH_{weak}^{\Bbb N}\times
C[0,T]^{\Bbb N},
$$
and, in addition to the assumptions in Theorem \eqref{thm:main}, we
would have to assume that the limits $g_j$ of the approximating
sequences $g_{j,h\tau}$ as $(h,\tau)\to (0,0)$, for $j\in\Bbb N$,
satisfy
$$
\sum_{j=1}^\infty\int_0^T|(g_j,v)|^2\,ds<\infty\text{ a.s. for every }v\in U
$$
so that the series in \eqref{eqn:Spde_inf} converges.

\section{Examples} \label{sec:examples} In this section we present
three examples that illustrate the applicability of the convergence
theory for parabolic systems that exhibit distinctly different
structural properties. In the first instance we consider the
incompressible Navier-Stokes equation driven by multiplicative noise
which has the structure of a diffusion equation. The stability
estimate for this class of problems follows upon multiplying the
equation by the solution itself.  The second example is a gradient
flow for which the spatial operator is the gradient of a (typically
non--convex) stored energy function, $I(u)$.  In the deterministic
setting stability follows upon multiplying the equation by the time
derivative of the solution. However, in the stochastic setting this is
not possible, so it is necessary to multiply the equation by $A(u)$
instead. In the final example considers the situation where $A$ is a
maximal monotone operator.

\subsection{Structural Properties}
In this section we review how structural properties of the spatial
operators give rise to specific bounds upon the solution. Following
this, we recall a convenient statement of the Brouwer fixed point
theorem which is used ubiquitously in the deterministic setting to
establish existence of solutions to the discrete problems.  Since
solutions of the nonlinear problems may not be unique, in the
stochastic setting it is necessary to establish the existence of a
measurable selection.

\subsubsection{Bounding Solutions}
Let $D \subset \Re^d$ be a bounded Lipschitz domain, $T>0$, and $f \in
L^2[0,T; L^2(D)]$ be given.  The classical heat equation with Neumann
boundary data,
\begin{equation}\label{he-1}
\partial_t u - \Delta u = f, \quad 
\text{ in } (0,T) \times D,
\qquad 
\left. \dbydp{u}{n}\right|_{\partial D} = 0,
\end{equation}
 has the structure of  both, a classical diffusion equation and a
gradient flow. Multiplying by $u$ and integrating shows
$$
\half \dbyd{}{t} \ltwo{u}^2 + \ltwo{\nabla u}^2 = (f,u),
$$
while multiplying by $\partial_t u$ gives
$$
\ltwo{\partial_t u}^2 + \half \dbyd{}{t} \ltwo{\nabla u}^2 = (f,\partial_t u).
$$
When a stochastic term is included on the right-hand side of
(\ref{he-1}), there is a loss of temporal regularity and the scalar
product of $\partial_t u$ and the stochastic term can not be
bounded. Since the spatial regularity is not degraded to the same
extent, it is frequently possible to multiply the equation by the
variational derivative, $\delta I(u)/\delta u$, of the energy. For the
heat equation (\ref{he-1}) this corresponds to multiplying by $-\Delta
u$ to get
$$
\half \dbyd{}{t} \ltwo{\nabla u}^2 
+ \ltwo{\Delta u}^2 = -(f, \Delta u).
$$
The second problem that we present in Section \ref{sec:harmonic} has this structure and, in addition, the
solution takes values in a manifold. In this instance the PDE can be
viewed as an equation on the tangent space so the stochastic term needs
to be restricted appropriately; this results in Stratonovich noise.

The numerical schemes will satisfy an estimate
of the form
\begin{equation} \label{eqn:bound2}
  I(u^n_{h\tau}) + \half \normH{u^n_{h\tau}-u^{n-1}_{h\tau}}^2
  + \tau \norm{ a^n_{h\tau}}^q_U
  \leq I(u^{n-1}_{h\tau}) + \tau (f^n_{h\tau}, a^n_{h\tau})
    + (g^{n-1}_{\tau},u^n_{h\tau}) \xi^n_\tau,
\end{equation}
where the energy $I(u)$ is non--negative and $g^{n-1}_{h\tau}$ may
depend upon $u^{n-1}_{h\tau}$. The following mild generalization of
Lemma \ref{lem:bound} establishes bounds upon the solution. To
accommodate examples like gradient flow of the heat equation, where
$\normH{u} = \ltwo{\nabla u}$, the pairing $(.,.)_H$ is only assumed
to be a semi--inner product.

\begin{lemma} \label{lem:bound2} 
Let $(\Omega, \calF, \bbP)$ be a probability space and let $U \embed
H$ be an embedding of a normed linear space into a semi--inner product
space $H$. Suppose that $I:U \rightarrow \Re$ is continuous and
satisfies $\normH{u} \leq I(u)$ for $u \in U$.  Let Assumptions
\ref{ass:spde0} and \ref{ass:heat} hold, and inequality
\eqnref{:bound2} be satisfied with random variables for which:
  \begin{itemize}
  \item $\{f^n_{h\tau}\}_{n=1}^N$ takes values in $U'$ and 
    $\{a^n_{h\tau}\}_{n=1}^N$ takes values in $U$.

  \item $\{u^n_{h\tau}\}_{n=0}^N$ takes values in $U_h$ and is adapted to
    the filtration $\{ {\mathcal F}^n\}_{n=0}^N$.

  \item $\{g^n_{h\tau}\}_{n=0}^{N-1}$ takes values in $H$ and is
    adapted to the filtration $\{ {\mathcal F}^n\}_{n=0}^N$, and there
    exists a constant $C > 0$ such that $\normH{g^{n-1}_{h\tau}} \leq
    C I(u^{n-1}_{h\tau})^{1/2} + k^{n-1}_{h\tau}$ where
    $k^{n-1}_{h\tau} \in L^p(\Omega)$ for some $p \geq 2$.
  \end{itemize}
  Then
  \begin{eqnarray*}
    \lefteqn{
      \norm{\max_{1 \leq n \leq N} I(u^n_{h\tau})^{1/2}}_{L^p(\Omega)}
      + \norm{a_{h\tau}}_{L^{pq/2}(\Omega, \LqU)}^{q/2}
      + \bbE \left[ \left( \sum_{n=1}^N 
        \normH{u_{h\tau}^n
          - u^{n-1}_{h\tau}}^2 \right)^{p/2} \right]^{1/p} } \\
    &\leq& 
    C(p,T)  \Bigl(1+C T/N\vph\Bigr)^{N/p} \left(
      \norm{I(u^0_{h\tau})^{1/2}}_{L^p(\Omega)}
      + \norm{f_{h\tau}}_{L^{pq'/2}(\Omega, \LqpUp)}^{q'/2}
      + \norm{k_{h\tau}}_{L^p((0,T) \times \Omega)} \right).
  \end{eqnarray*}
\end{lemma}

\begin{proof} (Sketch) Starting from \eqnref{:bound2}, and 
  upon neglecting the dependence of $g$ upon $u$ the estimate
  \begin{eqnarray*}
    \lefteqn{
      \norm{\max_{1 \leq n \leq N} I(u^n_{h\tau})^{1/2}}_{L^p(\Omega)}
      + \norm{a_{h\tau}}^{q/2}_{L^{pq/2}(\Omega, \LqU)}
      + \bbE \left[ \left( \sum_{n=1}^N 
          \normH{u^n_{h\tau} - u^{n-1}_{h\tau}}^2 \right)^{p/2} \right]^{1/p}
      \nonumber } \\
    &\leq& 
    C(p) \left(
      \norm{I(u^0_{h\tau}})^{1/2}_{L^p(\Omega)}
      + \norm{f_\tau}^{q'/2}_{L^{pq'/2}(\Omega, \LqpUp)}
      + T^{1/2-1/p} \norm{g_\tau}_{L^p(\Omega, \LpH)} \right)
  \end{eqnarray*}
  follows {\em mutatis mutandis} as in the proof of Lemma
  \ref{lem:bound}.  Bounding the last term as
  \begin{eqnarray*}
    \norm{g_{h\tau}}_{L^p(\Omega, \LpH)}
    &\leq& \norm{k_{h\tau}}_{L^p((0,T) \times \Omega)} 
    + C \left(\sum_{n=1}^N 
      \tau \norm{I(u^{n-1}_{h\tau})^{1/2}}_{L^p(\Omega)}^p \right)^{1/p} \\
    &\leq& \norm{k_{h\tau}}_{L^p((0,T) \times \Omega)} 
    + C \left(\sum_{n=0}^{N-1} \tau 
      \norm{\max_{0 \leq m \leq n}  
        I(u^m_{h\tau})^{1/2}}_{L^p(\Omega)}^p \right)^{1/p},
  \end{eqnarray*}
  and noting that the upper bound $N$ was arbitrary shows
  \begin{eqnarray*}
    \lefteqn{
      M^n
      + \norm{a_{h\tau}}_{L^{pq/2}(\Omega, L^q[0,t^n;U])}^{pq/2}
      + \bbE \left[ \left( \sum_{m=1}^n 
          \normH{u^m_{h\tau} - u^{m-1}_{h\tau}}^2 \right)^{p/2} \right]
      \nonumber } \\
    &\leq& 
    C(p,T) \left(
      M^0
      + \norm{f_{h\tau}}_{L^{pq'/2}(\Omega, L^{q'}[0,t^n;Up])}^{pq'/2}
      +  \norm{k_{h\tau}}_{L^p((0,t^n) \times \Omega)}^p
      + \sum_{m=0}^{n-1} \tau M^m
    \right),
  \end{eqnarray*}
  where $M^n \equiv \norm{\max_{0 \leq m \leq n}
    I(u^m_{h\tau})^{1/2}}_{L^p(\Omega)}^p$. The lemma now follows from the
  discrete Gronwall inequality.
\end{proof}

\subsubsection{Existence and Measurability of Solutions} \label{sec:measurableSoln}
Given $\omega \in \Omega$, solutions of the discrete problems will be
established using the following formulation of Brouwer's fixed point
theorem \cite[Proposition 2.1]{Sh97}.

\begin{theorem} \label{thm:brower}
  Let $\psi:\Re^M \rightarrow \Re^M$ be continuous and suppose that
  there exists $R > 0$ such that $\psi(\bfu).\bfu \geq 0$ whenever $|\bfu| =
  R$.  Then there exists $\bfu \in \Re^M$ with $|\bfu| \leq R$
  for which $\psi(\bfu) = 0$.
\end{theorem}

In the numerical context $\bfu$ is the vector of coefficients
representing the solution $u^n_{h\tau}(\omega) \in U_h$ for a given
basis of $U_h$, and at each time step $\psi$ will depend on the sample point
$\omega \in \Omega$ implicitly through the stochastic increment, data,
and the solution at the prior time step, i.e.,
$$
\psi(\omega,\bfu) \equiv \psi \left(\bfu;
u^{n-1}_{h\tau}(\omega), f^n_{h\tau}(\omega),
g^{n-1}_{h\tau}(\omega), \xi^n_{h\tau}(\omega) \right).$$ In all
instances the dependence of $\psi$ upon $\omega$ will be
${\mathcal F}^n$-measurable,
and in this situation the following lemma shows that it is
possible to select an ${\mathcal F}^n$-measurable  solution of
$\psi(\omega, \bfu) = 0$ for every $\omega \in \Omega$.

\begin{lemma}\label{lem:measurable}
  Let $(\Omega, \calF)$ be a measurable space,
  $\psi:\Omega\times\Bbb R^M \to \Re^M$ be a mapping for which
    \begin{itemize}
    \item $\omega\mapsto \psi(\omega,\bfu)$ is $\mathcal F$-measurable
      for every $\bfu \in\Bbb R^d$.
      
    \item $\bfu \mapsto \psi(\omega,\bfu)$ is continuous for every
      $\omega\in\Omega$.
      
    \item For every $\omega\in\Omega$, there exists $\bfu \in\Bbb R^d$
      such that $\psi(\omega,\bfu)=0$.
    \end{itemize}
    Then there exists an $\mathcal F$-measurable mapping
    $\bfu:\Omega\to\Bbb R^d$ such that $\psi(\omega,\bfu(\omega))=0$
    holds for every $\omega\in\Omega$.
\end{lemma}

Results of this form e.g. appear in \cite{GTW2017,D1} and are obtained using the
following lemma from \cite{KN_65}.

\begin{lemma}[Kuratowski and Ryll-Nardzewski \cite{KN_65}]
  \label{lem:KurRyll} Let $(\Omega,\mathcal F)$ be a measurable space,
  $Y$ a complete, separable metric space, and for every
  $\omega\in\Omega$ let $F(\omega)$ be a non-empty closed set in $Y$
  such that
  \begin{equation}\label{condition}
    \{\omega\in\Omega:F(\omega)\cap G\ne\emptyset\}\in\mathcal F
  \end{equation}
  holds for every open set $G$ in $Y$. Then there exists an $\mathcal
  F$-measurable mapping $\zeta:\Omega\to Y$ such that
  $\zeta(\omega)\in F(\omega)$ holds for every $\omega\in\Omega$.
\end{lemma}

\begin{remark}
  The hypothesis \eqref{condition}  holds if
  $$
  \{\omega\in\Omega:F(\omega)\cap B\ne\emptyset\}\in\mathcal F
  $$
  for every closed ball $B$ in $Y$ since every open set $G$ in a
  separable metric space is a countable union of closed
  balls.
\end{remark}

\begin{proof}(of Lemma \ref{lem:measurable})
  Define $F(\omega)=\{\bfu \in\Bbb R^d: |\psi(\omega, \bfu)|=0\}$ for
  $\omega\in\Omega$. Then $F(\omega)$ is non-empty and closed and
  $$
  \{\omega\in\Omega:F(\omega)\cap B\ne\emptyset\}
  =\{\omega\in\Omega:
  \inf_{\bfu \in B} |\psi(\omega, \bfu)|=0\}\in\mathcal F
  $$
  holds for every closed ball $B$ in $\Bbb R^d$ as
  $\omega\mapsto\inf_{\bfu \in B} |\psi(\omega, \bfu)|$ is $\mathcal
  F$-measurable. The existence of a measurable solution of
  $\psi(\omega, \bfu(\omega))$ then follows from the
  Kuratowski Ryll-Nardzewski Lemma \ref{lem:KurRyll}.
\end{proof}

The proof of the Kuratowski Ryll-Nardzewski Lemma is not constructive
so it is not clear that the computed solutions are measurable. If the
time step $\tau$ is sufficiently small, solutions of the nonlinear
problem can often be established using the Banach fixed point theorem. In this situation solutions depend continuously upon
the data, and hence are measurable; however, usually the bound on the
time step is prohibitively small and fixed point iterations converge
slowly, so a (quasi) Newton method is employed. If, for every $\omega
\in \Omega$, convergence is achieved for a bounded number of
iterations, the solution would depend continuously upon the data, and
measurability would follow.

\subsection{Stochastic Navier-Stokes Equation} \label{ex:ns} The
strong form of the incompressible stochastic Navier-Stokes equations
on a bounded Lipschitz domain $D \subset \Re^3$ takes the
form
\begin{gather}\label{nse-1}
du 
+\left((u.\nabla)u - D(u) + \nabla p \vph\right) \, dt 
= f \, dt + g(u) \, dW, \\
div(u) = 0, \nonumber
\end{gather}
with initial and boundary conditions 
$$
u \bigl\vert_{t=0} = u^0, \qquad u \bigl\vert_{\partial D} = 0,
$$
and $W$ an ${\mathbb R}$-valued Wiener process.
Here $u$ is the vector-valued velocity of the fluid, $p$ the pressure,
and $f$ and $g$ are vector-valued and $D(u)$ is the symmetric
part of the gradient as in \eqnref{:Ans}. In the above 
\begin{equation} \label{eqn:gNse}
g(u)(t,x,\omega) = \gamma\left(t,x, u(t,x,\omega) \vph\right)
\end{equation}
where $\gamma:(0,T) \times D \times \Re \rightarrow \Re^d$ is
Caratheodory with linear growth. That is, for $u \in \Re$ fixed 
$(t,x) \mapsto \gamma(t,x,u)$ is
measurable, and for $(t,x) \in (0,T) \times D$ fixed
$u \mapsto \gamma(t,x,u)$ is 
continuous, and $|\gamma(t,x,u)| \leq C |u| + k(t,x)$ where
$k \in L^p[0,T; \Ltwo]$ with $p > 4$.

To pose these equations in the abstract setting introduced in Section
\ref{sec:introSpde} let
\begin{equation} \label{eqn:nsSpaces}
  U = \Honeo^3, \qquad
  H = \Ltwo^3,  \qquad
  U_0 = \{u \in U: \ div(u) = 0 \ a.e.~in \ D\},
\end{equation}
and consider the weak statement of (\ref{nse-1}) for which $u$ takes
values in $L^2[0,T;U_0]$ and satisfies
$$
(u(t), v)_H
+ \int_0^t \left\{ 
  \left((u.\nabla) u, v \right) + \left(D(u), \nabla v\right)
  \vph \right\} \, ds
= (u^0, v)_H
+ \int_0^t (f,v) \, ds + \int_0^t (g(u),v)_H \, dW,
\quad
v \in U_0,
$$
where $(.,.)$ denotes an $L^2$ pairing on $D$.  Restricting
the test functions to be in the space of divergence-free functions
eliminates the pressure which is necessary since even in the
deterministic setting the temporal regularity of $p$ is very low
\cite{Li96a}.

To motivate the numerical scheme, recall that in the deterministic
setting the natural stability estimate is found upon taking the dot
product of the equation (\ref{nse-1}) with the solution and
integrating by parts to obtain
$$
\half \dbyd{}{t} \normH{u}^2 
+ \left((u.\nabla)u, u \vph\right) + \normU{u}^2 = (f,u),
$$
where $\normU{u} \equiv \ltwo{D(u)}$ is equivalent to the usual
norm on $U$.  The key step is to observe that the cubic term (which
for large data could not be dominated by the quadratic terms) is skew
symmetric; specifically, integration by parts shows
\begin{equation} \label{eqn:skewSym}
 \left((u.\nabla)u, v \vph\right) 
= -\left(u, (u.\nabla) v \vph\right) 
+ \left(div(u)\, u, v \vph\right),
\qquad u, v \in U.
\end{equation}
It follows that $\left((u.\nabla) u, u \right) = 0$ when $u \in U_0$, so
bounds upon the solution in $\LinfH \cap \LtwoU$ follow as for
the heat equation.

In general, it is difficult to construct subspaces of the
divergence-free space $U_0$ with good approximation properties, so in
a numerical context a velocity and pressure pair are constructed;
$(u_h, p_h) \in V_h \times P_h$ with $V_h \subset U$ and $P_h \subset
\Ltwoo$.  The divergence-free condition is then approximated by
requiring $u_h$ to take values in the ``discretely divergence-free
subspace'' $U_h \subset V_h$ defined by
\begin{equation}\label{discretelydiv1}
U_h = \{u_h \in V_h \sst \left(div(u_h), q_h \right)=0, \,\, q_h \in P_h \}.
\end{equation}
Note that $U_h \not\subset U_0$, and in order to guarantee that
functions $u \in U_0$ can be well-approximated by functions $u_h \in
U_h$ the pair $(V_h, P_h)$ is required to satisfy the discrete inf--sup
(Ladyzhenskaya-Babuska-Brezzi) condition \cite{BrFo91}: there
exists a constant $c > 0$ independent of $h$ such that
\begin{equation} \label{eqn:infsup}
\sup_{v_h \in V_h} 
\frac{ \left(p_h, div(v_h) \vph \right)}{\ltwo{\nabla v_h}}
\geq c \ltwoo{p_h},
\qquad p_h \in P_h.
\end{equation}
We now come back to (\ref{nse-1}), and a corresponding discretization.
Letting $\tau = T/N$ be a time step and $\{\xi^n_{\tau}\}_{n=1}^N$ be
stochastic increments, for each $n = 1,2,\ldots, N$ and all $\omega
\in \Omega$ we let $\bigl(u^n_{h\tau}({\omega}), p^n_{h\tau}({\omega})
\bigr) \in V_h \times P_h$ satisfy
\begin{multline}   \label{eqn:nsh}
(u^n_{h\tau} - u^{n-1}_{h\tau},v_h) 
+ (\tau/2) \left((u^n_{h\tau} . \nabla)u^n_{h\tau}, v_h \vph \right)
- (\tau/2) \left(u_{h\tau}^n, (u^n_{h\tau}.\nabla), v_h \vph\right) \\ 
+ \tau \left(D(u^n_{h\tau}), \nabla v_h \vph\right) 
- \tau (p^n_{h\tau}, div\, v_h) 
= \tau (f^n_{h\tau}, v_h) 
+ \bigl( g^{n-1}_{h\tau}, v_h\bigr)\xi^n_{\tau},
\quad v_h \in  V_h,
\end{multline}
$$
(div(u^n_{h\tau}), q_h) = 0, \qquad q_h \in P_h,
$$
where
\begin{equation} \label{eqn:ghtauNs}
g^{n-1}_{h\tau}(x,\omega)
= \frac{1}{\tau} \int_{t^{n-1}}^{t^n} 
\gamma\left(t,x, u^{n-1}_{h\tau}(x,\omega) \vph\right) \, dt.
\end{equation}
The second equation is simply the requirement that
$u^n_{h\tau}(\omega) \in U_h$, and it is immediate that the term
involving the pressure vanishes when $v_h \in U_h$. Equation
\eqnref{:skewSym} was used to formulate an approximation of the
convective derivative that is skew symmetric when
$div\bigl(u^n_{h\tau}\bigr)$ may not vanish. Note too that the
convective derivative $(u^n_{h\tau}.\nabla)$ could be lagged to
$(u^{n-1}_{h\tau}.\nabla)$ to give a linearly implicit scheme.

The methodology introduced in Section \ref{sec:measurableSoln} is used
to establish a (measurable) solution of the discrete scheme
\eqnref{:nsh}.  Given a basis for $U_h$, an element $u_h(\omega) \in
U_h$ is identified with an element ${\bf u}(\omega)$ of
$\Re^M$ where $M = dim(U_h)$. With $\omega
\in \Omega$ fixed, and identifying an element $v_h \in U_h$ with a
vector of coefficients ${\bf v} \in {\mathbb R}^M$, the Riesz theorem
is used to construct $\psi:\Re^M \rightarrow \Re^M$ satisfying
\begin{multline*}\nonumber
\psi({\bf u}).{\bf v} := 
(u_h - u^{n-1}_{h\tau},v_h) 
+ (\tau/2) \left((u_h . \nabla)u_h, v_h \vph \right)
- (\tau/2) \left(u_h, (u_h.\nabla), v_h \vph\right) \\ 
+ \tau \left(D(\nabla u_h), \nabla v_h \vph\right) 
- \tau (f^n_{h\tau}, v_h) 
- \bigl( g^{n-1}_{h\tau}, v_h\bigr)\xi^n_{\tau},
\qquad {\bf v} \in {\mathbb R}^M.
\end{multline*}
Fixing $\omega \in \Omega$, and setting $v_h = u_h(\omega)$ then leads to
\begin{eqnarray*}
  \psi ({\bf u}).{\bf u}
  &=& \half \left( \normH{u_h}^2 
    + \normH{u_h - u^{n-1}_{h\tau}}^2
    - \normH{u^{n-1}_{h\tau}}^2 \right)
  + \tau \normU{u_h}^2 
  - \tau (f^n_{h\tau}, u_h) 
  - \bigl( g^{n-1}_{h\tau}, u_h\bigr) \xi^n_{\tau} \\
  &\geq& 
  \half \left( \normH{u_h}^2 
    - \normH{u^{n-1}_{h\tau}}^2 \right)
  + \tau \normU{u_h}^2 
  - \left(\tau \normUp{f^n_{h\tau}} 
    + \normUp{g^{n-1}_{h\tau}} |\xi^n_{h\tau}|
  \right) \normU{u_h},
\end{eqnarray*}
and it is clear that this is non--negative whenever
$$
\min\left(\normH{u_h}, \normU{u_h} \vph\right) 
\geq \max\left(\normH{u^{n-1}_{h\tau}}, \normUp{f^n_{h\tau}} 
    + \normUp{g^{n-1}_{h\tau}} |\xi^n_{h\tau}|/\tau \vph\right).
$$
Existence of a pressure then follows from the inf--sup condition.

To bound the solutions, set $v_h = u^n_{h\tau}(\omega)$ in the discrete
weak statement \eqnref{:nsh} and complete the square 
(as in \eqnref{:square}) to get
\begin{equation} \label{eqn:uhns}
(1/2) \normH{u^n_{h\tau}}^2 
+ (1/2) \normH{u_{h\tau}^n-u_{h\tau}^{n-1}}^2 
+ \tau \normU{u_{h\tau}^n}^2
= (1/2) \normH{u_{h\tau}^{n-1}}^2 
+\tau ( f^n_{h\tau} , u_{h\tau}^n )
+ (g^{n-1}_{h\tau}, u_{h\tau}^n) \xi^n_{\tau},
\end{equation}
which, due to the skew symmetry of the nonlinear term, is identical in
form to the corresponding equation \eqnref{:energyEst} for the heat
equation.

The following theorem establishes convergence of solutions to the
numerical scheme \eqnref{:nsh} to a weak martingale solution of the
stochastic Navier-Stokes equation (\ref{nse-1}).

\begin{theorem}
  Fix $T>0$, and let $D \subset \Re^3$ be a bounded Lipschitz domain.
  Let $U = \Honeo^3$, $H = \Ltwo^3$, $U_0 \subset U$ be the divergence-free
  subspace, and let $(\Omega, \calF, \bbP)$ be a probability
  space.  Let Assumptions \ref{ass:spde1}, and \ref{ass:spde0} hold
  with parameter $2p > 4$.  Let $\tau = T/N$ with $N \in \bbN$ denote
  a time step, and let $\{(V_h, P_h)\}_{h>0} \subset U \times \Ltwoo$
  be finite-dimensional subspaces satisfying:
  \begin{itemize}
  \item For each $(v,q) \in U \times \Ltwoo$ there exists a sequence
    $\{(v_h,q_h)\}_{h>0}$ with $(v_h,q_h) \in U_h \times P_h$
    such that $(v_h,q_h) \rightarrow (v,q)$ as $h \rightarrow 0$.

  \item The restriction of the orthogonal projection $Q_h:H \rightarrow
    V_h$ to $U$ is stable. That is, there exists $C > 0$ independent
    of $h$ such that $\normU{Q_h u} \leq C \normU{u}$.
      
  \item The discrete inf--sup condition \eqnref{:infsup} holds with a constant
    $c>0$ independent of $h > 0$. Denote the discretely divergence-free subspace
    by $U_h = \{u_h \in V_h \sst (div(u_h), q_h) = 0, \,\, q_h \in P_h\}$.
  \end{itemize}
  Let
  $\{u_{h\tau}\}_{h,\tau > 0}$ be a sequence of solutions of
  \eqnref{:nsh} with data satisfying
  \begin{enumerate}
  \item $\{u^0_{h\tau}\}_{h,\tau > 0}$ is bounded in $L^{2p}(\Omega,H)$, and
    converges in $L^2(\Omega, H)$.

  \item $\{f_{h\tau}\}_{h,\tau > 0}$ is bounded in $L^{2p}(\Omega,
    \LtwoUp)$ converges in $L^2(\Omega, \LtwoUp)$.

  \item $g_{h\tau}$ is given by equation \eqnref{:ghtauNs} with
      $\gamma:(0,T) \times D \times \Re \rightarrow \Re^d$ Caratheodory
      with linear growth; $|\gamma(t,x,u)| \leq C |u| + k(t,x)$
      with $k \in L^{2p}[0,T;\Ltwo]$.
  \end{enumerate}
  Denote the discrete Wiener process with increments
  $\{\xi_\tau^n\}_{n=1}^N$ by $\What_{\tau}$, and write
  $$
  (A(u), v) 
  = \left(D(u), \nabla v \vph\right) 
  + (1/2) \left((u . \nabla)u, v \vph \right)
  - (1/2) \left(u, (u.\nabla), v \vph\right).
  $$
  The there exist a probability space $(\Omegat, \calFt, \bbPt)$ and a
  random variable $(u,(f,a),g,W)$ on $\Omegat$ with values in $(\bbX,
  \calB(\bbX))$ with
  $$
  \bbX = G[0,T;U'] \cap L^2[0,T;U]_{weak} 
  \times (L^{4/3}[0,T; U'] \times L^{4/3}[0,T; U']_{weak})
  \times \LtwoH 
  \times C[0,T],
  $$
  and a subsequence $(\tau_k,h_k) \rightarrow (0,0)$ for which the
  $$
  \calL(u_{h_k\tau_k}, (f_{h_k\tau_k}, A(u^n_{h\tau})), 
  g_{h_k\tau_k}, \What_{\tau_k})
  \ \Rightarrow \ \calL(u,(f,a),g,W) \equiv \bbPt.
  $$
  In addition, $\bbPt[div(u)=0] = 1$, $\bbPt[a = A(u)] = 1$, and there
  exists a filtration $\{\calFt(t)\}_{0 \leq t \leq T}$ satisfying the
  usual conditions for which $(u,f,g,W)$ is adapted and $W$ is a
  real-valued Wiener process for which
  $$
  (u(t), v)_H  
  + \int_0^t \left( (u.\nabla)u,v) + (D(u), \nabla v) \vph\right) \, ds
  = (u^0,v)_H 
  + \int_0^t ( f, v ) \, ds
  + \int_0^t {( g(u), v )} \, dW,
  \quad v \in U_0,
  $$
  where $g(u)(t,x,\omega) = \gamma(t,x, u(t,x,\omega))$.
\end{theorem}

\begin{proof}
  Lemma \ref{lem:bound2} is first used to bound the solutions of the
  numerical scheme. Equation \eqnref{:uhns} establishes the bounds
  needed at each time step, and using the structural properties
  of $\gamma$ give
  \begin{equation} \label{eqn:gNs}
    \ltwo{g^{n-1}_{h\tau}}
    \leq C \left( \ltwo{u^{n-1}_{h\tau}} + k^n_\tau \right)
    \quad \text{ where } \quad
    k^n_\tau = (1/\sqrt{\tau}) \norm{k}_{L^2[t^{n-1},t^n; \Ltwo]}.
  \end{equation}
  Lemma \ref{lem:bound2} with parameters $2p$ and $q=q'=2$ then shows
  \begin{multline*}
    \norm{u_{h\tau}}_{L^{2p}(\Omega, L^\infty[0,T;H])}
    + \norm{u_{h\tau}}_{L^{2p}(\Omega, L^2[0,T;U])} \\
    \leq C \left( \norm{u^0_{h\tau}}_{L^{2p}(\Omega,H)}
      + \norm{f_{h\tau}}_{L^{2p}(\Omega,\LtwoU')}
      + \norm{k}_{L^{2p}[0,T;\Ltwo]}
    \right).
  \end{multline*}

  We now verify that $\{(u_{h\tau}, (f_{h\tau}, A(u_{h\tau})),
  g_{h\tau}, \What_{\tau})\}_{h,\tau > 0}$ satisfies the hypotheses of
  Theorem \ref{thm:main} with parameters $r=2$, $q=8$, and $q'=8/7$.
  \begin{enumerate}
  \item 
    The embedding $\Hone \embed L^6(D)$ is first used to verify
    $\norm{u}_{L^3(D)} \leq C \normH{u}^{1/2} \normU{u}^{1/2}$. Then
    $$
    |(A(u), v)|
    \leq \left( \normU{u} \right) \normU{v}
    + (1/2) \norm{u}_{L^3(D)} \normU{u} \norm{v}_{L^6(D)} 
    + (1/2) \norm{u}_{L^3(D)} \norm{u}_{L^6(D)} \normU{v}, 
    $$
    so that
    $$
    \normUp{A(u)} \leq
    \normU{u} + C \normH{u}^{1/2} \normU{u}^{3/2}.
    $$   
    Repeated application of H\"older's inequality then shows
    $$
    \norm{A(u_{h\tau})}_{L^p(\Omega, L^{4/3}[0,T;U'])} 
    \leq
    \norm{u_{h\tau}}_{L^p(\Omega, L^{4/3}[0,T;U])} 
    + C \norm{u_{h\tau}}_{L^{2p}(\Omega, L^\infty[0,T;H])}^{1/2}
    \norm{u_{h\tau}}_{L^{2p}(\Omega, L^2[0,T;U])}^{3/2}.
    $$
    The bounds upon $u_{h\tau}$ and embedding $L^{4/3}[0,T;U']
    \embed L^{8/7}[0,T;U']$ then show
    $\norm{A(u_{h\tau})}_{L^p(\Omega, L^{8/7}[0,T;U'])}$ is also
    bounded.

  \item Writing $F^n_{h\tau} = f^n_{h\tau} - A(u^n_{h\tau})$ we have
    $(F^n_{h\tau}, u^n_{h\tau}) = (f^n_{h\tau},
    u^n_{h\tau}) - \normU{u^n_{h\tau}}^2$, and the
    Cauchy-Schwarz inequality gives
    $$
    \norm{(F^n_{h\tau}, u^n_{h\tau})}_{L^{p/2}(\Omega, L^1(0,T))}
    \leq \norm{f_{h\tau}}_{L^p(\Omega,\LtwoU')} 
    \norm{u_{h\tau}}_{L^p(\Omega,\LtwoU)} 
    + \norm{u_{h\tau}}_{L^p(\Omega,\LtwoU)}^2.
    $$

  \item Equation \eqnref{:gNs}, and the bounds upon
    $\{u_{h\tau}\}_{h,\tau > 0}$ show that $\{g_{h\tau}\}_{h,\tau >
      0}$ will be bounded in $L^{2p}(\Omega,L^{2p}[0,T;H])$ provided
    $\{k_\tau\}_{\tau > 0}$ is bounded in $L^{2p}(0,T)$ (note that
    $k$ is deterministic). This follows from repeated applications
    of H\"older's inequality,
    $$
    \norm{k_\tau}_{L^{2p}(0,T)}^{2p}
    = \sum_{n=1}^N \tau \left( \frac{1}{\tau}
      \int_{t^{n-1}}^{t^n} \normH{k(t)}^2 \right)^p
    \leq \sum_{n=1}^N \int_{t^{n-1}}^{t^n} \normH{k(t)}^{2p} \, dt
    = \norm{k}_{L^{2p}[0,T;H]}^{2p}.
    $$
    
  \item The initial data $u^0$ satisfies the
    properties assumed in Theorem \ref{thm:main} by hypothesis.
  \end{enumerate}
  It follows that upon passing to a sub--sequence $(h_k, \tau_k)
  \rightarrow 0$ there exist a filtered probability space $(\Omegat,
  \calFt, \{\calFt(t)\}_{t=0}^T, \bbPt)$ and a random variable
  $(u,F,g,W)$ with values in $\bbX$ for which $W$ is a standard Wiener
  process, $\calL(u_{h_k \tau_k},(f_{h_k \tau_k}, A(u_{h_k \tau_k})),g_{h_k
    \tau_k},\What_{h_k \tau_k}) \Rightarrow \calL(u,(f.a),g,W) \equiv
  \bbPt$, and
  $$
  (u(t),v)_H = (u^0,v)_H + \int_0^t (f-a,v) \, ds 
  + \int_0^t (g,v)_H \, dW,
  \qquad v \in U_0, \quad 0 \leq t \leq T.
  $$
  For $q \in L^2[0,T; \Ltwo]$ fixed, the function
  $$
  (u,(f,a),g,W) \mapsto \left|\int_0^T (div(u), q)\, ds\right| \wedge 1
  $$
  is continuous and bounded on $\bbX$. Letting $q_k \in L^2[0,T; P_{h_k}]$ be
  chosen so that $q_k \rightarrow q$ it follows that
  $$
  \bbEt\left[\left|\int_0^T (div(u), q)\right| \wedge 1 \right]
  = \lim_{k \rightarrow \infty}
  \bbE\left[\left|\int_0^T (div(u_{h_k \tau_k}), q)\right| \wedge 1 \right]
  = \lim_{k \rightarrow \infty}  
  \bbE\left[\left|\int_0^T (div(u_{h_k \tau_k}), q_k)\right| \wedge 1 \right]
  = 0,
  $$
  whence $\bbPt[div(u)=0] = 1$. Example \ref{eg:2} shows that $(a,v) =
  (D(u),\nabla v) + (1/2)((u.\nabla)u, v) - (1/2)(u,(u.\nabla)v)$
  almost surely on the support of $\bbPt$.

  To verify that $g(t,x,\omega) = \gamma(t,x, u(t,x,\omega)$ (we write
  $g = \gamma(u)$) on the support of $\bbPt$, note that the map
  $u(t,x) \mapsto \gamma(t,x,u(t,x))$ is continuous from $\LtwoH$ to
  itself, so if $v \in \LtwoH$ is fixed
  $$
  \bbEt\left[\left|(\gamma(u)-g, v)_{\LtwoH} \right| \wedge 1 \vph\right]
  = \lim_{k \rightarrow \infty} 
  \bbE\left[\left| 
      (\gamma(u_{h_k\tau_k})-g_{h_k,\tau_k}, v)_{\LtwoH} \right|
    \wedge 1 \vph\right].
  $$
  For the numerical scheme $g_{h\tau}(\omega)$ is the orthogonal
  projection of $\gamma(u_{h\tau}(\omega))$ onto the subspace of
  functions in $\LtwoH$ which are piecewise constant in time. Thus if
  $v_k \in \LtwoH$ is piecewise constant in time and $v_k \rightarrow
  v$ we have
  $$
  \bbEt\left[\left|(\gamma(u)-g, v)_{\LtwoH} \right| \wedge 1 \vph\right]
  = \lim_{k \rightarrow \infty} 
  \bbE\left[\left| 
      (\gamma(u_{h_k\tau_k})-g_{h_k,\tau_k}, v_k)_{\LtwoH} \right|
    \wedge 1 \vph\right]
  = 0.
  $$
\end{proof}

\subsection{Harmonic Heat Flow} \label{sec:harmonic}

The stochastic harmonic heat flow equation on a domain 
$D \subset \Re^3$ is the vector-valued equation
$$
du + (- \Delta u + \lambda u) \, dt
= f \, dt + (u \times g) \circ dW,
\quad \text{ with constraint } \quad u \in \bbS^2,
$$
and initial and boundary data $u \vert_{t=0} = u^0$ and
$\partial u / \partial n \vert_{\partial D} = 0$. Here $\lambda$ is
a Lagrange multiplier dual to the constraint $|u|=1$, and
$$
(u \times g) \circ \, dW 
\equiv (1/2) (u \times g) \times g \, dt + (u \times g) \, dW
$$
denotes the Stratonovich integral. In order to preserve the constraint
the noise term is selected to be tangent to $u \in \bbS^2$, and in
order to eliminate a significant amount of technical overhead we will
assume that the datum $g(x) = \gamma \in \Re^3$ is independent of $x
\in D$. The numerical analysis of the spatially dependent data (and
operator-valued colored noise) is undertaken in \cite{BBNP1} for the
stochastic Landau-Lifshitz-Gilbert equation.

The analysis of the harmonic heat flow equation is complicated by the
fact that solutions may exhibit singularities.  In this situation
essentially nothing is known about the structure of the Lagrange
multiplier, and this gap in the theory plagues  both the
construction and analysis of numerical schemes.  For this
reason the constraint is usually approximated using a penalty scheme
and this is the approach considered here. Specifically, we consider
numerical approximations of the equation
\begin{equation}\label{eqn:harmonicHeat}
du + \left(- \Delta u + D\phi(u) \vph\right) \, dt 
= f \, dt + (u \times \gamma) \circ dW,
\end{equation}
where $\phi(u) = (1/2\epsilon)(|u|^2-1)^2$ with $\epsilon >0$. The drift
term on the left is the variational derivative of the energy
$$
I(u) = \int_D \half |\nabla u|^2\, dx + \phi(u),
$$
and in the deterministic case bounds upon the solution
independent of the penalty constant $\epsilon$ follow upon taking
the product of the equation with either $u_t$ or $-\Delta u + D\phi(u)$
to obtain
$$
\ltwo{u_t}^2 + \dbyd{}{t} I(u) = (f,u_t),
\quad \text{ or } \quad
\dbyd{}{t} I(u) + \ltwo{-\Delta u + D\phi(u)}^2
= \left(f, -\Delta u + D\phi(u) \right).
$$
When the stochastic term is present, we derive an analog of the second
estimate. However, in a numerical context where $u_h(\omega) \in U_h
\subset U \equiv \Hone^3$, the function $-\Delta u_h +
\phi(u_h) \not\in U_h$ is not available as a test function. For this
reason we will use a mixed method where $a \equiv -\Delta u +
D\phi(u)$ is introduced as an additional variable.  Letting $\tau =
T/N$ with $N \in \bbN$ be a time step and $f^n_{h\tau} \simeq f(n\tau)$, 
we approximate solutions of \eqnref{:harmonicHeat} by
$\bigl(u_{h\tau}^n(\omega), a_{h\tau}^n (\omega) \bigr) \in U_h \times U_h$,
\begin{gather} 
(u_{h\tau}^n - u_{h\tau}^{n-1},v_h) + \tau (a^n_{h\tau},v_h) 
= \tau (f^n_{h\tau}, v_h)
+ \left(u_{h\tau}^{n-1/2} \times \gamma ,v_h \vph\right) \xi^n_{\tau}
\label{eqn:harmonicHeatU} \\
(a^n_{h\tau}, b_h) = (\nabla u_{h\tau}^n, \nabla b_h) 
+ (1/\epsilon) \left( (|u_{h\tau}^n|^2 + |u_{h\tau}^{n-1}|^2 - 2)
  u_{h\tau}^{n-1/2}, b_h \right), 
\label{eqn:harmonicHeatA}
\end{gather}
for all $(v_h, b_h) \in U_h \times U_h$, where $u_{h\tau}^{n-1/2} \equiv
(1/2) (u^{n} _{h\tau}+ u^{n-1}_{h\tau})$ and $\xi^n_{\tau}$ are stochastic increments
satisfying Assumption \ref{ass:spde0}. This scheme was constructed so
that:
\begin{itemize}
\item The approximation of $D\phi(u) = (2/\epsilon)(|u|^2-1) u$ in
  \eqnref{:harmonicHeatA} inherits a discrete version of the identity
  $(D\phi(u), u_t) = d \phi/ dt$,
  $$
  (1/\epsilon) \left( (|u_{h\tau}^n|^2 + |u_{h\tau}^{n-1}|^2 - 2)
    u_{h\tau}^{n-1/2}, u^n_{h\tau} - u^{n-1}_{h\tau} \right) 
  = \phi(u^n_{h\tau}) - \phi(u^{n-1}_{h\tau}).
  $$
  This is essential in order to obtain bounds independent of $\epsilon$.

\item Since $D\phi(u)$ is parallel to $u$ it follows that
  $(D\phi(u) , u \times \gamma) = 0$. The discrete approximation
  of $D\phi(u)$ is parallel to $u^{n-1/2}_{h\tau}$ and is perpendicular
  to the coefficient $u^{n-1/2}_{h\tau} \times \gamma$ of $\xi^n_{\tau}$.

  Note too that $u^{n-1/2}_{h\tau}(\omega) \times \gamma \in U_h$, so
  is admissible as test function, and $(\nabla u , \nabla(u \times
  \gamma)) = 0$ when $u \in U$; both following since $g(x) = \gamma
  \in \Re^3$ was taken to be independent of $x$.
\end{itemize}
Selecting the test functions in
\eqnref{:harmonicHeatU}-\eqnref{:harmonicHeatA} to be 
$$
(v_h,b_h) =
\left(a_{h\tau}^n, u_{h\tau}^n -u_{n\tau}^{n-1} + (u_{n\tau}^{n-1/2}
\times \gamma) \xi_{\tau}^n \right)
$$ 
and using these structural properties shows
\begin{eqnarray}
\lefteqn{
\frac{1}{2} \ltwo{\nabla u_{h\tau}^n}^2 + \lone{\phi(u_{h\tau}^n)}
+ \frac{1}{2} \ltwo{\nabla (u_{h\tau}^n - u^{n-1}_{h\tau})}^2 
+ \tau \ltwo{a_{h\tau}^n}^2 }
\label{eqn:harmonicBound} \\
&=&
\frac{1}{2} \ltwo{\nabla u_{h\tau}^{n-1}}^2 + \lone{\phi(u_{h\tau}^{n-1})}
+ \tau \left(f_{h\tau}^n, a^n_{h\tau} \vph\right)
+ \left(\nabla u^n_{h\tau},
\nabla (u_{h\tau}^{n-1/2} \times \gamma)  \vph\right)\xi^n_{\tau} .
\nonumber
\end{eqnarray}
Lemma \ref{lem:bound2} with $U = H = \Ltwo^3$ then establishes bounds upon
the gradient of the solution independent of $\epsilon$; an additional
calculation then establishes a bound upon the (spatial) average of the
solution, and the Poincare inequality  then bounds the solution
itself.

\begin{lemma} \label{lem:harmonicHeatBound1}
  Let $D
  \subset \Re^3$ be a bounded Lipschitz domain, set $U = \Hone^3$
  and $H = \Ltwo^3$, and let $I:U \rightarrow \Re$ be the function $I(u) =
  (1/2)(\ltwo{\nabla u}^2 + \lone{\phi(u)})$ where $\phi(u)
  = (1/(2\epsilon))(|u|^2-1)^2$, with $\epsilon > 0$ fixed. Let 
   $(\Omega, \calF, \bbP)$ be a probability space.
   
   Suppose that the Assumptions \ref{ass:spde1}, and \ref{ass:spde0}
   with parameter $p>2$ hold, and that $\{u^0_{h\tau}\}$ is bounded in
   $L^p(\Omega,U)$, and $\{f_{h\tau}\}$ is
   bounded in $L^{p}(\Omega, \LtwoH)$.  Then there exists a sequence
   $\{ (u^n_{h\tau}, a_{h\tau}^n)\}_{n \geq 1}$ of $U_h \times
   U_h$-valued random variables adapted to $\{ {\mathcal
     F}^n\}_{n=0}^N$ which satisfy
   \eqnref{:harmonicHeatU}--\eqnref{:harmonicHeatA} and
  \begin{eqnarray} \nonumber
    {\rm (i)}&& \lefteqn{
      \bbE\left[ \max_{1 \leq n\leq N} 
        \left( \normH{\nabla u_{h\tau}^n}^p
          + \lone{\phi(u_{h\tau}^n)}^{p/2} \right)
        + \left(\sum_{n=1}^N 
          \normH{\nabla (u_{h\tau}^n - u^{n-1}_{h\tau})}^2 \right)^{p/2}
        + \left( \sum_{n=1}^N \tau 
          \normH{a_{h\tau}^n}^2 \right)^{p/2} 
      \right]^{1/p} } \\ \nonumber
     &&\hspace*{1.25in} \leq
    C \left(
      \norm{\nabla u^0_{h\tau}}_{L^p(\Omega,H)} 
      + \norm{\phi(u^0_{h\tau})}_{L^{p/2}(\Omega, \Lone)}^{1/2} 
      + \norm{f_{h\tau}}_{L^p(\Omega, \LtwoH)}
    \right). \\ \nonumber
    {\rm (ii)}&&
    \lefteqn{
      \bbE\left[
        \max_{1 \leq n\leq N} \normH{u_{h\tau}^n}^p
        + \left( 
          \sum_{n=1}^N \normH{u_{h\tau}^n - u^{n-1}_{h\tau}}^2 
        \right)^{p/2} \right]^{1/p} } \\ \nonumber
    &&\qquad \qquad \qquad \qquad \leq
    C \left(
      \norm{u^0_{h\tau}}_{L^p(\Omega,U)} 
      + \norm{\phi(u^0_{h\tau})}_{L^{p/2}(\Omega, \Lone)}^{1/2} 
      + \norm{f_{h\tau}}_{L^p(\Omega, \LtwoH)}
    \right).   
  \end{eqnarray}
\end{lemma}
\begin{proof}
  Theorem \ref{thm:brower} will be used to establish the existence of a
  solution to the scheme by solving for the variable $\delta u = u^n -
  u^{n-1}$.  Inductively assume that the $U_h$-valued random variable $u^{n-1}_{h\tau}$ is given and
  use the Riesz theorem to construct the solution operator $a_h:U_h
  \rightarrow U_h$ of equation \eqnref{:harmonicHeatA} with
  $u^n_{h\tau} = u^{n-1}_{h\tau} + \delta u$. Upon introducing a basis
  for $U_h$ the Riesz theorem on $\Re^M$ with $M = dim(U_h)$
  guarantees the existence of a continuous function $\psi:\Re^M
  \rightarrow \Re^M$ which, for each $\omega \in \Omega$, satisfies
  $$
  \psi(\delta \bfu).{\bf v}
  = \Big( \delta u
  + \tau a_h(\delta u) - \tau f^n_{h\tau}
  - \left((\delta u/2 + u^{n-1}_{h\tau}) \times \gamma \right)\xi^n_{\tau}, 
    v_h \Big)_H , 
    \qquad {\bf v} \in {\mathbb R}^M,
  $$
  where $\delta \bfu$, $\bfv \in {\mathbb R}^M$
  denote the vectors of coefficients of $U_h$-valued functions $\delta
  u$ and $v_h$.  Using equation \eqnref{:harmonicHeatA} we
  find
  \begin{eqnarray*}
    \bigl(a_h(\delta u), \delta u \bigr)_H
    &=& \left(\nabla (\delta u 
      + u^{n-1}_{h\tau}), \nabla \delta u \vph\right)_H
    + \lone{\phi(\delta u + u^{n-1}_{h\tau})}  
    - \lone{\phi(u^{n-1}_{h\tau})} \\
    &=& \half \left( \normH{\nabla \delta u}^2
      + \normH{\nabla (\delta u + u^{n-1}_{h\tau})}^2
      - \normH{\nabla u^{n-1}_{h\tau}}^2 \vph\right)\\
    &&+ \lone{\phi(\delta u + u^{n-1}_{h\tau})}  
    - \lone{\phi(u^{n-1}_{h\tau})}.
  \end{eqnarray*}
  Zeros of $\psi(.)$ then exist since $(\delta u \times \gamma, \delta
  u) = 0$, so
  \begin{eqnarray*}
    \psi(\delta \bfu).\delta \bfu
    &=& \normH{\delta u}^2 
    + \frac{\tau}{2} \normH{\nabla \delta u}^2
    + \frac{\tau}{2} \normH{\nabla (\delta u + u^{n-1}_{h\tau})}^2 
    + \tau \lone{\phi(\delta u + u^{n-1}_{h\tau})} \\
    && 
    - \left(u^{n-1}_{h\tau} \times \gamma, \delta u \right)_H \xi^n_{\tau}
    - \tau ( f^n_{h\tau}, \delta u) 
    - \frac{\tau}{2} \normH{\nabla u^{n-1}_{h\tau}}^2
    - \tau \lone{\phi(u^{n-1}_{h\tau})}
  \end{eqnarray*}
  is non--negative whenever $\ltwo{\delta u}^2 + \tau \ltwo{\nabla
    \delta u}^2$ is sufficiently large. Equation
  \eqnref{:harmonicBound} and the measurable selection theorem Lemma
  \ref{lem:measurable} then establish the hypothesis of Lemma
  \ref{lem:bound2} from which estimate (i) in the lemma follows.

  To establish estimate (ii), let $\ubar^n_{h\tau} = (1/|D|)
  \int_\Omega u^n_{h\tau}\, dx$ denote the spatial average.  Selecting
  the test function in \eqnref{:harmonicHeatU} to be $v_h =
  \ubar^{n-1/2}_{h\tau}$ and summing gives
  \begin{eqnarray*}
    |\ubar^n_{h\tau}|^2 
    &=&  |\ubar^0_{h\tau}|^2 
    + (1/|D|) 
    \sum_{m=1}^n \tau (f^m_{h\tau} - a^m_{h\tau}, \ubar^{m-1/2}_{h\tau}) \\
    &\leq&  |\ubar^0_{h\tau}|^2 
    + (1/|D|) \left(\sum_{m=1}^n 
      \tau \normH{f^m_{h\tau} - a^m_{h\tau}}^2 \right)^{1/2}
    \left(\sum_{m=1}^n \tau |D| \, (\ubar^{m-1/2}_{h\tau})^2 \right)^{1/2} \\
    &\leq& |\ubar^0_{h\tau}|^2 
    + \ltwoH{f_{h\tau} - a_{h\tau}}
    (T |D|)^{1/2} \max_{0 \leq m \leq n} |\ubar^m_{h\tau}|, 
    \qquad 1 \leq n \leq N.
  \end{eqnarray*}
  It readily follows that
  $$
  \norm{\max_{1 \leq m \leq N} |\ubar^m_{h\tau}|}_{L^p(\Omega)}
  \leq C(p,T/|D|) \left(
    \norm{\ubar^0_{h\tau}}_{L^p(\Omega)} 
    + \norm{f_{h\tau} - a_{h\tau}}_{L^p(\Omega,\LtwoH)}
  \right).
  $$
  Next, select the test function in \eqnref{:harmonicHeatU} to be 
  $v_h = \ubar^n_{h\tau}$ to obtain
  $$
  \half |\ubar^n_{h\tau}|^2 
  + \half |\ubar^n_{h\tau} - \ubar^{n-1}_{h\tau}|^2 
  + \frac{\tau}{|D|}  |\ubar^n_{h\tau}|^2
  =  \half |\ubar^{n-1}_{h\tau}|^2 
  + \frac{\tau}{|D|}
  \left(f^m_{h\tau} - a^m_{h\tau} 
    - \ubar^n_{h\tau}, \ubar^n_{h\tau} \vph\right)
  + \frac{1}{2|D|} \left(\ubar^n_{h\tau}, 
    \ubar^{n-1}_{h\tau} \times \gamma \vph \right) \xi^n_{\tau}.
  $$
  Lemma \ref{lem:bound2} with $H=U=\Re^3$ then shows
  \begin{eqnarray*}
    \lefteqn{
      \norm{\max_{1 \leq n \leq N} |\ubar^n_{h\tau}|}_{L^p(\Omega,H)}
      + \norm{\ubar_{h\tau}}_{L^p(\Omega, L^2(0,T))}
      + \bbE \left[ \left( \sum_{n=1}^N 
          |\ubar^n_{h\tau} - \ubar^{n-1}_{h\tau}|^2 \right)^{p/2} 
      \right]^{1/p} 
    } \\
    &\leq& 
    C(p,T) \left(
      \norm{\ubar^0_{h\tau}}_{L^p(\Omega,H)}
      + \norm{f_{h\tau} - a_{h\tau} - \ubar_{h\tau}}_{L^p(\Omega, \LtwoH)}
    \right).
  \end{eqnarray*}
  Estimate {\rm (ii)} in the lemma now follows from the bounds upon
  $a_{h\tau}$ and $\ubar_{h\tau}$ obtained above and the Poincare
  inequality.
\end{proof}

To cast the above scheme into the setting of Theorem \ref{thm:main}, set
\begin{equation} \label{eqn:harmonicHeatF}
F^n_{h\tau} = f^n_{h\tau} - a^n_{h\tau} 
+ (1/2\tau)(u_{h\tau}^n - u_{h\tau}^{n-1}) \times \gamma \,
 \xi^n_{\tau}
\qquad \text{ and } \qquad
g^{n-1}_{h\tau} = u_{h\tau}^{n-1} \times \gamma,
\end{equation}
so that the equation \eqnref{:harmonicHeatU} becomes
$$
(u_{h\tau}^n - u_{h\tau}^{n-1}, v_h)_H 
= \tau (F_{h\tau}^n, v_h)_H
+ (g_{h\tau}^{n-1}, v_h)_H  \xi_{\tau}^n,
\qquad
v_h \in U_h.
$$
The following lemma bounds the last term of $F_{h\tau}$, which is
the discrete analog of the Stratonovich correction.

\begin{lemma} \label{lem:harmonicHeatBound2}
Under the hypothesis of Lemma \ref{lem:harmonicHeatBound1} with
parameter $p \geq 4$
$$
\norm{F^{(2)}_{h\tau}}_{L^{8/3}(\Omega, L^{4/3}[0,T;H])}
\leq \bbE\left[ \left(\sum_{n=1}^N \normH{u_{h\tau}^n 
- u_{h\tau}^{n-1}}^2 \right)^4 \right]^{1/8},
$$
where $F^{(2)}_{h\tau}$ denotes the piecewise constant function in time
taking values $(1/2\tau)(u_{h\tau}^n -
u_{h\tau}^{n-1}) \times \gamma \,  \xi^n_{\tau}$ on $(t^{n-1}, t^n)$.
\end{lemma}
\begin{proof}
First compute
\begin{eqnarray*}
\norm{F^{(2)}_{h\tau}}^{4/3}_{L^{4/3}[0,T;H]}
&\leq& (|\gamma|/2\tau)^{4/3} \sum_{n=1}^N \tau
\normH{u_{h\tau}^n-u_{h\tau}^{n-1}}^{4/3}
| \xi_{\tau}^n|^{4/3} \\
&\leq& C  \tau^{-1/3}
\left(\sum_{n=1}^N \normH{u_{h\tau}^n - u_{h\tau}^{n-1}}^2 \right)^{2/3}
\left(\sum_{n=1}^N | \xi_{\tau}^n|^4 \right)^{1/3}.
\end{eqnarray*}
The stochastic increments satisfy $\bbE\left[|\xi^n_\tau|^4 \right] \leq C
\tau^2$ when $p \geq 4$, and will cancel the factor of
$\tau^{-1/3}$;
\begin{eqnarray*}
\lefteqn{
\norm{F^{(2)}_{h\tau}}_{L^{8/3}(\Omega, L^{4/3}[0,T;H])}^{8/3}
= \bbE\left[ \norm{F^{(2)}_{h\tau}}^{8/3}_{L^{4/3}[0,T;H]} \right] } \\
&\leq& C  \tau^{-2/3} \bbE\left[
\left(\sum_{n=1}^N \normH{u_{h\tau}^n - u_{h\tau}^{n-1}}^2 \right)^{4/3}
\left(\sum_{n=1}^N | \xi_{\tau}^n|^4 \right)^{2/3} \right] \\
&\leq& C  \tau^{-2/3} 
\bbE\left[ \left(\sum_{n=1}^N \normH{u_{h\tau}^n 
- u_{h\tau}^{n-1}}^2 \right)^4 \right]^{1/3}
\bbE\left[\sum_{n=1}^N | \xi_{\tau}^n|^4 \right]^{2/3}  \\
&\leq& C T^{2/3} 
\bbE\left[ \left(\sum_{n=1}^N \normH{u_{h\tau}^n 
- u_{h\tau}^{n-1}}^2 \right)^4 \right]^{1/3},
\end{eqnarray*}
which completes the proof.
\end{proof}

\begin{theorem}
  Fix $T > 0$ and let $D \subset \Re^3$ be a bounded Lipschitz domain,
  $U = \Hone^3$, $H = \Ltwo^3$, and let $(\Omega, \calF, \bbP)$ be a
  probability space.  Let the Assumptions \ref{ass:spde1}, and
  \ref{ass:spde0} hold with $p=8$ moments. Let $\tau = T/N$
  with $N \in \bbN$ denote a time step, and let $\{U_h\}_{h>0} \subset
  U$ be finite dimensional subspaces satisfying:
  \begin{itemize}
  \item For each $U \in U$ there exists a sequence
    $\{(u_h)\}_{h>0}$ with $u_h \in U_h$
    such that $u_h \rightarrow u$ as $h \rightarrow 0$.

  \item The restriction of the orthogonal projection $P_h:H \rightarrow
    U_h$ to $U$ is stable. That is, there exists $C > 0$ independent
    of $h$ such that $\normU{P_h u} \leq C \normU{u}$.
  \end{itemize}
  Let $\{(u_{h\tau},a_{h\tau})\}_{h,\tau > 0}$ denote the
  solution of \eqnref{:harmonicHeatU}--\eqnref{:harmonicHeatA} with
  data satisfying
  \begin{enumerate}
  \item $\{u^0_{h\tau}\}_{h,\tau > 0}$ is bounded in $L^8(\Omega,U)$ and
    converges to a limit $u^0$ in $L^2(\Omega, U)$
    as $(h,\tau) \rightarrow 0$.

  \item $\{f_{h\tau}\}_{h,\tau > 0}$ is bounded in $L^8(\Omega, \LtwoH)$ and
    converges to a limit $f$ in $L^{8/3}(\Omega, L^{4/3}[0,T;H])$
    as $(h,\tau) \rightarrow 0$.
  \end{enumerate}
  Denote the discrete Wiener process with increments
  $\{\xi_\tau^n\}_{n=1}^N$ by $\What_{h\tau}$, and let
  $$
  F^n_{h\tau} = f^n_{h\tau} + a^n_{h\tau} 
  + (1/2\tau)(u_{h\tau}^n - u_{h\tau}^{n-1}) \times \gamma \,
  \xi^n_{\tau}
  \qquad \text{ and } \qquad
  g^{n-1}_{h\tau} = u_{h\tau}^{n-1} \times \gamma.
  $$
  Then there exist a probability space $(\Omegat, \calFt, \bbPt)$ and a
  random variable $(u,F,g,W)$ on $\Omegat$ with values in $(\bbX,
  \calB(\bbX))$ with
  $$
  \bbX = G[0,T;U'] \cap L^4[0,T;U]_{weak} \cap L^4[0,T; L^4(D)^3]
  \times L^{4/3}[0,T; U']_{weak} 
  \times \LtwoH 
  \times C[0,T],
  $$
  and a subsequence $(\tau_k,h_k) \rightarrow (0,0)$ for which the
  laws of $\bigl\{ (u_{h_k \tau_k}, F_{h_k \tau_k}, g_{h_k \tau_k},
  \hat{W}_{\tau_k})\bigr\}_{k=1}^\infty$ converge to the law of
  $(u,F,g,W)$,
  $$
  \calL(u_{h_k\tau_k}, F_{h_k\tau_k}, g_{h_k\tau_k}, \What_{\tau_k})
  \ \Rightarrow \ \calL(u,F,g,W).
  $$
  In addition, there exists a filtration
  $\{\calFt(t)\}_{0 \leq t \leq T}$ satisfying the usual conditions
  for which $(u,f,g,W)$ is adapted and $W$ is a real-valued Wiener
  process for which
  \begin{equation}\label{eqn:hm1}
  (u(t),v) = (u^0,v) 
  + \int_0^t (F,v)\, ds + \int_0^t (u \times \gamma, v)_H \, dW,
  \qquad v \in \Hone,
  \end{equation}
  where
  \begin{equation}\label{eqn:hm2}
  (F,v) = f - (\nabla u, \nabla v) - (D\phi(u), v) 
  - (1/2) \left(u \times \gamma) \times \gamma, v \vph\right).
  \end{equation}
\end{theorem}

\begin{proof}
We verify that $\bigl\{ (u_{h\tau}, F_{h\tau}, g_{h\tau},
\hat{W}_\tau)\bigr\}_{h,\tau > 0}$ satisfy the hypothesis of Theorem
\ref{thm:main} with parameters $r=q=4$, $q'=4/3$, $p=8/3$, and
$L^s[0,T;V] = L^4[0,T; L^4(D)^3]$ in Statement \ref{it:LsV} of the
theorem.

Note first that $\lone{\phi(u)} \leq C \lfour{u}^4 \leq C \normU{u}^4$
since $\Hone \embed \Lfour$. Then under the hypotheses assumed upon the
data
$$
\norm{\phi(u^0_{h\tau})}_{L^{4/3}(\Omega,\Lone)}
\leq C \norm{u^0_{h\tau}}_{L^{16/3}(\Omega,U)}^4 
\leq C \norm{u^0_{h\tau}}_{L^8(\Omega,U)}^4 < \infty.
$$

\begin{enumerate}
\item Lemma \ref{lem:harmonicHeatBound1} bounds
  $\{u_{h\tau}\}_{h,\tau > 0}$ in $L^8(\Omega, L^\infty[0,T;U]) \embed
  L^4(\Omega, L^4[0,T;U])$.

\item Lemma \ref{lem:harmonicHeatBound1} bounds $\{a_{h\tau}\}_{h,\tau
    > 0}$ in $L^8(\Omega,\LtwoH) \embed L^{8/3}(\Omega,
  L^{4/3}[0,T;U'])$.  Combining this with the bound in
  Lemma \ref{lem:harmonicHeatBound2} shows $\{F_{h\tau}\}_{h,\tau > 0}$ is
  bounded in $L^{8/3}(\Omega, L^{4/3}[0,T;U'])$.

\item Since $L^4[0,T;U]' = L^{4/3}[0,T;U']$ it is immediate that
  $(F_{h\tau},u_{h\tau})$ is bounded in $L^{4/3}(\Omega, L^1(0,T))$.

\item The embedding $U = \Hone^3 \cembed \Lfour^3$ is compact, and
  $\{u_{h\tau}\}_{h,\tau > 0}$ is bounded in $L^8(\Omega,
  L^\infty[0,T;U]) \embed L^4[0,T; L^4(D)^3]$, so from Statement
  \ref{it:LsV} of Theorem \ref{thm:main} it follows that
  upon passing to a subsequence $\calL(u_{h\tau}) \Rightarrow \calL(u)$
  on $L^4[0,T; L^4(D)^3]$.

\item Bounds upon $\{u_{h\tau}\}_{h,\tau > 0}$ immediately bound
  $g_{h\tau} = u_{h\tau} \times \gamma$ in $L^{8/3}(\Omega,
  L^{8/3}[0,T;H])$. In addition, it is immediate that
  $\calL(g_{h\tau}) \Rightarrow \calL(g)$ in $\LtwoH$ when
  $\calL(u_{h\tau}) \Rightarrow \calL(u)$ on $L^4[0,T; L^4(D)^3]
  \embed \LtwoH$.
\end{enumerate}
It follows that upon passing to a sub--sequence $(h_k, \tau_k)
\rightarrow (0,0)$ there exist a filtered probability space,
$(\Omegat, \calFt, \{\calF(t)\}_{0 \leq t \leq T}, \bbPt)$, and a random
variable $(u,F,g,W)$ taking values in $\bbX$ for which $\calL(u_{h_k
  \tau_k},F_{h_k \tau_k},g_{h_k \tau_k},\What_{h_k \tau_k})
\Rightarrow \calL(u,F,g,W)$, $W$ is a standard Wiener
process, and equation \eqnref{:hm1} is satisfied.

To show that $F$ takes the form shown in \eqnref{:hm2}, write
$F_{h\tau} = f_{h\tau} + F^{(1)}_{h\tau} + F^{(2)}_{h\tau}$ with
$$
(F^{(1)}_{h\tau}, v)
= \sum_{n=1}^{T/\tau} \tau -(a^n_{h\tau}, v^n_\tau)
\qquad \text{ and } \qquad
(F^{(2)}_{h\tau}, v)
= (1/2) \sum_{n=1}^{T/\tau} 
(u^n_{h\tau} - u^{n-1}_{h\tau}) \times \gamma, v^n_\tau) \, \xi^n_\tau,
$$
where $v^n_\tau$ is the average of $v \in L^4[0,T;U]$ on $((n-1)\tau,
n \tau)$. Since each summand is bounded in $L^{8/3}(\Omega,
L^{4/3}[0,T;U'])$ we may assume $(f_{h\tau}, F_{h\tau}^{(1)},
F_{h\tau}^{(2)}) \Rightarrow (f, F^{(1)}, F^{(2)})$ on $L^{(4/3)}[0,T;
U']^3_{weak}$ with $F = f + F^{(1)} + F^{(2)}$.

Let $A^{(1)}:L^4[0,T;U]_{weak} \cap L^4[0,T;L^4[0,T;\Lfour]
\rightarrow L^{4/3}[0,T;U']$ be characterized by
$$
(A^{(1)}(u),v) 
= \int_0^T  (\nabla u, \nabla v) + (D\phi(u),v)\, ds
= \int_0^T (\nabla u, \nabla v) + (2/\epsilon) ((u^2-1)u,v)\, ds.
$$
The map is continuous, and if $v \in L^4[0,T; U]$
then $\{(A^{(1)}(u_{h\tau}), v)\}_{h,\tau > 0}$ is bounded in
$L^{4/3}(\Omega)$ and the extended Portmanteau Lemma \ref{lem:mapping_theorem}
shows
$$
\bbEt\left[|(A^{(1)}(u),v)| \right] 
= \lim_{k \rightarrow \infty} \bbE\left[|(A^{(1)}(u_k), v)| \vph\right]
= \lim_{k \rightarrow \infty} \bbE\left[
  \sum_{n=1}^{T/\tau_k} \tau_k \left|\left(A^{(1)}(u^n_k), 
      v^n_{\tau_k} \right) \right| \right]\,,
$$
where we write $u_k \equiv u_{h_k \tau_k}$, and $v_{\tau_k}$ is the
piecewise constant interpolant of $\{v^n_{\tau_k}\}_{n=1}^{T/\tau_k}$. We then
compute
\begin{eqnarray*}
\lefteqn{ \bbEt[|(A^{(1)}(u) - F^{(1)},v)|] 
= \lim_{k \rightarrow \infty} \bbE\left[ 
 \tau_k \left| \sum_{n=1}^{T/\tau_k} 
   \left(a^n_k - D\phi(u^n_k), v^n_{\tau_k} \right)
  - \left(\nabla u^n_k, \nabla v^n_{\tau_k} \right) 
\right| \right] } \\
&=& \lim_{k \rightarrow \infty} \bbE\left[ 
 (\tau_k/\epsilon) \left| \sum_{n=1}^{T/\tau_k} \left(
  (|u^n_k|^2 + |u^{n-1}_k|^2 -2) u^{n-1/2}_k 
  - 2(|u^n_k|^2-1) u^n_k, v^n_{\tau_k} \right)
\right| \right] \\
&=& \lim_{k \rightarrow \infty} \bbE\left[ 
 (\tau_k/\epsilon) \left| \sum_{n=1}^{T/\tau_k} (1/2) \left(
  (|u^n_k|^2 + |u^{n-1}_k|^2 -2) (u^{n-1}_k - u^n_k), v^n_{\tau_k} \right)
  + \left((|u^{n-1}_k|^2 - |u^n_k|^2) u^n_k, v^n_{\tau_k} \right)
\right| \right].
\end{eqnarray*}
Bounding the right-hand side using H\"older's inequality,
and the embedding $U \embed \Lsix$ give
\begin{eqnarray*}
\lefteqn{ \bbEt[|(A^{(1)}(u) - F^{(1)},v) |]} \\
&\leq& (C/\epsilon) \lim_{k \rightarrow \infty} \bbE\left[ 
 \sum_{n=1}^{T/\tau_k} \tau_k \ltwo{u^n-u^{n-1}}
 (\lsix{u^n}^2 + \lsix{u^{n-1}}^2) \lsix{v^n_{\tau_k}} \right] \\
&\leq& (C/\epsilon) \lim_{k \rightarrow \infty} 
\bbE\left[ \sum_{n=1}^{T/\tau_k} \tau_k \ltwo{u^n-u^{n-1}}^2 \right]^{1/2}
 \norm{u_k}_{L^4(\Omega, L^4[0,T;U])}^2
 \norm{v}_{L^4[0,T;U]} \\
&=& \lim_{k \rightarrow \infty} O(\sqrt{\tau_k}) = 0
\end{eqnarray*}
where the last line follows from the estimate in Lemma
\ref{lem:harmonicHeatBound1} on the norm of the differences.

To identify the Stratonovich term, define
$A^{(2)}:L^4[0,T;L^4[0,T;\Lfour] \rightarrow L^{4/3}[0,T;U']$ by
$$
(A^{(2)}(u),v) = (1/2) \int_0^T 
\left( (u \times \gamma) \times \gamma, v \vph\right)
\qquad v \in L^4[0,T; U].
$$
Again this operator is continuous, and
the extended Portmanteau Lemma \ref{lem:mapping_theorem}
shows
\begin{eqnarray*}
\bbEt\left[ \big|(F^{(2)} - A^{(2)}(u), v) \big| \vph\right]
&=& \lim_{k \rightarrow \infty}
\bbE\left[ \big|(F^{(2)}_k - A^{(2)}(u_k), v) \big| \vph\right] \\
&=& \lim_{k \rightarrow \infty}
\bbE\left[ (1/2) \Big| \sum_{n=1}^{T/\tau_k}
\left( (u^{n-1}_k - u^n_k) \times \gamma \xi^n_{\tau_k}
  - (u^n_k \times \gamma) \times \gamma \tau_k, v^n_\tau \right)
\Big| \vph\right].
\end{eqnarray*}
Using the discrete scheme \eqnref{:harmonicHeatU} to rewrite the first
term gives
\begin{eqnarray*}
\lefteqn{
\bbEt\left[ \big|(F^{(2)} - A^{(2)}(u), v) \big| \vph\right] 
= \lim_{k \rightarrow \infty}
\bbE\left[ (1/2) \Big| \sum_{n=1}^{T/\tau_k} \left(
\big(f^n_k - a^n_k \big) \tau_k  \xi^n_{\tau_k}
  - \big(u^{n-1/2}_k (\xi^n_k)^2 - u^n_k \tau_k \big) \times \gamma),
  \gamma \times v^n_\tau \right) 
\Big| \vph\right] } \\
&=& \lim_{k \rightarrow \infty}
\bbE\left[ (1/2) \Big| \sum_{n=1}^{T/\tau_k} \left(
\big(f^n_k - a^n_k \big) \tau_k  \xi^n_{\tau_k}
  - (1/2)\big(u^{n-1}_k - u^n_k \big) \times \gamma \, (\xi^n_k)^2 
  + u^n_k \big((\xi^n_k)^2 - \tau_k \big) \times \gamma,
  \gamma \times v^n_\tau \right) 
\Big| \vph\right] .
\end{eqnarray*}
Each of the three summands on the right vanishes in the limit.
The first term is bounded using H\"older's inequality and the bounds assumed
upon the moments of the stochastic increments,
\begin{eqnarray*}
  \bbE\left[ \Big| \sum_{n=1}^{T/\tau_k} \left(
      \big(f^n_k - a^n_k \big) \tau_k  \xi^n_{\tau_k}
      \gamma \times v^n_\tau \right) 
    \Big| \vph\right]  
  &\leq& |\gamma| \, \norm{f_k-a_k}_{L^2(\Omega, \LtwoLtwo)}
  \bbE\left[ \sum_{n=1}^{T/\tau_k} \tau_k (\xi^n_k)^4 \right]^{1/4}
  \norm{v}_{L^4[0,T;\Lfour]} \\
  &\leq& |\gamma| \, \norm{f_k-a_k}_{L^2(\Omega, \LtwoLtwo)}
  (T \tau_k^2)^{1/4}
  \norm{v}_{L^4[0,T;\Lfour]} .
\end{eqnarray*}
To show that the second term vanishes we use the bound on the
differences $u^{n-1} - u^n$ from Lemma \ref{lem:harmonicHeatBound1},
\begin{eqnarray*}
\bbE\left[ \Big| \sum_{n=1}^{T/\tau_k} \left(
    \big(u^{n-1}_k - u^n_k \big) \times \gamma \, (\xi^n_k)^2 ,
  \gamma \times v^n_\tau \right) 
\Big| \vph\right] 
&\leq& |\gamma|^2
\bbE\left[ \sum_{n=1}^{T/\tau_k} \ltwo{u^{n-1}_k - u^n_k}^2 \right]^{1/2}
\bbE\left[ \sum_{n=1}^{T/\tau_k} (\xi^n_k)^4 \right]^{1/4}
\norm{v}_{L^4[0,T;\Lfour]} \\
&\leq& |\gamma|^2
\bbE\left[ \sum_{n=1}^{T/\tau_k} \ltwo{u^{n-1}_k - u^n_k}^2 \right]^{1/2}
(T \tau_k)^{1/4}
\norm{v}_{L^4[0,T;\Lfour]}.
\end{eqnarray*}
The final term is bounded as
\begin{eqnarray*}
  \bbE\left[ \Big| \sum_{n=1}^{T/\tau_k} \left(
      u^n_k \big((\xi^n_k)^2 - \tau_k \big) \times \gamma,
      \gamma \times v^n_\tau \right) 
    \Big| \vph\right]  
  &\leq& |\gamma|^2 \, \norm{u_k}_{L^4(\Omega, L^4[0,T;\Lfour])}
  \bbE\left[ \sum_{n=1}^{T/\tau_k} 
    \big((\xi^n_k)^2 - \tau_k \big)^2 \right]^{1/2}
  \norm{v}_{L^4[0,T;\Lfour]} \\
  &\leq& |\gamma|^2 \, \norm{u_k}_{L^4(\Omega, L^4[0,T;\Lfour])}
  C (T \tau_k)^{1/2}
  \norm{v}_{L^4[0,T;\Lfour]} ,
\end{eqnarray*}
where the final line follows from the properties the stochastic increments,
$$
\bbE\left[\big((\xi^n_k)^2 - \tau_k \big)^2 \right]
= \bbE\left[(\xi^n_k)^4 - 2 \tau_k (\xi^n_k)^2 + \tau_k^2 \vph\right]
= \bbE\left[(\xi^n_k)^4 \vph\right] - \tau_k^2 
\leq C \tau_k^2.
$$
\end{proof}

\subsection{Monotone Operators}
The canonical example of a maximally monotone operator is the
$q$ Laplacian, $A:U \rightarrow U'$, characterized by
$$
(A(u), v) = \int_D |\nabla u|^{q-2} \nabla u. \nabla v\, dx,
\qquad u, v \in U,
$$
defined on the Sobolev space
$$
U = \Woneqo = \{u \in \Lq \sst \nabla u \in \Lq^d 
\text{ and } u|_{\partial D} = 0 \},
$$
with $D \subset \Re^d$ a bounded domain with Lipschitz boundary.  In
this section we consider the stochastic version of evolution equations
taking the form
\begin{equation} \label{eqn:monotone}
du + A(u) \, dt = f \, dt + g \, dW,
\qquad u(0) = u^0,
\end{equation}
with $A:U \rightarrow U'$ satisfying the following assumptions.

\begin{assumption} \label{ass:monotone} $U$ is a separable reflexive
  Banach space and $H$ is a Hilbert space with $U \cembed H \cembed
  U'$, and there exist constants $C, c > 0$ and $q \in (1,\infty)$ such
  that
  \begin{enumerate}
  \item Monotone: $(A(v)-A(u), v-u) \geq 0$
    for all $u,v \in U$.
  \item Demicontinuous: $A:U_{strong} \rightarrow U'_{weak}$ is continuous.
  \item Bounded: $\normUp{A(u)} \leq C (1 + \normU{u}^{q-1})$ for all $u \in U$.
  \item Coercive: $(A(u), u) \geq c \normU{u}^q$ for all $u \in U$.
  \end{enumerate}
\end{assumption}

\begin{theorem} \label{thm:monotone} Let $U$ be a separable reflexive
  Banach space, $H$ a Hilbert space, $U \cembed H$ be a compact, dense
  embedding, and let $(\Omega, \calF, \bbP)$ be a probability space.
  Let $1 < q < \infty$ and the operators of the abstract difference
  scheme \eqnref{:Spdeh} and data satisfy Assumptions
  \ref{ass:monotone} and \ref{ass:spde1} respectively and let the
  stochastic increments satisfy Assumption \ref{ass:spde0} with $p >
  4$.  Denote the discrete Wiener process with increments
  $\{\xi^m_{\tau}\}_{m=1}^N$ by $\What^n_{\tau}$, and let
  $\{u_{h\tau}\}_{h,\tau >0}$ be a sequence of solutions of the
  corresponding implicit Euler scheme \eqnref{:Spdeh} with data
  satisfying:
  \begin{enumerate}
  \item $\{u^0_{h\tau}\}$ is bounded in $L^p(\Omega,H)$ and
    converges in $L^2(\Omega, H)$
    as $h \rightarrow 0$.

  \item $\{f_{h\tau}\}$ is bounded in $L^{pq'/2}(\Omega, \LqpUp)$ and
    converges as $\tau,h \rightarrow 0$.

  \item $\{g_{h\tau}\}$ is bounded in $L^p(\Omega, \LpH)$ and
    converges in $L^2(\Omega, \LtwoH)$ as $\tau,h \rightarrow 0$.
  \end{enumerate}
  Let
  $$
  \bbX \equiv G[0,T;U'] \cap L^q[0,T;U]_{weak}
  \times \LqpUp 
  \times \LqpUp_{weak}
  \times \LtwoH
  \times C[0,T]\, .
  $$
  Then there exist a probability space $(\Omegat, \calFt, \bbPt)$ and
  a random variable $(u,f,a,g,W)$ on $\Omegat$ with values in $(\bbX,
  \calB(\bbX))$ for
  which the laws of $\bigl\{ u_{h \tau}, f_{h \tau}, A(u_{h\tau}),
  g_{h \tau}, \hat{W}_{\tau})\bigr\}_{k=1}^\infty$ converge to
  the law of $(u,f,a,g,W)$,
  $$
  \calL(\uhat_{h\tau}, f_{h\tau}, A(u_{h\tau}), 
  g_{h\tau}, \What_{\tau})
  \ \Rightarrow \ \calL(u,f,a,g,W).
  $$
  In addition, $\bbPt[u \in C[0,T;U'] \cap \LinfH] = \bbPt[a = A(u)] =
  1$, and there exists a filtration $\{\calFt(t)\}_{0 \leq t \leq T}$
  satisfying the usual conditions for which $(u,f,g,W)$ is adapted and
  $W$ is a real-valued Wiener process for which
  $$
  (u(t), v)_H   + \int_0^t (A(u),v) \, ds
  = (u^0,v)_H 
  + \int_0^t ( f, v ) \, ds
  + \int_0^t {( g, v )} \, dW,
  \qquad v \in U.
  $$
\end{theorem}

In the previous examples the proof of consistency used the property
that the principle part of the operator $A:U \rightarrow U'$ was
linear.  For monotone operators this is no longer the case and the
following lemmas provide the properties required to establish the
assertion $\bbPt[a=A(u)]$ in the proof of Theorem \ref{thm:monotone}.
The first result is used to establish consistency in the deterministic
setting \cite[Lemmas III.2.1 and III.4.2]{Sh97}.

\begin{lemma} \label{lem:monotone} 
  Let $A:U \rightarrow U'$ be monotone, demicontinuous,
  and bounded (i.e. bounded sets map to bounded sets).  
  \begin{itemize}
  \item If $u_n \weak u$ in $U$ and $A(u_n) \weak a$ in $U'$ and
    $\limsup_{n \rightarrow \infty} (A(u_n), u_n) \leq (a, u)$, then
    $a = A(u)$.

  \item If $A$ satisfies Assumptions \ref{ass:monotone}, then
    so too does its realization $\calA:L^q[0,T;U] \rightarrow 
    L^{q'}[0,T;U']$ given by
    $$
    (\calA(u),v) = \int_0^T (A(u(t),v(t)) \, dt.
    $$
  \end{itemize}
\end{lemma}

The following lemma is the analog of this lemma for random variables.
In the proof of Theorem \ref{thm:monotone} this lemma will be used
with Banach space $\calU = \LqU$.

\begin{lemma}[Identification] \label{lem:identification} 
  Let $\calU$ be a separable reflexive Banach space and $\calA:\calU
  \rightarrow \calU'$ be monotone, demicontinuous and bounded. Let
  $(\Omega, \bbP, \calF)$ be a probability triple and
  $\{u_n\}_{n=1}^\infty$ be random variables with values in $\calU$
  satisfying:
  \begin{itemize}
  \item $\calL(u_n, \calA(u_n)) \Rightarrow \calL(u,a)$ in 
    $\calU_{weak}\times \calU'_{weak}$.

  \item $\sup_n \bbE \left[\norm{u_n}_\calU^s +
    \norm{\calA(u_n)}^{s'}_{\calU'} \right] < \infty$ for some $s > 1$.

  \item $\liminf_{n\to\infty}\bbE [(\calA(u_n),u_n)] \le \bbE[(a,u)].$
  \end{itemize}
  Then $\calL(u,a)[a=\calA(u)]=1$.
\end{lemma}

We postpone the proof of this lemma until the end of this section. 

\begin{proof}(of Theorem \ref{thm:monotone})
Writing $a(u,v) = (A(u), v)$, we consider the numerical approximation
of solutions to equation \eqnref{:monotone} using the scheme
\eqnref{:Spdeh} with data (\eqnref{:AA1}) from Section
\ref{sec:introSpde}.  Selecting the test function $v_h = u_{h\tau}^n$
in the discrete scheme \eqnref{:Spdeh}, the coercivity hypothesis
gives the bound
$$
(1/2) \normH{u^n}^2 + (1/2) \normH{u^n - u^{n-1}}^2 + c \tau \normU{u^n}^q
\leq 
(1/2) \normH{u^{n-1}}^2 
+ \tau (f_{h\tau}^n, u^n_{h\tau})
+ (g^{m-1}_{h\tau}, u^m_{h\tau})_H \xi^n_\tau.
$$
It follows from Lemma \ref{lem:bound2} that
\begin{multline*}
\norm{\max_{0 \leq t \leq T} \uhat_{h\tau}}_{L^p(\Omega,H)}
+ \norm{u_{h\tau}}^{q/2}_{L^{pq/2}(\Omega,\LqU)} \\
\leq C(T) \left( \norm{u^0_{h\tau}}_{L^p(\Omega,H)}
+ \norm{f_{h\tau}}^{q'/2}_{L^{p q'/2}(\Omega,\LqpUp)}
+ \norm{g_{h\tau}}_{L^p(\Omega,\LtwoH)}
\vph\right).
\end{multline*}
Granted bounds upon the data $(u^0,f,g)$, this estimate establishes
the hypotheses of Theorem \ref{thm:main} with $F_{h\tau} = f_{h\tau} -
A(u_{h\tau}) \equiv F^{(1)}_{h\tau} + F^{(2)}_{h\tau}$ (and moment
parameter $\min(p q/2, p q'/2) > 2$), so that, upon passing to a
subsequence, there exist a filtered probability space $(\Omegat,
\calFt, \{\calFt(t)\}_{0 \leq t \leq T}, \bbPt)$ and a random variable
$(u,f,a,g,W)$ with values in $\bbX$ for which
$\calL(u_{h\tau},f_{h\tau}, A(u_{h\tau}), g_{h\tau}, \What_{h\tau})
\Rightarrow \calL(u,f,a,g,W)$ and
$$
(u(t),v) = (u^0,v) + \int_0^t (f(s)-a(s),v) \, ds 
+ \int_0^t (g(s),v)_H \, dW(s),
\qquad 0 \leq t \leq T, \,\, v \in U.
$$
Since $A:U \rightarrow U'$ satisfies Assumptions \ref{ass:A}, uniqueness
in law holds for solutions of \eqnref{:monotone}, so that upon showing
$a = A(u)$ it will follow that it whole sequence converges as asserted
in the statement of the theorem.

Lemma \ref{lem:identification} with $s = pq/2$ is used to verify 
that $a = A(u)$. Since $A$ has $(q-1)$ growth it follows that
$$
\lqpUp{A(u)}^{pq'/2} 
\leq C \left(1 + \lqU{u}^{pq/2} \right).
$$
Then $s > 1$ and $s' < pq'/2$ when $q > 1$ and $p > 2$, so the growth
hypothesis of Lemma \ref{lem:identification} is satisfied.  The third
hypothesis is established by showing that the continuous and discrete
pairings satisfy
\begin{eqnarray}
\bbEt \left[ \int_0^T (a,u) \, ds\right]
&=& \bbEt \left[ (1/2) \left(\normH{u(0)}^2 - \normH{u(T)}^2 \right) 
  + \int_0^T \left((f,u) + (1/2) \normH{g}^2 \vph\right) \, ds\right],
\label{eqn:monotone0} \\
\bbE\left[\int_0^T (A(u_{h\tau}), u_{h\tau}) \, ds\right]
&\leq& \bbE\left[ 
  (1/2) \left(\normH{u^0_{h\tau}}^2 - \normH{u^N_{h\tau}}^2 \right)
  + \int_0^T \left( (f_{h\tau},u_{h\tau}) 
    + (1/2) \normH{g_{h\tau}}^2 \vph\right) \, ds\right]
\label{eqn:monotoneh}
\end{eqnarray}
and to then show that the limit on the right-hand side of the second
equation is bounded by the right-hand side of the first.

To verify equation \eqnref{:monotone0}, recall that Ito's formula,
Theorem \ref{thm:Ito}, shows
$$
\bbEt\left[(1/2) \normH{u(T)}^2 \right]
= \bbEt\left[(1/2) \normH{u(0)}^2 
  + \int_0^T \left( (f-a,u) + (1/2) \normH{g}^2 \vph\right) \, ds\right]\, ,
$$
which is precisely equation \eqnref{:monotone0}.

To verify equation \eqnref{:monotoneh}, select the test function
$v_h = u^n_{h\tau}$ in the discrete scheme \eqnref{:Spdeh} to get
$$
(1/2) \normH{u^n_{h\tau}}^2 
+ (1/2) \normH{u^n_{h\tau} - u^{n-1}_{h\tau}}^2 \phi_n
+ (A(u^n_{h\tau}), u^n_{h\tau})
= (1/2) \normH{u^{n-1}_{h\tau}}^2
+ f(u^n_{h\tau}) \phi^n + (g^{n-1}_{h\tau}, u^n_{h\tau}) \xi^n.
$$
Summing this identity and independence of
the increments, $\bbE[(g^{n-1}_{h\tau}, u^{n-1}_{h\tau})_H \xi^n] = 0$, shows
\begin{multline*}
\bbE\left[\sum_{n=1}^N 
  (1/2) \normH{u^N_{h\tau}}^2 
+ (1/2) \normH{u^n_{h\tau} - u^{n-1}_{h\tau}}^2
+ \int_0^T (A(u_{h\tau}), u_{h\tau}) \, ds\right] \\
= \bbE\left[ 
  (1/2) \normH{u^0_{h\tau}}^2
+ \sum_{n=1}^N (f^n_{h\tau},u^n_{h\tau})
+ \sum_{n=1}^N (g^{n-1}_{h\tau}, u^n_{h\tau}-u^{n-1})_H \xi^n \right].
\end{multline*}
Equation \eqnref{:monotoneh} follows upon bounding
the last term as
$$
\bbE\left[\sum_{n=1}^N (g^{n-1}_{h\tau}, u^n_{h\tau}-u^{n-1})_H \xi^n \right]
\leq 
\half \bbE\left[\sum_{n=1}^N \normH{g^{n-1}_{h\tau}}^2 (\xi^n)^2 \right]
+ \half \bbE\left[\sum_{n=1}^N \normH{u^n_{h\tau}-u^{n-1}}^2\right],
$$
and recalling that the variance of the increments is the time step,
$\bbE\left[ \normH{g^{n-1}_{h\tau}}^2 (\xi^n)^2\right] =
\bbE\left[\normH{g^{n-1}_{h\tau}}^2 \tau\right]$.

To pass to the limit on the right of \eqnref{:monotoneh}, recall that
Example \ref{eg:ic} shows that, under the hypotheses of the theorem,
$$
\bbEt\left[\normH{u(0)}^2 \right]
= \lim_{h,\tau \rightarrow 0} \bbE\left[\normH{u^0_{h\tau}}^2 \right]
\quad \text{ and } \quad
\bbEt\left[\normH{u(T)}^2 \right]
\leq \lim_{h,\tau \rightarrow 0} \bbE\left[\normH{u^N_\tau}^2 \right],
$$
where $N = T/\tau$. The function
$$
(u, f,a, g, W) \mapsto \int_0^T (f,u) + (1/2) \normH{g}^2\, ds
$$
is continuous on $\bbX$ and the numerical approximation of each
term has moments with modulus strictly greater than one, so
$$
\lim_{h,\tau \rightarrow 0}
\int_0^T (f_{h\tau},u_{h\tau}) + (1/2) \normH{g_{h\tau}}^2\, ds
= \int_0^T (f,u) + (1/2) \normH{g}^2\, ds\, .
$$
\end{proof}

We finish this section with the proof of Lemma \ref{lem:identification}.

\begin{proof}(of Lemma \ref{lem:identification})
Since $\calU$ is separable and reflexive it follows that $\calU^\prime$
is also separable, and if $u$ is a Borel measurable random variable
with values in $\calU$ then $\calA(u)$ is a Borel measurable random
variable in $\calU^\prime$ since $\calA$ is demi--continuous.  The
separability of $\calU$ and $\calU^\prime$ also implies that
$$
\calB(\calU_{weak}\times \calU^\prime_{weak})
= \calB(\calU\times \calU^\prime)
= \calB(\calU)\otimes\calB(\calU^\prime).
$$
Define $\bbX=\calU_{weak}\times \calU^\prime_{weak}$, denote by $\bbPt$ the
law of $(u,a)$ on $\calB(\bbX)$, and let $B_1,\dots,B_m$ be Borel sets
in $\bbX$ such that
$$
\bbPt\left[\partial B_1\cup\dots\cup\partial B_k \vph\right]=0.
$$
Fix $v_1,\dots,v_k\in \calU$ and define 
$$
f(z)=\sum_{j=1}^k\mathbf 1_{B_j}(z)v_j.
$$
Then $f:\bbX\to \calU_{strong}$ and $\calA(f):\bbX\to
\calU^\prime_{strong}$ are uniformly bounded on $\bbX$ and continuous
with respect to sequences $z_n\to z$ where $z$ belongs to
$\bbX\setminus(\partial B_1\cup\dots\cup\partial B_k)$; a set of
$\bbPt$-measure one. In particular, by the extended Portmanteau Lemma 
\ref{lem:mapping_theorem},
\begin{align}
\lim_{n\to\infty}\bbE\left[\left( \calA(u_n),f(u_n,\calA(u_n))\vph\right)\right]
&=\widetilde{\bbE} \left[\left( a,f(u,a)\vph\right)\right]
\label{remterm1} \\
\lim_{n\to\infty}\bbE\left[\left( \calA(f(u_n,\calA(u_n))),u_n\vph\right)\right]
&=\widetilde{\bbE} \left[\left( \calA(f(u,a)),u\vph\right)\right]
\label{remterm2} \\
\lim_{n\to\infty}\bbE\left[\left( \calA(f(u_n,\calA(u_n))),
f(u_n,\calA(u_n))\vph\right)\right]
&=\widetilde{\bbE}\,\left[\left( \calA(f(u,a)),f(u,a)\vph\right)\right]
\label{remterm3}
\end{align}
despite $\calA$ not being weakly continuous. By monotonicity,
$$
\Bbb E\left[\left( \calA(u_n)-\calA(f(u_n,\calA(u_n))),
u_n-f(u_n,\calA(u_n))\vph\right)\right]\ge 0
$$
so, by the upper semi--continuity assumption on $\{\Bbb
E\left[\left(\calA(u_n),u_n\vph\right)\right]\}$ and
\eqref{remterm1}-\eqref{remterm3},
\begin{equation}\label{gf1}
\widetilde{\Bbb E}\,\left[\left( a-\calA(f(u,a)),u-f(u,a)\vph\right)\right]
\ge 0.
\end{equation}
Now
$
\calB_0=\{B\in\calB(\bbX):\,\bbPt(\partial B)=0\}
$
is an algebra such that $\sigma(\calB_0)=\calB(\bbX)$, thus, if
$B_1,\dots,B_k$ belong to $\calB(\bbX)$, then there exist
$B_1^n,\dots,B^n_k$ in $\calB_0$ with $n\in\bbN$ such that
$$
f_n(z)=\sum_{j=1}^k\mathbf 1_{B^n_j}(z)v_j
\rightarrow
f(z)=\sum_{j=1}^k\mathbf 1_{B_j}(z)v_j,
\qquad
\text{$\bbPt$-almost surely.}
$$
Consequently, \eqref{gf1} holds for every Borel simple function $f$.
Demi-continuity of $\calA$ then implies that \eqref{gf1} holds for
every Borel measurable bounded function $f$, which then extends
\eqref{gf1} to $f\in L^s[(\bbX,\calB(\bbX),\bbPt); \calU]$ by a cut-off
argument.  In particular, if $\xi:\bbX\to U$ is Borel measurable and
bounded, then applying $f=\pi_1+t\xi$ to \eqref{gf1} and letting $t\to
0$, we get
$$
\bbEt \,\left[\left( a-\calA(u),\xi(u,a)\vph\right)\right]=0
$$
by demi-continuity of $\calA$. In particular, $\bbPt \,[a=\calA(u)]=1$.
\end{proof}

\pagebreak

\appendix

\section{Laws and Random Variables}

Classical probability is well developed for random variables taking
values in Polish (complete separable metric) spaces; however, the weak
topologies of Banach space that arise for problems involving partial
differential operators are not metrizable. In this appendix extensions
of the classical results to the current setting are presented.

\subsection{Portmanteau Theorem for Non-Metrizable Spaces}

The following proof is a generalization of the proof of the mapping
theorem in \cite[Theorem 2.7]{Billingsley99} which admits sequences of
functions which may not be continuous but may, for example, be
sequentially continuous or lower semi--continuous.

\begin{proof} (of Lemma \ref{lem:mapping_theorem})
To prove the first assertion, define $\nu_k=\bbP_k(\zeta_k\in\cdot)$
and $\nu=\bbP(\zeta_k\in\cdot)$. For $\epsilon > 0$, let $C_\varepsilon$ be
a compact subset\footnote{Note that compacts subsets of $\calX$ are
  metrizable.} of $\calX$ such that $\bbP_k(C_\varepsilon)\ge
1-\varepsilon$ and let $V$ be a closed set in $\Bbb R$. Then
$$
\limsup_{k\to\infty}\nu_k(V)
\le \varepsilon+\limsup_{k\to\infty}\bbP_k([\zeta_k\in V]\cap C_\varepsilon)
\le \varepsilon+\bbP\left(\overline{\bigcup_{k\ge n}[\zeta_k\in V]
    \cap C_\varepsilon}\right)
$$
holds by the Portmanteau theorem for every $n\ge 1$, hence
$$
\limsup_{k\to\infty}\nu_k(V)\le\varepsilon+\bbP\left(\bigcap_{n\ge
1}\overline{\bigcup_{k\ge n}[\zeta_k\in V]\cap
C_\varepsilon}\right)\le\varepsilon+\bbP(\zeta\in V)+\bbP^*(N),
$$
thus $\nu_k\Rightarrow\nu$ by the Portmanteau theorem.

For the second assertion, let $C_\varepsilon$ be a compact set as
above. Then
$$
\limsup_{k\to\infty}\bbP_k([\zeta\le t]\cap C_\varepsilon)\le \bbP([\zeta\le
t]\cap C_\varepsilon)
$$
by the Portmanteau theorem. So
$
\liminf_{k\to\infty}\bbP_k[\zeta>t]\ge \bbP[\zeta>t]
$
and
$$
\int_{\calX}\zeta\,d\bbP=\int_0^\infty
\bbP[\zeta>t] \, dt\le\liminf_{k\to\infty}\int_0^\infty
\bbP_k[\zeta>t] \, dt=\liminf_{k\to\infty}\int_{\calX}\zeta\,d\bbP_k
$$
by the Fatou lemma.
\end{proof}

\subsection{Compactness of Piecewise Linear and Constant Interpolants}

\begin{lemma} Let $Z = C[0,T;U^\prime]$ or $G[0,T;U^\prime]$,
  $R\in(0,\infty)$ and define
  $$
  M_R=\{u\in Z\cap \LrU_{weak}:\,\|u\|_{\LrU}\le R\}.
  $$
  Then $M_R$ is closed and metrizable. In particular
  \begin{itemize}
  \item If $\calF$ is a compact in $Z$ then $\calF\cap M_R$ is a compact in
    $Z\cap \LrU_{weak}$.

  \item If $\calF$ is a compact in $Z\cap \LrU_{weak}$ then $\calF$ is a
    compact in $Z$ and there exists $R>0$ such that $\calF \subseteq M_R$.
  \end{itemize}
\end{lemma}

\begin{proof}
  Closed balls of separable reflexive Banach spaces (here
  $\LrU$), equipped with the weak topology, are metrizable and
  intersections of metric spaces is also a metric
  space.
\end{proof}

\begin{proof} (of Lemma \ref{lem:UandUhat})
 Let us consider the modulus of continuity
 (see (6.2) in \cite[Section 3.6]{EiKu86}) 
$$
 w(u,\delta)
 =\inf \left\{\sup \big\{\|u(t)-u(s)\|_{U^\prime}:\,s,t\in (s_j,s_{j+1}],
 \,0\le j\le m \big\} \sst \min_j\,(s_{j+1}-s_j)>\delta \right\},
 $$
 and observe that
 $w(u_n,\delta)=0$ if $\delta<T/n$ and
 $$
 w(u_n,\delta)\le 2m(\hat u_n,T/n)+m(\hat u_n,2\delta)\le 3m(\hat
 u_n,2\delta)\quad\text{if}\quad \delta\ge T/n,
 $$
 where $m$ is the standard modulus of continuity in $C[0,T;U^\prime]$.
 In particular,
 $$
 w(u_n,\delta)\le 3m(\hat u_n,2\delta)\quad\delta\in(0,T).
 $$
 Also, $Rg(u_n)\subseteq Rg(\hat u_n)$.
 Hence, tightness of $\calL(\hat u_n)$ in
 $C[0,T;U^\prime]$ implies tightness of $\calL(u_n)$ in
 $G[0,T;U^\prime]$. If $\mu$ is the accumulation probability
 measure then there exists a subsequence $n_k$ such that
\begin{itemize}
\item $\calL(u_{n_k})\Rightarrow\mu$ in
  $G[0,T;U^\prime]\cap \LrU_{weak}$,

\item $\calL(u_{n_k},\hat u_{n_k})\Rightarrow\theta$ in
  $G[0,T;U^\prime]\cap \LrU_{weak} \times C[0,T;U^\prime]$.
\end{itemize}
Then $\mu$ is the first marginal of $\theta$, and
$$
d_G(u_n,\hat u_n)\le \linfUp{u_n-\hat u_n} \le m(\hat u^n,T/n),
$$
so
$$
1=\lim_{k\to\infty} \calL\,(u_{n_k},\hat u_{n_k})\,\{(x,y):d(x,y)
\le\varepsilon\}\le\theta\,\{(x,y):d(x,y)
\le\varepsilon\}, \qquad \varepsilon>0,
$$
by the Portmanteau theorem. Hence $\theta(V)=1$ where
$V=\{(x,y):\,x=y\}$ and
\begin{eqnarray*}
  \mu(C[0,T;U^\prime]\cap \LrU_{weak})
  &=& \theta(C[0,T;U^\prime]\cap \LrU_{weak} \times C[0,T;U^\prime]) \\
  &=& \theta(C[0,T;U^\prime]\cap \LrU_{weak}
  \times C[0,T;U^\prime]\cap V) \\
  &=& \theta(G[0,T;U^\prime]\cap \LrU_{weak}
  \times C[0,T;U^\prime]\cap V)=1.
\end{eqnarray*}
\end{proof}

\subsection{Proof of Theorems \ref{thm:uniqueLaw} and \ref{thm:SE}}
We adopt the context of Section \ref{sec:unique}; specifically, $U$ is
a separable Banach space, $H$ is a Hilbert space, and $U \embed H
\embed U'$ are dense embeddings, and write $\bbX_1 = C[0,T;U'] \cap
L^r[0,T;U]_{weak}$.

The proof Theorems will \ref{thm:uniqueLaw} and \ref{thm:SE} follow
from the following two results for random variables taking values in
topological spaces.  We start with a lemma on existence of a regular
version of a random probability measure.

\begin{lemma}\label{lem:RPM} Let $X$ be a topological space such that
  there exist real continuous functions $h_n:X\to\bbR$ and, for every
  $x_0,x_1\in X$ distinct, there exists $n\in\bbN$ satisfying
  $h_n(x_0)\ne h_n(x_1)$. Let $(H,\calH,\mu)$ be a probability space,
  \begin{enumerate}
  \item $r_B:H\to [0,1]$ be $\calH$-measurable for every $B\in\calB(X)$, 
  \item $\mu(r_\emptyset=0)=1$, $\mu(r_X=1)=1$,
  \item $\mu(r_{B_0}+r_{B_1}+r_{B_2}+\dots=r_B)=1$ whenever
    $B_0,B_1,B_2,\dots$ are pair-wise disjoint Borel sets in $X$ and
    $B$ denotes their union,
  \item $\mu(r_S=1)=1$ for some $\sigma$-compact set $S$ in $X$.
  \end{enumerate}
  Then there exists 
  \begin{enumerate}
  \item $R_B:H\to [0,1]$ which is $\calH$-measurable for every
    $B\in\calB(X)$,
  \item $B\mapsto R_B(h)$ is a Borel probability measure supported in
    $S$, for every $h\in H$,
  \item $\mu(r_B=R_B)=1$ for every $B\in\calB(X)$.
  \end{enumerate}
\end{lemma}

\begin{proof}
  Existence of regular versions of random probability measures is well
  know for Polish spaces. Use the functions $\{h_n\}$ to construct
  an injective mapping $F:X\to Z$ for a suitable Polish space $Z$. If
  $C$ is a compact set in $X$ then $F|C:C\to F[C]$ is a homeomorphism.
  Hence $F|S:S\to F[S]$ and $(F|S)_{-1}:F[S]\to S$ are Borel
  measurable. Denote by $K$ a regular version of the random
  probability measure $r_{F^{-1}[A]}(h)$ for $A\in\calB(Z)$ and
  $h\in H$, i.e.
  \begin{enumerate}
  \item $K_A:H\to [0,1]$ is $\calH$-measurable for every $A\in\calB(Z)$,
  \item $A\mapsto K_A(h)$ is a Borel probability measure for every $h\in H$,
  \item $\mu(K_A=r_{F^{-1}[A]})=1$ for every $A\in\calB(Z)$,
  \end{enumerate}
  and define $U_B(h)=K_{F[B\cap S]}(h)$ for $B\in\calB(X)$ and $h\in H$. Then
  \begin{enumerate}
  \item $U_B:H\to [0,1]$ is $\calH$-measurable for every $B\in\calB(X)$,
  \item $B\mapsto U_B(h)$ is a Borel measure for every $h\in H$,
  \item $\mu(U_B=r_B)=1$ for every $B\in\calB(X)$.
  \end{enumerate}
  Now we define $K_B(h)=U_B(h)$ for $h\in[U_S=1]$ and
  $K_B(h)=\delta_s(B)$ for $h\notin[U_S=1]$.
\end{proof}

\begin{proposition} \label{prop:YW} Let Assumption \ref{ass:A} hold
  and $\theta$ be a Borel probability measure on $C[0,T;U^\prime]$.
  Then there exists a Borel measurable mapping
  $$
  k_\theta:C[0,T;U^\prime]\to\bbX_1
  $$
  with a range in a $\sigma$-compact set, and with the following
  property: If $(u,V)$ is a solution of \eqref{eqdet} on a probability
  space $(\Omega,\calF,\bbP)$ and $\calL(V)=\theta$ then
  $$
  \bbP\,\left[u=k_\theta(V)\right]=1.
  $$
\end{proposition}

\begin{proof}
  The proof follows the argument of the Yamada-Watanabe
  theorem. Let $\bbY=C[0,T;U^\prime]$ and assume that $(u^i,V^i)$
  is a solution of \eqref{eqdet} on a probability space
  $(\Omega^i,\calF^i,\bbP^i)$ with $\calL(V^i)=\theta$.
  Then
  $$
  \calB(\bbX_1)\otimes\calB(\bbX_1)\otimes\calB(\bbY)
  =\calB(\bbX_1\times\bbX_1\times\bbY),
  $$
  $$
  \calB(\bbX_1)\otimes\calB(\bbY)=\calB(\bbX_1\times\bbY), \quad
  \calB(\bbX_1)\otimes\calB(\bbX_1)=\calB(\bbX_1\times\bbX_1)
  $$
  because
  $$
  \{u\in\bbX_1:\|u\|_{L^r[0,T;U]}\le n\}
  $$
  is separable and metrizable for every $n\in\bbN$. In particular,
  $\calL(u^i,V^i)$, $i=0,1$ are Borel probability measures on $\Bbb
  X_1\times\bbY$. If $Q$ is a Borel set in $\bbX_1$ then
  $\calL(u^i,V^i)(Q\times\cdot)$ is absolutely continuous with respect
  to $\theta$. So, by Lemma \ref{lem:RPM}, there exists
  $R^i:\bbY\times\calB(\bbX_1)\to [0,1]$ such that
  \begin{enumerate}
  \item $R^i(\cdot,Q):\bbY\to [0,1]$ is Borel measurable for every
    $Q\in\calB(\bbX_1)$,
  \item $Q\mapsto R^i(y,Q)$ is a Borel probability measure supported
    in $S^i$ for every $y\in\bbY$
  \end{enumerate}
  and
  $$
  \calL(u^i,V^i)(Q\times J)=\int_JR^i(y,Q)\,d\theta(y), \qquad 
  Q\in\calB(\bbX_1),\quad J\in\calB(\bbY),\quad i=0,1.
  $$
  Define a Borel probability measure
  $$
  \bbP^*(L)=\int_{\bbY}(R^0_y\otimes R^1_y)(L^y)\,d\theta(y), \qquad 
  L\in\calB(\bbX_1\times\bbX_1\times \bbY).
  $$
  and random variables $U^1(a,b,c)=a$, $U^2(a,b,c)=b$ and $V(a,b,c)=c$
  on $\bbX_1\times\bbX_1\times \bbY$. Then
  $$
  \calL(U^0,V)=\calL(u^0,V^0),\qquad\calL(U^1,V)=\calL(u^1,V^1)
  $$
  so
  $$
  \bbP^*\left[U^i(t)=V(t) - \int_0^t A(U^i(s))\,ds\right]=1, \qquad 
  t\in[0,T],\quad i=0,1
  $$
  and, by the uniqueness of the deterministic equation, we obtain that
  $$
  \bbP^*\left[U^0=U^1\right]=1.
  $$
  Hence, if we denote by $D$ the diagonal in $\bbX_1\times\bbX_1$, we get
  $$
  1=\bbP^*(D\times\bbY)=\int_{\bbY}(R^0_y\otimes R^1_y)(D)\,d\theta(y).
  $$
  In particular $(R^0_y\otimes R^1_y)(D)=1$ for every $y\in
  M\in\calB(\bbY)$ where $\theta(M)=1$. So there exists a unique
  $k(y)\in\Bbb X_1$ such that $R^0_y=R^1_y=\delta_{k(y)}$ for every
  $y\in M$. Set $k(y)=x$ for $y\notin M$ where $x\in\bbX_1$ is
  arbitrary. Now $k:\bbY\to\bbX_1$ is Borel measurable with the range
  in a $\sigma$-compact set in $\bbX_1$ since
  $$
  \{y\in\bbY:k(y)\in B\}\cap M=\{y\in\bbY:R^0(y,B)=1\}\cap M,
  $$
  and
  $$
  \calL(u^i,V^i)(N)=\theta(\{y\in\bbY:(k(y),y)\in N\}),\qquad
  N\in\calB(\bbX_1\times\bbY).
  $$
  In particular,
  $$
  \bbP^i\,[u^i=k(V^i)]=1,\qquad i=0,1.
  $$
\end{proof}

\begin{proof} (of Theorem \ref{thm:uniqueLaw})
  Proposition \ref{prop:YW} yields that
  $$
  \calL(u^0,V^0)
  =\calL(k_\theta(V^0),V^0)
  =\calL(k_\theta(V^1),V^1)=\calL(u^1,V^1)
  $$
  where $\theta:=\calL(V^0)=\calL(V^1)$.
\end{proof}

\begin{proof} (of Theorem \ref{thm:SE}) We apply Proposition
  \ref{prop:YW} with $\theta=\calL(V)$ and $u=k_\theta(V)$. To prove
  that $u$ is $(\calF_t^{V,0})$-adapted let $\tau\in(0,T]$ and define
  $\lambda=\tau/T\in(0,1]$, $\tilde u_\lambda(t)=\tilde u(\lambda t)$,
  $\tilde V_\lambda(t)=\tilde V(\lambda t)$ and
  $V_\lambda(t)=V(\lambda t)$ for $t\in[0,T]$, and
  $\theta_\tau:=\calL(\tilde V_\lambda)=\calL(V_\lambda)$.  Then
  $(\tilde u_\lambda,\tilde V_\lambda)$ solve
  $$
  du=dV - \lambda A(u)\,dt
  $$
  since
  $
  \{w(\lambda\cdot):w\in S\}
  $
  is $\sigma$-compact in $\bbX_1$ when $w\mapsto w(\lambda\cdot)$ is
  continuous from $\bbX_1$ to $\bbX_1$. If we define
  $u_\lambda:=k_{\theta_\tau}(V_\lambda)$ then
  $$
  du_\lambda=dV_\lambda - \lambda A(u_\lambda)\,dt\quad\text{a.s.}
  $$
  But we also have that
  $$
  du(\lambda\cdot)=dV_\lambda - \lambda A(u(\lambda\cdot))\,dt\quad\text{a.s.}
  $$
  so Assumption \ref{ass:A} yields that $u_\lambda(T)=u(\lambda
  T)=u(\tau)$ a.s. Now $u_\lambda=k_{\theta_\tau}(V_\lambda)$ is
  $\calF_T^{V_\lambda,0}$-measurable and
  $
  \calF_T^{V_\lambda,0}=\calF_\tau^{V,0}.
  $
  So $u(\tau)$ is $\calF_\tau^{V,0}$-measurable.
\end{proof}

\bibliography{heatEqn}

\end{document}